\documentclass[journal,12pt,onecolumn,draftclsnofoot,]{IEEEtran}
\usepackage{amssymb}
\usepackage{amsmath}	
\usepackage{amsthm}
\usepackage{tikz}
\usetikzlibrary{arrows}
\usepackage{verbatim}
\usepackage{rotating}
\usepackage{amsbsy}
\usepackage{algorithm}	
\usepackage{algpseudocode}
\usepackage{graphicx}	
\usepackage{epstopdf}
\usepackage{epsfig}
\usepackage{lipsum}
\usepackage{mathdots}
\usepackage{mathtools}
\usepackage{fixltx2e}
\usepackage{subcaption}
\usepackage{multicol}
\usepackage{setspace}
\usepackage{xcolor}
\usepackage{tabularx}
\newtheorem{theorem}{Theorem}

\newtheorem{lemma}{Lemma}
\theoremstyle{definition}
\newtheorem{definition}{Definition}

\usepackage{pgfplots}
\newcommand{\myvec}[1]{{\bf\text{#1}}}

\usepackage{lipsum}
\usepackage{chngcntr}
\usepackage{apptools}
\AtAppendix{\counterwithin{proposition}{section}}

\makeatletter
\def\footnoterule{\relax%
  \kern-1pt
  \hbox to \columnwidth{\vrule width 0.5\columnwidth height 0.4pt\hfill}
  \kern4.6pt}
\makeatother

\usepackage{caption}
\usepackage{url}
\usepackage{breqn}
\usepackage{stmaryrd}
\usepackage[mathscr]{euscript}

\DeclareMathOperator{\vecc}{vec}
\DeclareMathOperator{\Cov}{cov}

\begin{document}
\title{An Introduction to Complex Random Tensors}
\author{Divyanshu~Pandey, 
		Alexis~Decurninge,
        Harry~Leib
\thanks{Divyanshu Pandey (email: dp76@rice.edu) is with the Department of Electrical and Computer Engineering, Rice University, Houston, TX, USA. Alexis Decurninge (email: alexis.decurninge@huawei.com) is with Huawei Technologies, Boulogne-Billancourt, France. Harry Leib (email : harry.leib@mcgill.ca) is with the Department of Electrical and Computer Engineering, McGill University, Montreal, QC, Canada.}
} 
\maketitle

\vspace*{-1.5cm}

\begin{spacing}{1}
\begin{abstract}
This work considers the notion of random tensors and reviews some fundamental concepts in statistics when applied to a tensor based data or signal. In several engineering fields such as Communications, Signal Processing, Machine learning, and Control systems, the concepts of linear algebra combined with random variables have been indispensable tools. With the evolution of these subjects to multi-domain communication systems, multi-way signal processing, high dimensional data analysis, and multi-linear systems theory, there is a need to bring in multi-linear algebra equipped with the notion of random tensors. Also, since several such application areas deal with complex-valued entities, it is imperative to study this subject from a complex random tensor perspective, which is the focus of this paper. Using tools from multi-linear algebra, we characterize statistical properties of complex random tensors, both proper and improper, study various correlation structures, and fundamentals of tensor valued random processes. Furthermore, the asymptotic distribution of various tensor eigenvalue and singular value definitions is also considered, which is used for the study of spiked real tensor models that deals with recovery of low rank tensor signals perturbed by noise. This paper aims to provide an overview of the state of the art in random tensor theory of both complex and real valued tensors, for the purpose of enabling its application in engineering and applied science.
\end{abstract} 

\section{Introduction}

Tensors are multi-way arrays which are now widely used in several engineering applications to represent data and model systems spanning more than one domain. The application of tensors was initially introduced  in Physics and Continuum Mechanics \cite{PierreComon}, and was later adopted for applications in Psychometric \cite{tucker64extension} and Chemometrics \cite{bro2006review} as well. In the past two decades, tensors have been extensively employed in signal processing applications \cite{NikosCDMA, Cichocki, MMSEJournal}. More recently, with the emergence of data driven systems, where data is inherently multi-modal in nature, the use of tensors has almost become ubiquitous \cite{cichocki2014era}. Due to its wide spread applications, tensor data analysis has garnered significant attention in the statistical learning community. 

Realizing the wide applications of tensor tools, over the past few years there have been several tutorial-style papers detailing the concepts of tensor algebra addressing directly the engineering community and their specific applications \cite{KoldaTensor,BaderTensor,NikosTensor,TensorBook2020,TensorTut2021,pandey2023linear}. However, most of these tutorials focus only on the study of the algebraic and structural properties of deterministic tensors. The notion of random tensors and their statistical properties are only discussed sparsely across these papers, for specific application and use-cases. To the best of our knowledge, there is no single document which exhaustively explores the properties of random tensors in a suitable manner. This paper provides such an overview of several random tensor concepts and properties, with the target audience being graduate students and researchers in various engineering disciplines. It is important to mention that we will not be looking at random tensor theory as studied in theoretical Physics \cite{guruau2017random} which considers tensor models to define and analyze geometries in higher dimension. In this work, we look at tensors as a higher order generalization of vectors and matrices, where the physical interpretation of its various modes can be application specific.  

Random tensors appear in several engineering areas, such as modelling signals and channels in multi-domain communication systems \cite{PandeyIMMSE}, computer vision \cite{panagakis2021tensor}, multiway image denoising \cite{muti2008lower}, big data analytics \cite{song2019tensor}, control systems \cite{MLTI2}, bio-medical signal processing \cite{zhou2016linked}, and many more. Thus extending the results from random vector and matrix theory to higher order tensors is essential for a complete statistical understanding of multi-modal data or signals in all these fields. A commonly used approach while handling multi-modal data/signals is to ignore their structure or treat them as a vector in order to apply known methods from linear algebra and random vectors theory. However, such characterization fails to capture the mutual effect of various modes of the tensor, and makes it difficult to identify low rank structures in the tensor, or find mode specific correlations. Hence, it is important to study the notions of random tensors without distorting the tensor structure. The Einstein product of tensors and its isomorphism with the space of their transformed matrices has extended several matrix algebra results to a tensor setting without restructuring the tensor \cite{TamonTensorInversion,LuEVD,pandey2023linear}. The Einstein product is a special form of tensor contraction, and can also be used to extend results from random vector/matrix theory to higher order tensors, as discussed in detail in this paper.
     
Random matrix theory was used in particular to study the asymptotic properties of the matrix spectrum with useful applications in the analysis of wireless communications systems \cite{RandomMatrixForWireless}, neural network \cite{Pennington17} or mechanical systems \cite{soize05}. The adaptation of these tools to random tensors encounter the lack of universally adopted generalization of matrix eigenvalues and singular values to higher order tensors \cite{limsingular}. Indeed, the notion of tensor spectrum is defined in various manners in the literature. Two classes of tensor eigenvalues are of particular interest, the first being strongly related to the tensor unfolding, transforming the tensors into matrices \cite{pandey2023linear} while the second considers the tensor eigenvalues as solutions of a fixed point equation \cite{limsingular}. Using different tools, both classes can be used to study models where the observations are the sum of a low-rank \textit{signal} tensor with an additive \textit{noise} tensor.
These models are often referred as spiked tensor models since the signal tensor can be seen as a spike. The analysis of the asymptotic properties of estimators of the signal has attracted recent interest and is an active area of research.
For example, Ben Arous et al. analyzed asymptotic properties of the matrix unfolding spectrum with such observations in \cite{benarous21}.
Several authors \cite{montanari14, perry20, jagannath20, Chen21} took another path leveraging spin glass theory in order to derive the quality of signal estimators assuming some prior distribution on the low-rank signal tensor.
Another line of research directly uses tools from random matrix theory in \cite{de2022random, seddik23} to study the tensor contractions involved in the fixed point equations defining the second class of eigenvalues and eigenvectors.
 
In this paper, we leverage the tools from multi-linear algebra as described in \cite{pandey2023linear} to extend the subject of random vectors to higher order arrays. The objective of this work is to provide a tutorial style presentation of the notion of random tensors starting from basics, while highlighting the key differences that the higher order structures bring forth.

The paper is organized as follows : Section \ref{Sec2} presents basic tensor algebra preliminaries. Section \ref{Sec3} presents complete characterization of complex random tensors using joint distributions, along with their first and second order moments. In particular, tensor correlation structures are presented in details. Section \ref{Sec4} considers  Gaussian tensors and discusses various specific cases of the same. Section \ref{Sec5} deals with random tensor processes. Section \ref{sec6} considers various distributions of tensor singular and eigenvalues and their asymptotic properties. Section \ref{sec7} presents the asymptotic properties of both symmetric and  asymmetric spiked random tensors. Note that the spiked random tensor properties are so far only known for real tensors, and its extension to complex tensors is a future research direction. The paper is concluded in Section \ref{sec8}.      

\section{Preliminaries and Notational Convention} \label{Sec2}
\subsection{Notations}
In this paper, deterministic vectors are represented using lowercase fonts, e.g. $\text{x}$ , deterministic matrices using uppercase fonts, e.g. $\text{X}$ and deterministic tensors of order 3 or more using uppercase calligraphic fonts, e.g. $\mathscr{X}$. Their corresponding random quantities will be denoted by bold fonts, e.g. $\myvec{\textbf{x}}$ for random vectors, $\textbf{X}$ for random matrices and $\pmb{\mathscr{X}}$ for random tensors. The individual elements of a tensor are indicated using the indices in subscript, e.g. the $(i,j,k,l)^{th}$ element of a fourth order tensor $\mathscr{X}$ is denoted by ${\mathscr{X}}_{i,j,k,l} $. A colon in subscript represents all the possible values of an index, e.g. $\mathscr{X}_{:,j,k,l}$ denotes a vector containing all entries of $\mathscr{X}$ in the first mode corresponding to $j^{th}$ second, $k^{th}$ third, and $l^{th}$ fourth mode. Similarly, $\mathscr{X}_{:,:,k,l}$ denotes a matrix slice in tensor $\mathscr{X}$ corresponding to $k^{th}$ third and $l^{th}$ fourth mode, $\mathscr{X}_{:,:,:,l}$ denotes a third order sub-tensor corresponding to $l^{th}$ fourth mode, and $\mathscr{X}_{:,:,:,:}$ is same as $\mathscr{X}$.  The $n^{th}$ element in a sequence is denoted by a superscript in parentheses, e.g. ${\mathscr{A}}^{(n)}$ denotes the $n^{th}$ tensor in a sequence of tensors. The set of complex numbers is denoted by $\mathbb{C}$ and the set of reals by $\mathbb{R}$. An all zero tensor is denoted by $0_{\mathscr{T}}$.

\subsection{Basics of tensor algebra}
Tensors, which are multi-way arrays, are used to represent systems and processes with variations across multiple modes. The number of modes is called the tensor \textit{order}. Tensors can be seen as a generalization of vectors and matrices to higher order. Thus, vectors are often referred as order 1 tensors and matrices as order 2 tensors. A tensor containing complex elements drawn from $\mathbb{C}$ is called a complex tensor, and a tensor containing real elements drawn from $\mathbb{R}$ is called a real tensor. Thus, an order-$N$ complex tensor is denoted as $\mathscr{X} \in \mathbb{C}^{I_1 \times \dots \times I_N}$ and an order-$N$ real tensor is denoted as $\mathscr{X} \in \mathbb{R}^{I_1 \times \dots \times I_N}$, where $I_n$ denotes the dimension of $n^{th}$ mode of the tensor. The well-defined notions of linear algebra which applies to vector and matrix based operations may differ significantly when applied to tensors because of multiple modes being involved. In this section we define a few basic operations associated with complex tensors which are required to understand the concepts in this paper. A more detailed review of these tensor algebra results can be found in \cite{pandey2023linear} and references within.

\subsubsection*{Tensor Transformation}
An order $N+M$ tensor $\mathscr{X} \in \mathbb{C}^{I_1 \times \dots I_N \times J_1 \times \dots \times J_M}$ can be transformed into a matrix $\text{X} \in \mathbb{C}^{I_1 \cdots I_N \times J_1 \cdots J_M}$ using the transformation function $f_{I_1,\dots,I_N|J_1,\dots,J_M}(\mathscr{X})=\text{X}$ defined as \cite{TamonTensorInversion}:
\begin{equation}\label{tranformprop}
\mathscr{X}_{i_1,i_2,\dots,i_N,j_1,j_2,\dots, j_M} \xrightarrow{f_{I_1,\dots,I_N|J_1,\dots,J_M}}\text{X}_{i_1+\sum_{k=2}^{N}(i_k-1)\prod_{l=1}^{k-1}I_l , j_1+\sum_{k=2}^{M}(j_k-1)\prod_{l=1}^{k-1}J_l }.
\end{equation}
Such a transformation is referred to as a matrix mapping of a tensor by partitioning its $N+M$ modes into two disjoint sets. This mapping is bijective \cite{TamonTensorInversion} and thus $f_{I_1,\dots,I_N|J_1,\dots,J_M}^{-1}(\text{X})=\mathscr{X}$. By exploiting the isomorphism between the matrix and tensor linear spaces, several concepts from linear algebra have been derived for tensors \cite{TamonTensorInversion, pandey2023linear}. Another commonly used matrix transformation of a tensor is mode-n unfolding where a tensor $\mathscr{A} \in \mathbb{C}^{I_1 \times \dots \times I_N}$ can be converted into a sequence of matrices $\{\text{A}^{(n)}\}_{n=1}^N$ such that for each $n$, the matrix $\text{A}^{(n)} \in \mathbb{C}^{I_n \times I_1 \cdots I_{n-1}I_{n+1}\cdots I_N}$ is formed by the the mode-n fibers of $\mathscr{A}$ taken as the columns of the resulting matrix \cite{KoldaTensor}.  

\subsubsection*{Tensor Products}
Since a tensor has multiple modes, two tensors can be multiplied across various modes. The most general form is called the tensor contraction product, and several other specific products can be defined using contraction over one or multiple modes . Some relevant products are briefly described below, more details can be found in \cite{pandey2023linear}.
\begin{itemize}
\item \textit{Contracted Product} : It is denoted by $\mathscr{Z} = \{ \mathscr{X},\mathscr{Y} \}_{\{1,\dots,M;1,\dots,M\}}$, where $\mathscr{X} \in \mathbb{C}^{I_{1} \times \dots \times I_{M} \times J_{1} \times \dots \times J_{N} }$, $\mathscr{Y} \in \mathbb{C}^{I_{1} \times \dots \times I_{M} \times K_{1} \times \dots \times K_{P} }$, and $\mathscr{Z} \in \mathbb{C}^{J_{1} \times \dots \times J_{N} \times K_{1} \times \dots \times K_{P} } $ where each element of the resulting tensor is defined as $
\mathscr{Z}_{j_1,\dots,j_N,k_1,\dots,k_P} = \sum_{i_1,\dots,i_M} \mathscr{X}_{i_{1}, \dots, i_{M}, j_{1} , \dots, j_{N}}\mathscr{Y}_{i_{1} , \dots, i_{M} , k_{1} , \dots ,k_{P}}$. In general, the contraction need not be over consecutive modes.
\item \textit{Einstein Product} : For tensors $\mathscr{A} \in \mathbb{C}^{I_1 \times \dots \times I_P \times K_1 \dots \times K_N}$ and $\mathscr{B} \in \mathbb{C}^{K_1 \times \dots \times K_N \times J_{1} \dots \times J_M}$, Einstein product is a contraction between their $N$ common modes, denoted by $*_N$, and is defined as $
(\mathscr{A}*_N \mathscr{B})_{i_1,\dots,i_P,j_{1},\dots,j_{M}} = \sum_{k_1,\dots,k_N}\mathscr{A}_{i_1,i_2,\dots,i_P,k_1,\dots,k_N}\mathscr{B}_{k_1,\dots k_N,j_{1},j_{2},\dots,j_M}$.
\item \textit{Mode-n product} : For a tensor $\mathscr{A} \in \mathbb{C}^{I_{1} \times I_{2} \times \dots \times I_{N} }$ and a matrix $\text{U} \in \mathbb{C}^{J \times I_{n}}$, mode-n product is denoted by $\mathscr{B} = \mathscr{A} {\times}_{n} \text{U}$ for $n=1,\dots,N$, and defined as $
\mathscr{B}_{i_1,i_2,\dots,i_{n-1},j,i_{n+1},\dots,i_N} = \sum_{i_n = 1}^{I_n} \mathscr{A}_{i_1,i_2,\dots,i_N} \text{U}_{j,i_n}$ where $\mathscr{B} \in \mathbb{C}^{I_1 \times I_2 \times \dots \times I_{n-1} \times J \times I_{n+1} \times \dots I_N}$.
\item \textit{Outer product} : For $\mathscr{X} \in \mathbb{C}^{I_{1} \times \dots \times I_{N} }$ and $\mathscr{Y} \in \mathbb{C}^{J_{1} \times \dots \times J_{M} }$, the outer product is denoted by $\mathscr{X} \circ \mathscr{Y} $ and it is defined as $
( \mathscr{X} \circ \mathscr{Y} )_{i_{1},i_{2},\dots,i_{N},j_{1},j_{2},\dots,j_{M}} = \mathscr{X}_{i_{1},i_{2},\dots,i_{N}} \mathscr{Y}_{j_{1},j_{2},\dots,j_{M}}$.  
\end{itemize}

Furthermore, based on the definition of the Einstein product of tensors and the transformation in \eqref{tranformprop}, it was shown in \cite{TamonTensorInversion, pandey2023linear} that for tensors $\mathscr{X} \in \mathbb{C}^{I_1 \times \dots \times I_N \times J_1 \times \dots \times J_M}$ and $\mathscr{Y} \in \mathbb{C}^{J_1 \times \dots \times J_M \times K_1 \times \dots \times K_P}$ we have, $f_{I_1,\dots,I_N|K_1,\dots,K_P}(\mathscr{X}*_M \mathscr{Y})=f_{I_1,\dots,I_N|J_1,\dots,J_M}(\mathscr{X})\cdot f_{J_1,\dots,J_M|K_1,\dots,K_P}(\mathscr{Y})$. This result helps define several concepts such as tensor inversion, Hermitian , pseudo-diagonality, identity tensor, unitary tensor, EVD, SVD, determinant, trace, and their associated properties.

\subsubsection*{Tensor Structures} Various structures in a tensor can be observed depending on its specific order and dimensions. A tensor is called \textit{cubical} if every mode has the same size, for e.g. $\mathscr{X} \in  \mathbb{C}^{I \times I \times \dots \times I}$ \cite{KoldaTensor}. Further, a tensor $\mathscr{X} \in \mathbb{C}^{I_1 \times \dots \times I_N \times J_1 \times \dots \times J_M}$ is called a \textit{square} tensor if $N=M$ and $I_k=J_k$ for $k=1,\dots,N$ \cite{TensorDet}.  Thus any even order cubical tensor is also a square tensor. A tensor $\mathscr{X} \in \mathbb{C}^{I_1 \times \dots \times I_N \times J_1 \times \dots \times J_M}$ is called \textit{pseudo-diagonal} if its matrix transformation $f_{I_1,\dots,I_N|J_1,\dots,J_M}(\mathscr{X})$ is a diagonal matrix \cite{pandey2023linear}. 

A cubical tensor is called symmetric (sometimes also referred as supersymmetric \cite{KoldaTensor}) if its elements remain constant under any permutation of the indices. For a square tensor of order $2N$, the notion of partial symmetry can be defined based on partition after $N$ modes using the definition of tensor transpose. In general, the notions of tensor transpose and Hermitian can be defined with respect to any permutation of their indices, see \cite{pan2014tensor} \cite{pandey2023linear} for details. However for the purpose of this paper, we define it for a fixed permutation with partition after $N$ modes. The \textit{Hermitian} of a tensor $\mathscr{A} \in \mathbb{C}^{I_1 \times \dots \times I_N \times J_1 \times \dots \times J_M }$ is denoted as $\mathscr{A}^H \in \mathbb{C}^{J_1 \times \dots \times J_M \times I_1 \times \dots \times I_N }$ such that $(\mathscr{A}^H)_{j_1,j_2,\dots, j_M,i_1,i_2,\dots,i_N}^*=\mathscr{A}_{i_1,i_2,\dots,i_N,j_1,j_2,\dots,j_M}$. Similarly the \textit{transpose} is defined as $(\mathscr{A}^T)_{j_1,j_2,\dots, j_M,i_1,i_2,\dots,i_N}=\mathscr{A}_{i_1,i_2,\dots,i_N,j_1,j_2,\dots,j_M}$. 

A square tensor $\mathscr{X} \in \mathbb{C}^{I_1 \times \dots \times I_N \times I_1 \times \dots \times I_N}$ (for both real and complex case) can be referred as partially symmetric if $\mathscr{X}=\mathscr{X}^T$. Furthermore, a square complex tensor is called \textit{Hermitian} if $\mathscr{X}=\mathscr{X}^H$. Also, a square tensor is called unitary if $\mathscr{X}^H*_N \mathscr{X} = \mathscr{X}*_N \mathscr{X}^H = \mathscr{I}_N$, where $\mathscr{I}_N$ denotes a pseudo-diagonal identity tensor of order $2N$. Note that the transpose or Hermitian of a tensor can be defined for any order tensor by suitably defining a mode partition. However, the structures of partial symmetry or conjugate partial symmetry (Hermitian) are defined only for square tensors.     

\subsubsection*{Tensor Decompositions}
It is important to understand the two most commonly used tensor decompositions : Canonical Polyadic (CP) and Tucker decompositions, to understand the spiked random  tensor models. The CP decomposition factorizes a tensor into a sum of component rank-one tensors. The CP decomposition of a tensor $\mathscr{X} \in \mathbb{C}^{I_{1} \times I_{2} \times \dots \times I_{N} }$ can be stated as \cite{KoldaTensor}
\begin{equation}\label{CPdecomeq}
\mathscr{X} = \sum_{r=1}^{R}\text{a}_{r}^{(1)}\circ \text{a}_{r}^{(2)}\circ\dots\circ\text{a}_{r}^{(N)} := \llbracket  \text{A}^{(1)}, \text{A}^{(2)},\dots,\text{A}^{(N)} \rrbracket
\end{equation} 
where ${\text{a}_r}^{(i)} \in {\mathbb{C}}^{I_i}$ are the column vectors of matrices $\text{A}^{(i)} \in {\mathbb{C}}^{I_i \times R}$. The notion of tensor rank is tied to CP decomposition. An order-$N$ tensor is \textit{rank-one} if it can be written as the outer product of $N$ vectors \cite{KoldaTensor}. Thus, essentially CP decomposition writes a tensor $\mathscr{X}$ as a sum of $R$ rank-one tensors where $R$ is called the CP-rank of the tensor.

The Tucker decomposition, first introduced in \cite{Tucker1966}, decomposes a tensor into a core tensor multiplied (or transformed) by a matrix along each mode. The Tucker decomposition of the $N^{th}$ order tensor $\mathscr{X} \in \mathbb{C}^{I_1 \times \dots \times I_N}$ can be written as:
\begin{equation}
\mathscr{X} = \mathscr{G} {\times}_{1} \text{A}^{(1)} {\times}_{2} \text{A}^{(2)} {\times}_{3} \dots {\times}_{N} \text{A}^{(N)}
\end{equation}
where $\mathscr{G} \in \mathbb{C}^{J_1 \times \dots \times J_N}$ is the core tensor and $\text{A}^{(n)} \in \mathbb{C}^{I_n \times J_n}$ are the factor matrices. The Tucker decomposition is also called as Higher-Order Singular Value Decomposition (HOSVD) \cite{LathauwerSVD}.

\section{Characterization of Complex Random Tensors}\label{Sec3}
The statistical properties of random tensors can be characterized through joint distributions of all their elements. In particular, while dealing with complex random tensors, a complete characterization may benefit from an augmented representation such that the correlation between the real and imaginary components of a tensor can also be analyzed. In this section, we establish the basic notions of joint distributions and second order moments associated with complex random tensors.

\subsection{Complete characterization of a complex random tensors though joint PDF} 
A tensor $\pmb{\mathscr{X}} \in \mathbb{C}^{I_1 \times \dots \times I_N}$ is said to be random if its components $\pmb{\mathscr{X}}_{i_1,\dots,i_N}$ are random variables defined on the same probability space $(\Omega, \mathcal{F},P)$. The statistics of $N$ complex random variables $\textbf{x}_n$ for $n=1,\dots,N$ are often determined by the joint cumulative distribution function/probability mass function/probability density function (CDF/PMF/PDF) of the $2N$ real random variables $\Re(\textbf{x}_n)$ and $\Im{(\textbf{x}_n)}$. Thus, a complex random vector $\myvec{x} \in \mathbb{C}^{N}$, is often denoted using its composite real representation as $\myvec{x}^{c}=\begin{bmatrix}\Re{(\myvec{x})} \\ \Im{(\myvec{x})}\end{bmatrix} \in \mathbb{R}^{2N}$. Alternately, an augmented representation is sometimes used since it keeps the vector within the complex field, and it is defined as $\myvec{x}^{a}=\begin{bmatrix}\myvec{x} \\ \myvec{x}^*\end{bmatrix} \in \mathbb{C}^{2N}$ \cite{ProperComplex, ComplexPDFBook}. 

But unlike a vector, a tensor has more than one mode. Thus, an augmented or composite representation could be introduced across any mode. Given that the primary advantage of a tensor is its ability to maintain distinction between different domains, rather than concatenating and extending a given domain, we suggest that for a complex valued tensor $\pmb{\mathscr{X}} \in \mathbb{C}^{I_1\times \dots \times I_N}$, the augmented or composite tensors can be created by adding another domain of size 2. With such representation, a composite real tensor would be defined as $\pmb{\mathscr{X}}^{c}, \in \mathbb{R}^{I_1 \times \dots \times I_N \times 2}$  where $\pmb{\mathscr{X}}^c_{i_1,\dots,i_N,1}=\Re{(\pmb{\mathscr{X}}_{i_1,\dots,i_N})}$ and $\pmb{\mathscr{X}}^c_{i_1,\dots,i_N,2}=\Im{(\pmb{\mathscr{X}}_{i_1,\dots,i_N})}$. Similarly, the augmented complex tensor can be defined as $\pmb{\mathscr{X}}^{a}, \in \mathbb{C}^{I_1 \times \dots \times I_N \times 2}$ where $\pmb{\mathscr{X}}^a_{i_1,\dots,i_N,1}=\pmb{\mathscr{X}}_{i_1,\dots,i_N}$ and $\pmb{\mathscr{X}}^a_{i_1,\dots,i_N,2}=\pmb{\mathscr{X}}_{i_1,\dots,i_N}^*$. The composite and augmented structures are shown in Fig. \ref{First_Aug_Cor} where a square is used to represent a tensor with its order indicated in a smaller box on top right corner.  

\begin{figure}[h]
\center
\includegraphics[scale=0.6]{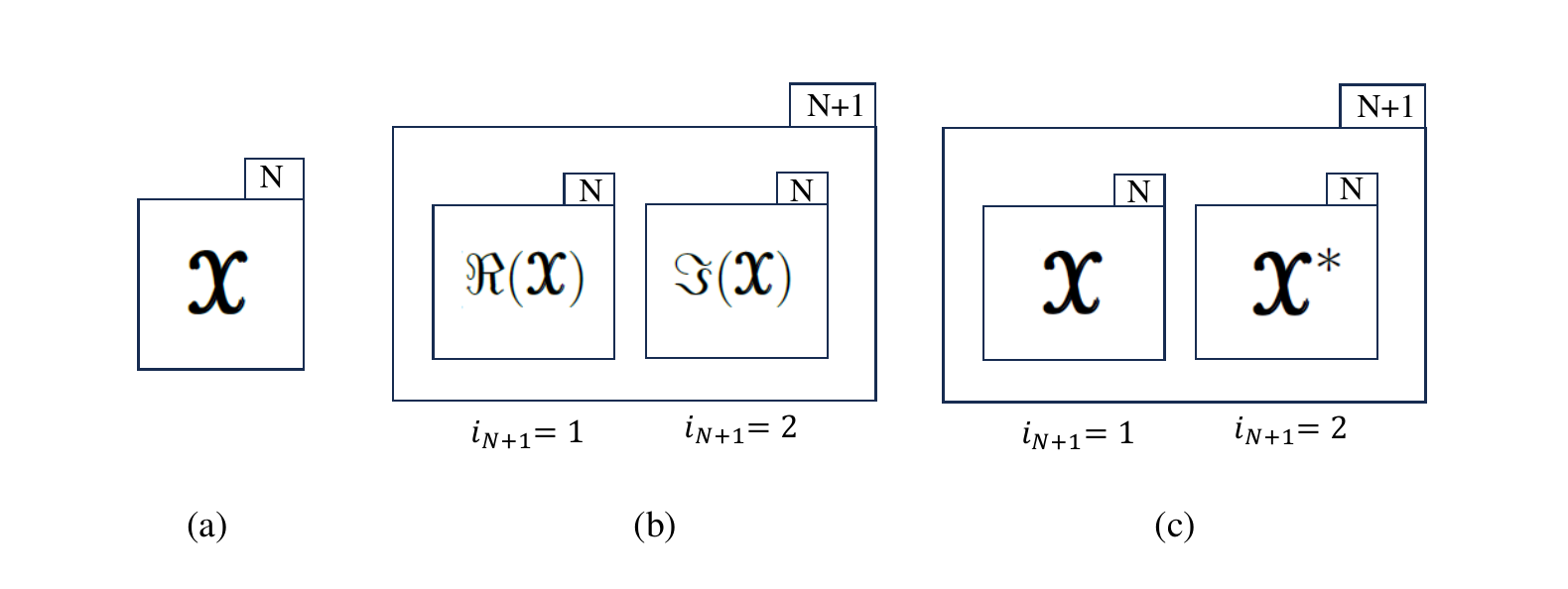}
\caption{(a) Order $N$ complex tensor, (b) Order $N+1$ real composite tensor, (c) Order $N+1$ complex augmented tensor.  \label{First_Aug_Cor}}
\end{figure}

The joint CDF of a real tensor $\pmb{\mathscr{X}} \in \mathbb{R}^{I_1 \times \dots \times I_N}$ is the joint CDF of all the elements of the tensor and can be written as:
\begin{equation}
F_{{}_{\pmb{\mathscr{X}}}}(\mathscr{X}) = Pr\left(\bigcap_{i_1,\dots,i_N} \pmb{\mathscr{X}}_{i_1,\dots,i_N} \in (-\infty,\mathscr{X}_{i_1,\dots,i_N}]\right).
\end{equation}
A random tensor is discrete if all its elements are discrete random variables, in which case its PMF is defined as:
\begin{equation}
p_{{}_{\pmb{\mathscr{X}}}}(\mathscr{X}) = Pr\left(\bigcap_{i_1,\dots,i_N} \pmb{\mathscr{X}}_{i_1,\dots,i_N} =\mathscr{X}_{i_1,\dots,i_N}\right).
\end{equation}

A random tensor is absolutely continuous if its CDF can be expressed as a multi-dimensional integral of a joint PDF of its components
\begin{equation}
F_{{}_{\pmb{\mathscr{X}}}}(\mathscr{X})=\int_{-\infty}^{ \mathscr{X}_{I_1,I_2,\dots,I_N}}\dots \int_{-\infty}^{ \mathscr{X}_{1,1,\dots,2}}\int_{-\infty}^{ \mathscr{X}_{1,1,\dots,1}}
p_{{}_{\pmb{\mathscr{X}}}}(\mathscr{U}) {d \mathscr{U}_{1,1,\dots,1}  d \mathscr{U}_{1,1,\dots,2} \dots d \mathscr{U}_{I_1,I_2,\dots,I_N}}
\end{equation}
Subsequently for any continuity point its joint PDF is given by taking the partial derivative of its CDF with respect to all the individual elements as:

\begin{equation}
p_{{}_{\pmb{\mathscr{X}}}}(\mathscr{X}) = \dfrac{\partial^{I_1\cdots I_N}F_{{}_{\pmb{\mathscr{X}}}}(\mathscr{X})}{\partial \mathscr{X}_{1,1,\dots,1}\partial \mathscr{X}_{1,1,\dots,2}\cdots \partial \mathscr{X}_{I_1,I_2,\dots,I_N}}.
\end{equation} 

In this paper, we are interested in absolutely continuous tensors. For a complex tensor $\pmb{\mathscr{X}} \in \mathbb{C}^{I_1 \times \dots \times I_N}$, its statistics can be described using the joint CDF/PDF of the real and imaginary components of all its elements. Hence either the augmented or the composite representations can be used to specify their joint CDFs and PDFs as:
\begin{equation}\label{CDF_X_xonjx}
F_{{}_{\pmb{\mathscr{X}}}}(\mathscr{X},\mathscr{X}^*) = F_{{}_{\pmb{\mathscr{X}}^c}}(\mathscr{X}^c) =F_{{}_{\Re{(\pmb{\mathscr{X}})},\Im{(\pmb{\mathscr{X}})}}}(\Re{(\mathscr{X})},\Im{(\mathscr{X})}),
\end{equation} 
\begin{equation}\label{PDF_X_xonjx}
p_{{}_{\pmb{\mathscr{X}}}}(\mathscr{X},\mathscr{X}^*) = p_{{}_{\pmb{\mathscr{X}}^c}}(\mathscr{X}^c) =p_{{}_{\Re{(\pmb{\mathscr{X}})},\Im{(\pmb{\mathscr{X}})}}}(\Re{(\mathscr{X})},\Im{(\mathscr{X})}).
\end{equation} 
The CDF and PDF as defined in \eqref{CDF_X_xonjx} and \eqref{PDF_X_xonjx} highlight that these distributions are a function of both the complex tensor and its conjugate. The composite and the augmented representations can both be used to define the statistical properties of the complex tensor. However, since the composite representation allows to remain in the complex field, it is more widely used. 

The first order moment, also known as mean or expectation, of a complex random tensor $\pmb{\mathscr{X}}$ is a tensor of same size as $\pmb{\mathscr{X}}$, and is denoted as $\mathscr{M} = \mathbb{E}[\pmb{\mathscr{X}}] = \mathbb{E}[\Re{(\pmb{\mathscr{X}})}] + j \mathbb{E}[\Im{(\pmb{\mathscr{X}})}] $ with each component consisting of the expected value of the corresponding element of $\pmb{\mathscr{X}}$. The augmented mean tensor will thus be denoted as $\mathscr{M}^a=\mathbb{E}[\pmb{\mathscr{X}}^a]$.

\subsubsection*{Characteristic Function of a random tensor}
The \textit{characteristic function} of a complex random vector $\myvec{x} \in \mathbb{C}^N$ is defined as $\Phi_{\myvec{x}}(\underline{\omega})= \mathbb{E}\big[ \exp(i\Re(\underline{\omega}^H \myvec{x}) ) \big]$ for $\underline{\omega} \in \mathbb{C}^N$\cite{andersenlinear}. Using Einstein Product, the characteristic function of a complex random tensor $\pmb{\mathscr{X}} \in \mathbb{C}^{I_1 \times \dots \times I_N}$ is 
\begin{equation}\label{characfuncdef}
\Phi_{\pmb{\mathscr{X}}}(\mathscr{W})= \mathbb{E}\big[ \exp(i\Re(\mathscr{W}^* *_N \pmb{\mathscr{X}}) ) \big]
\end{equation}
for tensor $\mathscr{W} \in \mathbb{C}^{I_1 \times \dots \times I_N}$.

\subsection{Second order characterization}

The covariance of a complex tensor $\pmb{\mathscr{X}} \in \mathbb{C}^{I_1 \times I_2 \times \dots \times I_N}$ can be defined as a tensor of size $I_1 \times I_2 \times \dots \times I_N \times I_1 \times I_2 \times \dots \times I_N$ represented by $\mathscr{Q}=\mathbb{E}[(\pmb{\mathscr{X}}-\mathscr{M})\circ(\pmb{\mathscr{X}}-\mathscr{M})^{*}]$ where $\mathscr{M}=\mathbb{E}[\pmb{\mathscr{X}}]$ is the mean tensor. However, a complete second-order characterization of complex tensors which accounts for correlation between the real and imaginary components of the tensor as well, requires defining pseudo-covariance, also known as complementary covariance \cite{ComplexPDFBook}. The pseudo-covariance of tensor $\pmb{\mathscr{X}}$ is given as $\tilde{\mathscr{Q}}=\mathbb{E}[(\pmb{\mathscr{X}}-\mathscr{M})\circ (\pmb{\mathscr{X}}-\mathscr{M})]$. The covariance tensor defined using the augmented version of $\pmb{\mathscr{X}}$ is specified by $\ddot{\mathscr{Q}}=\mathbb{E}[(\pmb{\mathscr{X}}^a-\mathscr{M}^a)\circ(\pmb{\mathscr{X}}^a-\mathscr{M}^a)^{*}]$ and is of size $I_1\times \dots \times I_N \times 2 \times I_1\times \dots \times I_N \times 2$. The augmented covariance tensor $\ddot{\mathscr{Q}}$ contains the covariance $\mathscr{Q}$ and the pseudo-covariance tensor $\tilde{\mathscr{Q}}$ along with their conjugates as $\ddot{\mathscr{Q}}_{\underbrace{:,\dots,:}_N,1,\underbrace{:,\dots,:}_N,1}=\mathscr{Q},\ddot{\mathscr{Q}}_{:,\dots,:,1,:,\dots,:,2}=\tilde{\mathscr{Q}}, \ddot{\mathscr{Q}}_{:,\dots,:,2,:,\dots,:,1}=\tilde{\mathscr{Q}}^* $ and $\ddot{\mathscr{Q}}_{:,\dots,:,2,:,\dots,:,2}=\mathscr{Q}^*$ as shown in Fig \ref{Aug_Corr}.

\begin{figure}[h]
\center
\includegraphics[scale=0.4]{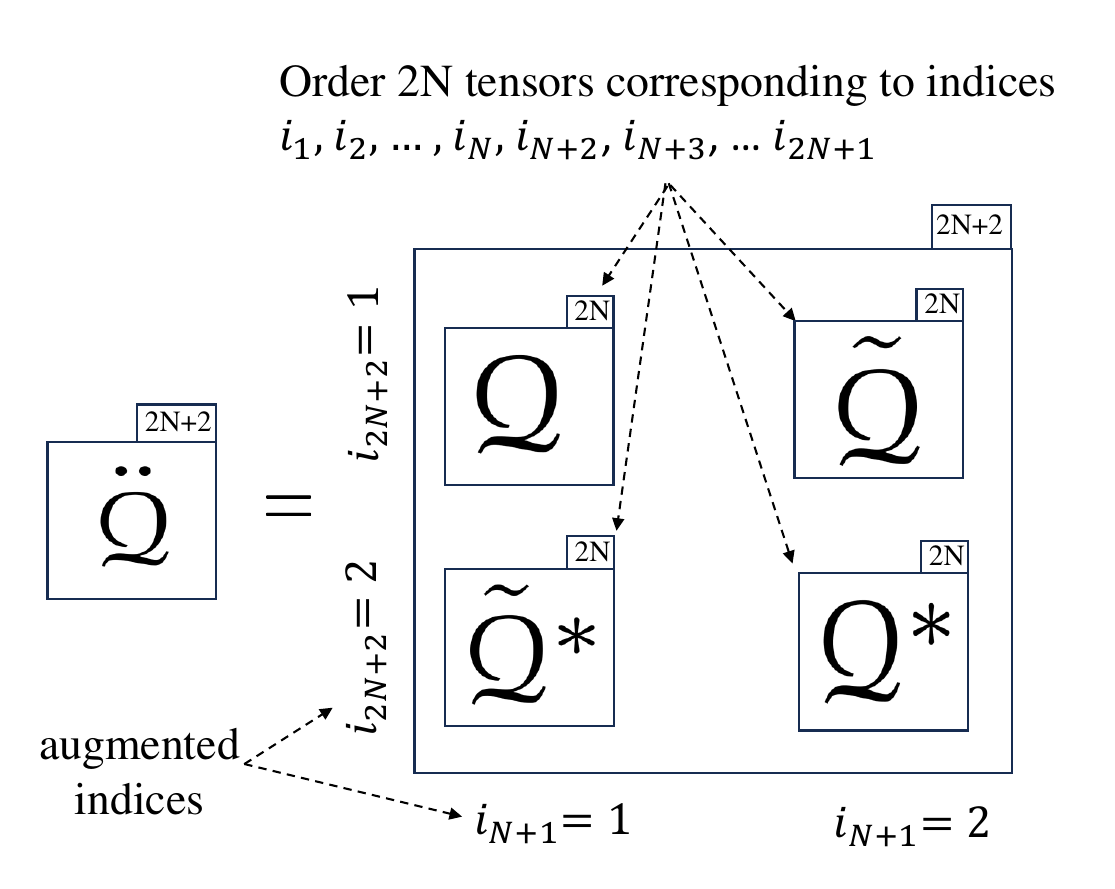}
\caption{Structure of augmented covariance tensor \label{Aug_Corr}}
\end{figure}

Similarly, cross covariance and cross pseudo-covariance between random tensors $\pmb{\mathscr{X}}$ and $\pmb{\mathscr{Y}}$ can be defined as $\mathscr{Q}_{\pmb{\mathscr{XY}}}=\mathbb{E}[(\pmb{\mathscr{X}}-\mathbb{E}[\pmb{\mathscr{X}}])\circ(\pmb{\mathscr{Y}}-\mathbb{E}[\pmb{\mathscr{Y}}])^{*}]$ and $\tilde{\mathscr{Q}}_{\pmb{\mathscr{XY}}}=\mathbb{E}[(\pmb{\mathscr{X}}-\mathbb{E}[\pmb{\mathscr{X}}])\circ(\pmb{\mathscr{Y}}-\mathbb{E}[\pmb{\mathscr{Y}}])]$ respectively. A complex random tensor is defined to be \textit{proper} if its pseudo-covariance vanishes, i.e $\tilde{\mathscr{Q}} = 0_{\mathscr{T}}$. Similarly random tensors $\pmb{\mathscr{X}}$ and $\pmb{\mathscr{Y}}$ are called \textit{cross proper} if their cross pseudo-covariance is $0_{\mathscr{T}}$.  Corresponding definitions for vectors can be found in \cite{ComplexSignalsPJS, ComplexPDFBook}.

Similar to covariance and pseudo-covariance, we can define the correlation and pseudo-correlation tensors as $\mathscr{R}=\mathbb{E}[\pmb{\mathscr{X}} \circ \pmb{\mathscr{X}}^*]$ and $\tilde{\mathscr{R}}=\mathbb{E}[\pmb{\mathscr{X}} \circ \pmb{\mathscr{X}}]$ respectively. The relation between covariance and correlation can be established as:
\begin{align}
\mathscr{Q} &= \mathbb{E}[(\pmb{\mathscr{X}}-\mathscr{M})\circ(\pmb{\mathscr{X}}-\mathscr{M})^{*}] \nonumber \\
&=\mathbb{E}[\pmb{\mathscr{X}} \circ \pmb{\mathscr{X}}^*]-\mathbb{E}[\pmb{\mathscr{X}} \circ \mathscr{M}^*] - \mathbb{E}[\mathscr{M}\circ\pmb{\mathscr{X}}^*]+ \mathbb{E}[\mathscr{M} \circ \mathscr{M}^*] \nonumber \\
&= \mathscr{R} - \mathscr{M} \circ \mathscr{M}^*.
\end{align}
Similarly, the relation between pseudo-covariance and pseudo-correlation can be shown to be:
\begin{equation}
\tilde{\mathscr{Q}} = \tilde{\mathscr{R}} - \mathscr{M} \circ \mathscr{M}.
\end{equation}
Both covariance and correlations tensors are Hermitian tensors, i.e. $\mathscr{Q} = \mathscr{Q}^H $ and $\mathscr{R} = \mathscr{R}^H $. Also, the pseudo-covariance and pseudo-correlation are partial symmetric, i.e. $\tilde{\mathscr{Q}} = \tilde{\mathscr{Q}}^T $ and $\tilde{\mathscr{R}} = \tilde{\mathscr{R}}^T $. In several engineering applications, the specific structure of the correlation tensor plays a crucial role in analyzing the statistical properties of the random tensor involved, as discussed in detail in the next section. 

\subsection{Tensor correlation structures}
The correlation of a zero mean random tensor $\pmb{\mathscr{A}} \in \mathbb{C}^{I_1 \times \dots \times I_N}$ of order $N$ is described using an order $2N$ correlation tensor of size $I_1 \times \dots \times I_N \times I_1 \times \dots \times I_N$ with each element defined as $\mathscr{R}_{i_1,\dots,i_N,i_1',\dots,i_N'}= \mathbb{E}[\pmb{\mathscr{A}}_{i_1,\dots,i_N} \cdot \pmb{\mathscr{A}}_{i_1',\dots,i_N'}^*]$. Thus the structure of $\mathscr{R}$ determines the correlation across all the elements, including elements within the same domain as well as elements across different domains. When elements of a specific domain are correlated with elements of the same domain, we will refer to it as intra-domain correlation, and when elements of a domain are correlated with elements of another domain, we will refer to it as inter-domain correlation. 

If $\pmb{\mathscr{A}}$ contains uncorrelated elements, then $\mathscr{R}$ is pseudo-diagonal. For details on pseudo-diagonal structure, refer to \cite{MDPIpaper,pandey2023linear}. In the presence of intra or inter-domain correlation, $\mathscr{R}$ contains non-zero elements on its pseudo-diagonal as well as some specific non pseudo-diagonal positions. First we consider the case with intra-domain correlation in any single mode. For such a case, we denote  the correlation tensor using the notation $\mathscr{R}^{(n)}$ where $\mathscr{R}^{(n)}$ corresponds to a correlation tensor when only the $n$th mode elements of $\pmb{\mathscr{A}}$ are correlated. Hence based on the position of non-zero elements among the indices of correlation tensor $(i_1,i_2,\dots,i_N,i_1',i_2',\dots,i_N')$ we can define different structures.

When elements are correlated only across mode-1, we have correlation tensor $\mathscr{R}^{(1)}$ such that $\mathscr{R}^{(1)}_{i_1,i_2,\dots,i_N,i_1',i_2',\dots,i_N'}$ is non-zero for any $i_1,i_1'$ and also the other indices that satisfy $i_2=i_2',i_3=i_3',\dots,i_N=i_N'$. For instance, if $\pmb{\mathscr{A}}$ is a $2 \times 2 \times 2$ tensor, and thus correlation is given by a $2 \times 2 \times 2 \times 2 \times 2 \times 2$ tensor, the non-zero elements occur at positions with $i_2=i_2',i_3=i_3'$ for all the possible values of $i_1,i_1'$. The positions are indicated in Figure \ref{Order6Corr} which illustrates an order 6 tensor. This includes all the eight pseudo-diagonal positions marked with red vertical lines where $i_1=i_1',i_2=i_2',i_3=i_3'$ which are (1,1,1,1,1,1), (1,1,2,1,1,2), (1,2,1,1,2,1), (1,2,2,1,2,2), (2,2,2,2,2,2), (2,1,1,2,1,1), (2,2,1,2,2,1), (2,1,2,2,1,2) and also the non pseudo-diagonal positions marked with green horizontal lines where $i_1 \neq i_1',i_2 = i_2',i_3=i_3'$, which are (1,1,1,2,1,1), (2,1,1,1,1,1), (1,2,1,2,2,1), (2,2,1,1,2,1), (1,1,2,2,1,2), (2,1,2,1,1,2), (1,2,2,2,2,2), (2,2,2,1,2,2).

\begin{figure}[h]
\center
\includegraphics[scale=0.7]{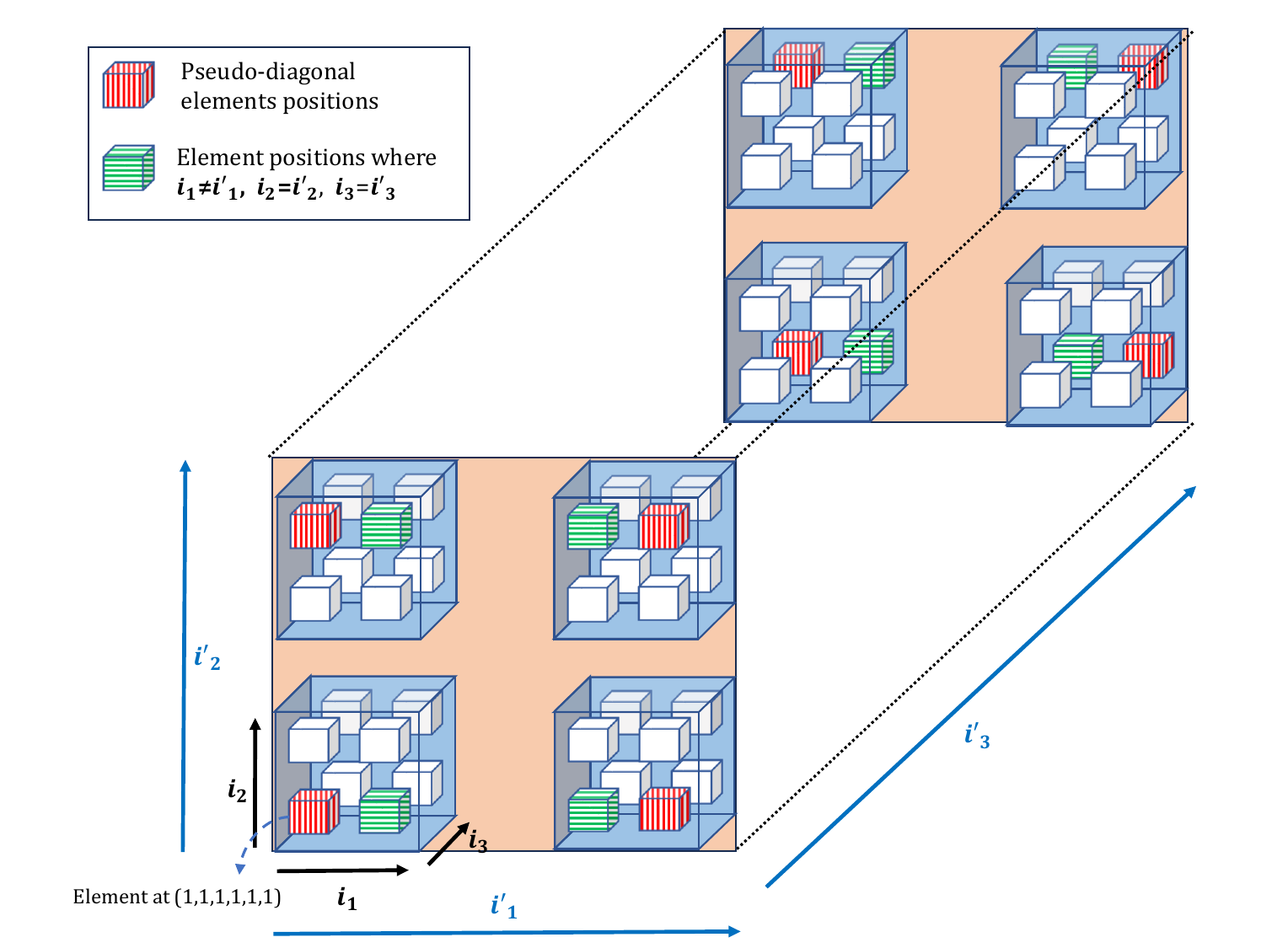}
\caption{Structure of order 6 correlation tensor when elements are correlated across mode-1 \label{Order6Corr}}
\end{figure}

Similarly with correlation in any specific mode-$n$ for $n=1,2,3,\dots,N$, we have $\mathscr{R}^{(n)}$  such that $\mathscr{R}^{(n)}_{i_1,i_2,\dots,i_N,i_1',i_2',\dots,i_N'}$ is non-zero for any $i_n,i_n'$ and also the other indices that satisfy $i_1=i_1',i_2=i_2',\dots,i_{n-1}=i_{n-1}',i_{n+1}=i_{n+1}', \dots,i_N=i_N'$. 
 
Note that any correlation tensor is a Hermitian positive semi-definite tensor and thus can be written as $\mathscr{R} = \mathscr{C} *_N \mathscr{C}^H$  where $\mathscr{C}$ is a tensor of the same size as $\mathscr{R}$ and can be seen as the square root \cite{duan2019newton} of $\mathscr{R}$, also denoted as $\mathscr{R}^{1/2}$. Since $\mathscr{R}$ is Hermitian, its tensor EVD is written as $\mathscr{R} = \mathscr{U}*_N \mathscr{D} *_N \mathscr{U}^H$ where $\mathscr{D}$ contains the real eigenvalues of $\mathscr{R}$ on its pseudo-diagonal \cite{pandey2023linear}. Thus square root of $\mathscr{R}$ is defined as $\mathscr{R}^{1/2} = \mathscr{U}*_N \mathscr{D}^{1/2} *_N \mathscr{U}^H$ where $\mathscr{D}^{1/2}$ is a pseudo-diagonal real tensor whose non zero elements are the square root of the corresponding elements of $\mathscr{D}$. Now consider a tensor $\pmb{\mathscr{B}} \in \mathbb{C}^{I_1 \times \dots \times I_N}$ with zero mean unit variance i.i.d entries such that its correlation is an order $2N$ identity tensor. We can generate an order $N$ tensor $\pmb{\mathscr{A}} \in \mathbb{C}^{I_1 \times \dots \times I_N}$ with correlation $\mathscr{R}=\mathscr{C}*_N \mathscr{C}^H$ using :
\begin{equation}\label{ARB}
\pmb{\mathscr{A}} = \mathscr{C} *_N \pmb{\mathscr{B}}.
\end{equation} 
We can verify that the correlation of $\pmb{\mathscr{A}}$ in \eqref{ARB} is indeed $\mathscr{R}$. 
\begin{align}
\mathbb{E}[\pmb{\mathscr{A}}\circ \pmb{\mathscr{A}}^*] &= \mathbb{E}[(\mathscr{C} *_N \pmb{\mathscr{B}})\circ (\mathscr{C} *_N \pmb{\mathscr{B}})^*] \\ 
&= \mathbb{E}[(\mathscr{C} *_N \pmb{\mathscr{B}})\circ (\pmb{\mathscr{B}}^* *_N \mathscr{C}^H)] \qquad \text{(using eq(22) in \cite{pandey2023linear})}\\
&=\mathscr{C} *_N \mathbb{E}[\pmb{\mathscr{B}}\circ \pmb{\mathscr{B}}^*] *_N \mathscr{C}^H \\
&=\mathscr{C} *_N \mathscr{I} *_N \mathscr{C}^H = \mathscr{R}.
\end{align}

Thus to generate a tensor with correlation in only $n$th domain such that the correlation tensor is $\mathscr{R}^{(n)} = \mathscr{C}^{(n)} *_N \mathscr{C}^{(n)H} $, we can use :
\begin{equation}
\pmb{\mathscr{A}} = \mathscr{C}^{(n)} *_N \pmb{\mathscr{B}}.
\end{equation}

Now let us consider the case where elements of two modes (assume mode-1 and mode-2) are correlated, i.e. both inter-domain and intra-domain correlation exist in mode-1 and mode-2. In this case, the  correlation tensor contains non-zero elements on its pseudo-diagonal as well as non pseudo-diagonal positions corresponding to $i_1 \neq i_1', i_2 \neq i_2'$. We denote such a correlation tensor using the notation $\mathscr{R}^{(1,2)}$ such that $\mathscr{R}^{(1,2)}_{i_1,i_2,\dots,i_N,i_1',i_2',\dots,i_N'}$ is non-zero for any $i_1,i_2,i_1',i_2'$ and where the other indices satisfy $i_3=i_3',\dots,i_N=i_N'$. In a more general case, we may have inter and intra domain correlation between any two or more domains, say mode-$k$ and mode-$m$, in which case we denote the correlation tensor using notation $\mathscr{R}^{(k,m)}$. Thus to generate a tensor $\pmb{\mathscr{A}}$ with correlation in $k$th and $m$th domain such that the correlation tensor is $\mathscr{R}^{(k,m)}$ with its square root denoted by $ \mathscr{C}^{(k,m)}$, we can use :
\begin{equation}
\pmb{\mathscr{A}} = \mathscr{C}^{(k,m)} *_N \pmb{\mathscr{B}}.
\end{equation}
Similarly the structure of the correlation tensor, depending on the position of its non-zero elements, can represent inter and intra-domain correlation across any number of domains. The structure $\mathscr{R}^{(1,2,3,\dots,N)}$ denotes the correlation of a tensor where entries across all the $N$ modes are correlated, hence all the elements of the correlation tensor are non-zero. For brevity of notation, we will drop the superscript on $\mathscr{R}^{(1,2,3,\dots,N)}$ and denote it as $\mathscr{R}$ whenever the superscript contains all the $N$ modes. 

\subsubsection{Separability of Correlation across a subset of modes}
For a matrix $\textbf{H} \in \mathbb{C}^{N_R \times N_T}$, the correlation is often represented using the Kronecker model where the correlation across the rows and columns are considered separable \cite{MIMOHaykin}. Such a correlation model is commonly used in modelling MIMO communication systems where $\textbf{H} \in \mathbb{C}^{N_R \times N_T}$ represents the channel matrix between an input/transmit side of size $N_T \times 1$ and an output/receive side of size $N_R \times 1$. The receive correlation matrix is defined as $\Psi_{R} = \mathbb{E}[\textbf{H} \textbf{H}^H]$ and the transmit correlation matrix as $\Psi_{T} = \mathbb{E}[\textbf{H}^T \textbf{H}^*]$. Correlation separability across the two domains (transmit and receiver antennas) implies that the correlation matrix of the vectorized channel, $\vecc(\textbf{H})$ is given by $\Psi_T \otimes \Psi_R$, where $\otimes$ denotes the Kronecker product. Note that the receive correlation matrix can also be written as $\Psi_R = \sum_{n=1}^{N_T}\mathbb{E}[\textbf{H}_{:,n}\textbf{H}_{:,n}^H]$, where $\textbf{H}_{:,n}$ is the vector channel between $n^{th}$ transmit antenna and all the receive antennas. Separable correlation assumes that every row vector of the matrix $\textbf{H}$ has same correlation matrix, i.e. $\mathbb{E}[\textbf{H}_{:,n}\textbf{H}_{:,n}^H]$ does not depend on $n$. Similarly, the transmit correlation matrix can be written as $\Psi_T = \sum_{m=1}^{N_R}\mathbb{E}[\textbf{H}_{m,:}\textbf{H}_{m,:}^H]$, where $\textbf{H}_{m,:}$ is the vector channel between $m^{th}$ receive antenna and all the transmit antennas. Separable correlation implies that $\mathbb{E}[\textbf{H}_{m,:}\textbf{H}_{m,:}^H]$ does not depend on $m$, i.e. every column vector of the matrix has same correlation matrix \cite{kermoal2002a, Kai2004modeling, yu2002models}. Hence under separable correlation, the receive correlation matrix is the correlation across all the receivers for a fixed transmitter, upto a scaling factor. Similarly the transmit correlation matrix is the correlation across all the transmitters for a fixed receiver. The separable correlation model may not be very accurate in all scenarios, but is still widely used because of its tractable analytical form, see for example 
\cite{CapacityCorrelatedMIMO,fu2020ber,he2018model,9234486}. 

Since a tensor has multiple modes, the separability of correlation can exist between all or a subset of the modes as well. For instance, consider a tensor channel of order $N+M$ defined as $\pmb{\mathscr{H}} \in \mathbb{C}^{J_1 \times \dots \times J_M \times I_1 \times \dots I_N}$ which links an input $\pmb{\mathscr{X}} \in \mathbb{C}^{I_1 \times \dots \times I_N}$ with an output $\pmb{\mathscr{Y}} \in \mathbb{C}^{J_1 \times \dots \times J_M}$ using the Einstein product as $\pmb{\mathscr{Y}} = \pmb{\mathscr{H}} *_N \pmb{\mathscr{X}}$. It is possible that not all the $N+M$ domains of the channel tensor are separable, but the correlation across first $M$ modes corresponding to receive side is separable from the last $N$ modes corresponding to transmit side. To understand separability of correlation in tensors, we first define the notion of transmit and receiver correlation tensors. An element in the output tensor indexed by a fixed sequence of indices say $(\bar{j}_1,\dots,\bar{j}_M)$, is a linear combination of all the elements of the input tensor where the coefficients of the combination are given by the sub-tensor $\pmb{\mathscr{H}}_{\bar{j}_1,\dots,\bar{j}_M,:,:,\dots,:}$. We denote such a sub-tensor as $\pmb{\mathscr{H}}_T^{(\bar{j}_1,\dots,\bar{j}_M)} \in \mathbb{C}^{I_1 \times \dots \times I_N}$. Hence if the elements of such a channel sub-tensor are correlated, then it will introduce correlation in the input tensor elements as seen at the $(\bar{j}_1,\dots,\bar{j}_M)$th receive element. Subsequently the transmit correlation tensor for a fixed receive element at index $(\bar{j}_1,\dots,\bar{j}_M)$ can be defined as a tensor of order $2N$ :
\begin{equation}\label{G_T_2}
\mathscr{G}_{T}^{(\bar{j}_1,\dots,\bar{j}_M)} = \mathbb{E}[\pmb{\mathscr{H}}_T^{(\bar{j}_1,\dots,\bar{j}_M)} \circ \pmb{\mathscr{H}}_T^{(\bar{j}_1,\dots,\bar{j}_M)*} ]
\end{equation} 
of size $I_1 \times \dots \times I_N \times I_1 \times \dots \times I_N$, which can be written in element wise form as:
\begin{equation}\label{G_T_3}
(\mathscr{G}_{T}^{(\bar{j}_1,\dots,\bar{j}_M)})_{i_1,\dots,i_N,i_1',\dots,i_N'} = \mathbb{E}[\pmb{\mathscr{H}}_{\bar{j}_1,\dots,\bar{j}_M,i_1,\dots,i_N} \cdot \pmb{\mathscr{H}}_{\bar{j}_1,\dots,\bar{j}_M,i_1',\dots,i_N'}^* ].
\end{equation}
Similarly, the channel between any fixed input tensor element at index $(\bar{i}_1,\dots,\bar{i}_N)$ and all the receive elements is the sub-tensor $\pmb{\mathscr{H}}_{:,\dots,:,\bar{i}_1,\dots,\bar{i}_N}$. We denote such a sub-tensor as $\pmb{\mathscr{H}}_R^{(\bar{i}_1,\dots,\bar{i}_N)} \in \mathbb{C}^{J_1 \times \dots \times J_M}$. Hence the receive correlation tensor for a fixed transmit element at index $(\bar{i}_1,\dots,\bar{i}_N)$ can be defined as a tensor of order $2M$: 
\begin{equation}\label{G_R_2}
\mathscr{G}_{R}^{(\bar{i}_1,\dots,\bar{i}_N)} = \mathbb{E}[\pmb{\mathscr{H}}_R^{(\bar{i}_1,\dots,\bar{i}_N)} \circ \pmb{\mathscr{H}}_R^{(\bar{i}_1,\dots,\bar{i}_N)*} ]
\end{equation} 
of size $J_1 \times \dots \times J_M \times J_1 \times \dots \times J_M$, which can be written in element wise form as :
\begin{equation}\label{G_R_3}
(\mathscr{G}_{R}^{(\bar{i}_1,\dots,\bar{i}_N)})_{j_1,\dots,j_M,j_1',\dots,j_M'} = \mathbb{E}[\pmb{\mathscr{H}}_{j_1,\dots,j_M,\bar{i}_1,\dots,\bar{i}_N} \cdot \pmb{\mathscr{H}}_{j_1',\dots,j_M',\bar{i}_1,\dots,\bar{i}_N}^* ].
\end{equation}

If $\mathscr{G}_R^{(\bar{i}_1,\dots,\bar{i}_N)}$ does not depend on $(\bar{i}_1,\dots,\bar{i}_N)$, we say that receive correlation tensor is uniform, or same for all the transmit domains. In such a case for notational convenience, we drop the superscript on $\mathscr{G}_R^{(i_1,\dots,i_N)}$ and denote it as $\mathscr{G}_R$.  Similarly if $\mathscr{G}_T^{(\bar{j}_1,\dots,\bar{j}_M)}$ does not depend on $(\bar{j}_1,\dots,\bar{j}_M)$, we say that transmit correlation tensor is uniform for all receive domains and denote it as $\mathscr{G}_R$. Thus with uniformity of correlation, using the simplified notation we have:
\begin{align}\label{G_R_1}
(\mathscr{G}_{R}^{(\bar{i}_1,\dots,\bar{i}_N)})_{j_1,\dots,j_M,j_1',\dots,j_M'} &= (\mathscr{G}_{R})_{j_1,\dots,j_M,j_1',\dots,j_M'} \quad \forall (\bar{i}_1,\dots,\bar{i}_N), \\ \label{G_T_1}
(\mathscr{G}_{T}^{(\bar{j}_1,\dots,\bar{j}_M)})_{i_1,\dots,i_N,i_1',\dots,i_N'} &= (\mathscr{G}_{T})_{i_1,\dots,i_N,i_1',\dots,i_N'} \quad \forall (\bar{j}_1,\dots,\bar{j}_M).
\end{align} 

In such a case correlation across the transmit and receive domains is said to be separable, since the overall channel correlation tensor can be expressed in terms of $\mathscr{G}_R$ and $\mathscr{G}_T$ as explained in the following lemma:

\begin{lemma}\label{lemma1}
For channel $\pmb{\mathscr{H}} \in \mathbb{C}^{J_1 \times \dots \times J_M \times I_1 \times \dots \times I_N }$ with zero mean unit variance circular Gaussian entries, if the elements of the correlation tensor of the channel are given by :
\begin{equation}\label{SepcortenTh}
\mathscr{R}_{j_1,\dots,j_M,i_1,\dots,i_N,j_1',\dots,j_M',i_1',\dots,i_N'} = (\mathscr{G}_{R})_{j_1,\dots,j_M,j_1',\dots,j_M'} \cdot (\mathscr{G}_{T})_{ i_1,\dots,i_N,i_1',\dots,i_N'}  
\end{equation} 
then the transmit correlation tensor given by $\mathscr{G}_T$ is uniform for all receive domains and the receive correlation tensor given by $\mathscr{G}_R$ is uniform for all transmit domains.
\end{lemma}
\begin{proof}
See Appendix
\end{proof}

\subsubsection{Issues with the  Matrix Kronecker Model}
The  Kronecker model has been commonly employed for MIMO channel modeling following \cite{kermoal2002a} which unfortunately contains some inaccuracies.
The work in \cite{kermoal2002a}  implies that for a MIMO channel matrix $\textbf{H}$ if the transmit and receive correlation matrices satisfy a uniformity property, i.e. $\text{R}_R^{(n)} = \mathbb{E}[\textbf{H}_{:,n}\textbf{H}_{:,n}^H]$ does not depend on $n$, and also  
$\text{R}_T^{(m)} = \mathbb{E}[\textbf{H}_{m,:}\textbf{H}_{m,:}^H]$ does not depend on $m$, then the correlation of the vectorized channel $\vecc(\textbf{H})$ is given as $\text{R}_{MIMO} = \text{R}_T \otimes \text{R}_R$.  The statement is true only the other way around. If the correlation  of the vectorized channel $\text{R}_{MIMO}$ is given as $\text{R}_T \otimes \text{R}_R$, then the transmit and receive correlation matrices satisfy the uniformity property. The proof of this confusing statement in \cite{kermoal2002a} considers a $2 \times 2$ matrix example, but in its equation (33), it assumes what is supposed to be proven to begin with. A simple counter example is provided to show that just because $\mathbb{E}[\textbf{H}_{:,n}\textbf{H}_{:,n}^H]$ does not depend on $n$ and
$\mathbb{E}[\textbf{H}_{m,:}\textbf{H}_{m,:}^H]$ does not depend on $m$, this does not necessarily imply that the correlation of the vectorized channel is given in terms of the Kronecker product.

Consider the following channel :
\begin{equation}
\textbf{H}=
\begin{bmatrix}
h_{11} & h_{12} \\
h_{21} & h_{22}
\end{bmatrix}
\end{equation}
Now lets consider the following specific example for the correlation matrix $\text{R}_{MIMO}$ :
\begin{equation}\label{counterexMat}
\text{R}_{MIMO} = 
\begin{bmatrix}
\mathbb{E}[h_{11}h_{11}^*] & \mathbb{E}[h_{11}h_{21}^*] & \mathbb{E}[h_{11}h_{12}^*] & \mathbb{E}[h_{11}h_{22}^*] \\
\mathbb{E}[h_{21}h_{11}^*] & \mathbb{E}[h_{21}h_{21}^*] & \mathbb{E}[h_{21}h_{12}^*] & \mathbb{E}[h_{21}h_{22}^*] \\
\mathbb{E}[h_{12}h_{11}^*] & \mathbb{E}[h_{12}h_{21}^*] & \mathbb{E}[h_{12}h_{12}^*] & \mathbb{E}[h_{12}h_{22}^*] \\
\mathbb{E}[h_{22}h_{11}^*] & \mathbb{E}[h_{22}h_{21}^*] & \mathbb{E}[h_{22}h_{12}^*] & \mathbb{E}[h_{22}h_{22}^*] \\
\end{bmatrix}
=
\begin{bmatrix}
1 & \rho & \mu & \nu \\
\rho^* & 1 & \gamma & \mu \\
\mu^* & \gamma^* & 1 & \rho \\
\nu^* & \mu^* & \rho^* & 1 \\  
\end{bmatrix}
\end{equation}
This example is a Hermitian matrix with all diagonal elements that are positive, but to ensure that this can represent a correlation matrix, specific range of values for the variables $\rho, \mu, \nu, \gamma$ need to be specified.  If a symmetric matrix is strictly row diagonally dominant and has strictly positive diagonal entries, then it is positive definite \cite[Page 549]{PosDetermBook}. A matrix $\text{A}$ of size $n \times n$ is called strictly row diagonally dominant if $|A_{i,i}| > \sum_{j=1,j\neq i}^n |\text{A}_{i,j}|$ for all $i=1,\dots,n$ \cite{PosDetermBook}.  Hence for   \eqref{counterexMat} any choice of parameters $\rho, \mu, \nu, \gamma$ which satisfy
\begin{align}
|\rho|+|\mu|+|\nu| &< 1 \\
|\rho|+|\gamma|+|\mu| &< 1
\end{align}
will yield a positive definite matrix $\text{R}_{MIMO}$.

Reference \cite{kermoal2002a} in its equation (33) assumes $\nu=\rho \mu$ and $\gamma = \rho \mu^*$ to begin with, for its proof. But that is not correct. Even if $\nu \neq \rho \mu$, and $\gamma \neq \rho \mu^*$ we may still have that the transmit and receive correlation satisfy the uniformity property as shown next:
\begin{equation}
\text{R}_R^{(1)} = \mathbb{E}[\textbf{H}_{:,1}\textbf{H}_{:,1}^H] = 
\begin{bmatrix}
\mathbb{E}[h_{11}h_{11}^*] & \mathbb{E}[h_{11}h_{21}^*] \\
\mathbb{E}[h_{21}h_{11}^*] & \mathbb{E}[h_{21}h_{21}^*] 
\end{bmatrix}
=
\begin{bmatrix}
1 & \rho\\
\rho^* & 1
\end{bmatrix}
\end{equation}
\begin{equation}
\text{R}_R^{(2)} = \mathbb{E}[\textbf{H}_{:,2}\textbf{H}_{:,2}^H] = 
\begin{bmatrix}
\mathbb{E}[h_{12}h_{12}^*] & \mathbb{E}[h_{12}h_{22}^*] \\
\mathbb{E}[h_{22}h_{12}^*] & \mathbb{E}[h_{22}h_{22}^*] 
\end{bmatrix}
=
\begin{bmatrix}
1 & \rho \\
\rho^* & 1
\end{bmatrix}
\end{equation}  
which shows that $\text{R}_R^{(1)}$ is same as $\text{R}_R^{(2)}$, i.e. $\text{R}_R^{(n)}$ does not depend on $n$. Similarly, we have :    
\begin{equation}
\text{R}_T^{(1)} = \mathbb{E}[\textbf{H}_{1,:}\textbf{H}_{1,:}^H] = 
\begin{bmatrix}
\mathbb{E}[h_{11}h_{11}^*] & \mathbb{E}[h_{11}h_{12}^*] \\
\mathbb{E}[h_{12}h_{11}^*] & \mathbb{E}[h_{12}h_{12}^*] 
\end{bmatrix}
=
\begin{bmatrix}
1 & \mu \\
\mu^* & 1
\end{bmatrix}
\end{equation}
\begin{equation}
\text{R}_T^{(2)} = \mathbb{E}[\textbf{H}_{2,:}\textbf{H}_{2,:}^H] = 
\begin{bmatrix}
\mathbb{E}[h_{21}h_{21}^*] & \mathbb{E}[h_{21}h_{22}^*] \\
\mathbb{E}[h_{22}h_{21}^*] & \mathbb{E}[h_{22}h_{22}^*] 
\end{bmatrix}
=
\begin{bmatrix}
1 & \mu \\
\mu^* & 1
\end{bmatrix}
\end{equation}  
Thus $\text{R}_T^{(1)}$ is same as $\text{R}_T^{(2)}$, i.e. $\text{R}_T^{(m)}$ and it does not depend on $m$. Essentially the transmit and receive correlation matrices do not include $\mathbb{E}[h_{11}h_{22}^*]$ given by $\nu$ and $\mathbb{E}[h_{12}h_{21}^*]$ given by $\gamma$. 
Hence, even though the transmit correlation matrix for different receive indices $m$ is the same given by $\text{R}_T = [1, \mu ; \mu^*, 1]$, and the receive correlation matrix for different transmit indices $n$ is the same given by $\text{R}_R = [1, \rho ; \rho^*, 1]$, it does not imply that  $\text{R}_{MIMO} = \text{R}_T \otimes \text{R}_R$. Since $\text{R}_T \otimes \text{R}_R$ in this case would be given as :
\begin{equation}
\text{R}_T \otimes \text{R}_R = 
\begin{bmatrix}
1 & \rho & \mu & \mu \rho \\
\rho^* & 1 & \mu \rho^* & \mu \\
\mu^* & \mu^*\rho & 1 & \rho \\
\mu^*\rho^* & \mu^* & \rho^* & 1 \\
\end{bmatrix}
\end{equation} 
which is not equal to \eqref{counterexMat}.

The separability of correlation across row and column modes generalizes to a separability across all the modes for a tensor case, which we consider next. However, it is important not to propagate the  rather misleading conclusion from \cite{kermoal2002a}  to the more general tensor case.   

\subsubsection{Separability Across all Modes}\label{sectionSepa}
A correlation separable across all the modes can be represented using the mode-n product of tensors with mode specific correlation matrices \cite{da2011multi,HoffCovariance,DenizJOAS}. Consider an order $N$ random tensor $\pmb{\mathscr{X}} \in \mathbb{C}^{I_1 \times \dots \times I_N}$ with i.i.d. zero mean and unit variance elements. Let $\Psi^{(n)} \in \mathbb{C}^{I_n \times I_n}$ for $n=1,\dots,N$ be a sequence of Hermitian matrices such that $\Psi^{(n)} = \text{A}^{(n)}\text{A}^{(n)H}$ where $\text{A}^{(n)} \in \mathbb{C}^{I_n \times I_n}$ is the square root of $\Psi^{(n)}$. The mode-$n$ product of tensor $\pmb{\mathscr{X}}$ across all the modes with these matrices is expressed using the Tucker product as \cite{favier2016nested}:
\begin{equation} \label{Hcorr_moden}
\pmb{\mathscr{X}}^{corr} = \pmb{\mathscr{X}} \times_1 \text{A}^{(1)} \times_2 \text{A}^{(2)} \times_3 \cdots \times_N \text{A}^{(N)}
\end{equation} 
Let $\vecc (\pmb{\mathscr{X}})$ be denoted as $\myvec{x}$, then using the property of mode-$n$ product from \cite{favier2016nested},\cite[Lemma 2.1]{6288476}, we can write \eqref{Hcorr_moden} using Kronecker product denoted by $\otimes$ as :
\begin{equation}
\vecc (\pmb{\mathscr{X}}^{corr}) = (\text{A}^{(N)} \otimes \cdots \otimes \text{A}^{(1)}) \myvec{x}
\end{equation}
Then the correlation matrix of the vectorized tensor is given as   :
\begin{align}
\mathbb{E}[\vecc (\pmb{\mathscr{X}}^{corr}) \vecc (\pmb{\mathscr{X}}^{corr})^H] & = \mathbb{E}[(\text{A}^{(N)}  \otimes \cdots \otimes \text{A}^{(1)})\myvec{x} \cdot \myvec{x}^H  (\text{A}^{(N)} \otimes \cdots \otimes \text{A}^{(1)})^H] \label{eq38}\\
& = (\text{A}^{(N)}  \otimes \cdots \otimes \text{A}^{(1)})\mathbb{E}[\myvec{x} \cdot \myvec{x}^H]  (\text{A}^{(N)} \otimes \cdots \otimes \text{A}^{(1)})^H \\
& = (\text{A}^{(N)}  \otimes \cdots \otimes \text{A}^{(1)})  (\text{A}^{(N)} \otimes \cdots \otimes \text{A}^{(1)})^H \\
& = (\text{A}^{(N)} \text{A}^{(N)H}  \otimes \cdots \otimes \text{A}^{(1)} \text{A}^{(1)H})\label{KronVecHcorr} \\ 
& = \Psi^{(N)}  \otimes \cdots \otimes \Psi^{(1)}
\end{align}
where \eqref{KronVecHcorr} follows from matrix Kronecker product properties \cite[Corollary 4]{zhang2013kronecker}.  Hence the correlation matrix of the vectorized tensor is given in terms of the Kronecker product of different mode-$n$ factor correlation matrices denoted by $\Psi^{(n)}$. Such a model is called a separable model and is considered for real random variables in \cite{HoffCovariance}, however we present it here for complex case. The separability for real case can be derived by replacing the Hermitian operation with transpose in \eqref{eq38}-\eqref{KronVecHcorr}.  While \eqref{KronVecHcorr} expresses the correlation as a matrix by vectorizing $\pmb{\mathscr{X}}^{corr}$, it is shown in \cite[Proposition 2.1]{HoffCovariance} that the correlation of $\pmb{\mathscr{X}}^{corr}$ from \eqref{Hcorr_moden} can also be expressed as an order $2N$ tensor obtained via the outer product of the factor matrices $\Psi^{(n)}$ defined as $\bar{\mathscr{R}}= \Psi^{(1)} \circ \cdots \circ \Psi^{(N)}$. Note that the correlation tensor when defined as $\mathscr{R} = \mathbb{E}[\pmb{\mathscr{X}}^{corr} \circ \pmb{\mathscr{X}}^{corr*}]$ is just a permuted version of $\bar{\mathscr{R}}$, where $\mathscr{R}_{i_1,\dots,i_N,i_1',\dots,i_N'} = \bar{\mathscr{R}}_{i_1,i_1',\dots,i_N,i_N'} = \mathbb{E}[ \pmb{\mathscr{X}}^{corr}_{i_1,\dots,i_N} \cdot \pmb{\mathscr{X}}_{i_1',\dots,i_N'}^{corr*}] = \Psi^{(1)}_{i_1,i_1'} \cdots \Psi^{(N)}_{i_N,i_N'}$. Hence the separable model implies that each element in the correlation tensor can be written in terms of product of the elements of the factor matrices. 

It is important to mention that while the separable model  makes structural assumptions and thus does not represent the most general second order characterization as a full correlation tensor does,  it is still often preferred since it significantly reduces the number of parameters required to specify the second order characteristics. For instance, consider an order $N$ tensor of size $L \times L \times \dots \times L$, thus containing $L^N$ elements. A full correlation model of such a tensor would be an order $2N$ tensor containing a total of  $L^{2N}$  elements. The number of pseudo-diagonal elements in such correlation tensor would be $L^N$ and non pseudo-diagonal elements would be $L^{2N}-L^N$. But since correlation tensors are Hermitian, only half of the non pseudo-diagonal elements are distinct. Thus total number of  parameters required to specify a full correlation model would be:
\begin{equation}
\gamma_{full} = \dfrac{1}{2}(L^{2N}-L^N)+ L^N = \dfrac{1}{2}(L^{2N}+L^N)
\end{equation}
which is exponential in order of the tensor $N$. On the other hand the separable covariance model is represented using $N$ factor matrices of size $L \times L$ which are all Hermitian. Thus the total number of parameters required to specify the separable correlation model  is given by: 
\begin{equation}
\gamma_{sep} = \sum_{n=1}^{N}\dfrac{1}{2}(L^{2}+L)= \dfrac{N}{2}(L^{2}+L)
\end{equation}
which is linear in $N$. The ratio $\gamma = \gamma_{sep}/\gamma_{full}$ is plotted against $N$ in Figure \ref{Fig2_ref} for different values of $L$. For the vector case, $N=1$, the correlation is a matrix, and thus the ratio is 1. However as $N$ increases the ratio drops exponentially clearly indicating that the number of parameters to represent full correlation would be too large in high order tensors as compared to a separable model. Moreover, as $L$ increases, the ratio decreass with increasing $N$ even faster.

\begin{figure}[h]
\center
\includegraphics[scale=0.5]{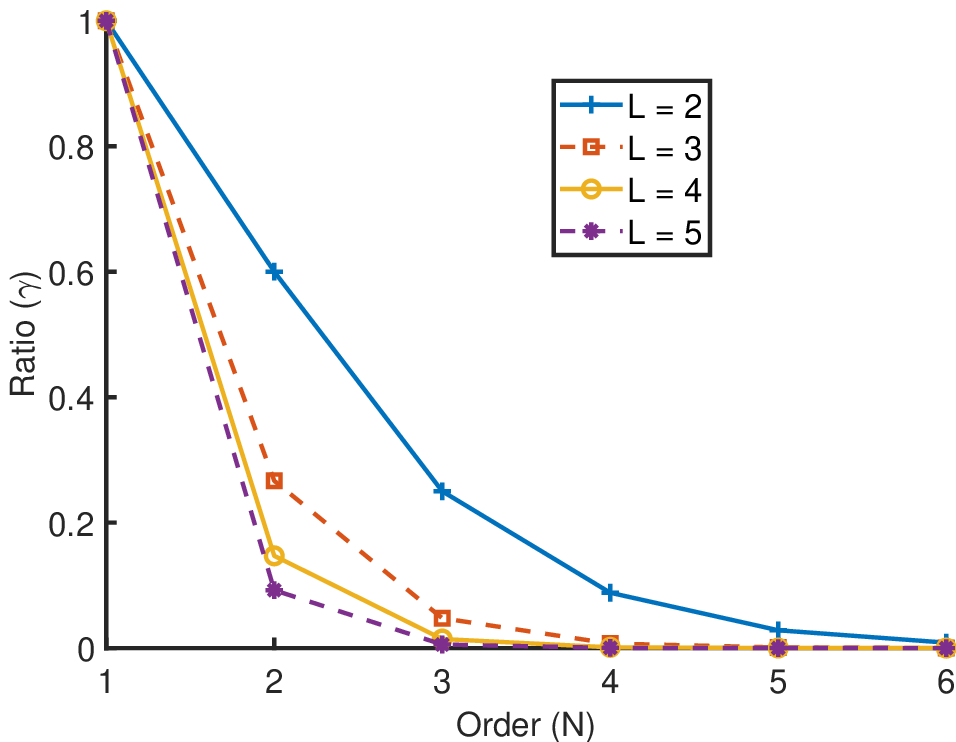}
\caption{Comparison of number of parameters needed to characterize correlation under full and separable assumptions   \label{Fig2_ref}}
\end{figure}

\section{Complex Gaussian tensors}\label{Sec4}
A Gaussian tensor is a random tensor where all its elements are random variables with a joint Gaussian distribution. 
\subsection{The complex Gaussian distribution}
Very often while discussing complex Gaussian distributions, the random entities are implicitly assumed to be proper. However, for a more general case including both proper and improper complex random entities, a complex Gaussian PDF of a vector is defined using its augmented or composite representation. Details on the complex Gaussian PDF for a vector and how it simplifies for the proper case can be found in \cite{ComplexPDFBook}. Since the distribution is a scalar function of all the individual elements of the tensor, it can be defined using the vectorization of the augmented tensor. 
  
The PDF of  Gaussian distributed complex-valued tensor $\pmb{\mathscr{X}} \in {\mathbb{C}}^{I_1 \times I_2 \times \dots \times I_N} $ of order $N$ can be specified using its augmented version as:
\begin{equation}\label{vectorGaussDist}
p_{{}_{\pmb{\mathscr{X}}}}(\text{x}) =p_{{}_{{\pmb{\mathscr{X}}^a}}}(\text{x}^a)= \dfrac{1}{{(\pi)}^{I_1I_2\dots I_N} (\det(\ddot{\text{Q}}))^{1/2} } \times \exp \Big\{ -\dfrac{1}{2}(\text{x}^a-\text{m}^a)^H \ddot{\text{Q}}^{-1} (\text{x}^a - \text{m}^a) \Big\}
\end{equation}
where $\text{x}^a=\vecc(\mathscr{X}^a)$, $\text{m}^a=\vecc(\mathbb{E}[\pmb{\mathscr{X}}^a])$ and $\ddot{\text{Q}}$ is the covariance matrix of the vectorized augmented tensor. However, Einstein product properties can be exploited to describe the PDF without any reshaping or vectorizing of the tensor as follows:
\begin{equation}\label{TensorGaussDist}
p_{{}_{\pmb{\mathscr{X}}}}(\mathscr{X}) = p_{{}_{\pmb{\mathscr{X}}^a}}(\mathscr{X}^a)= \dfrac{1}{{(\pi)}^{I_1I_2\dots I_N} (\det(\ddot{\mathscr{Q}}))^{1/2} } \times \exp \Big\{ -\dfrac{1}{2}(\mathscr{X}^a-\mathscr{M}^a)^* *_{N+1} \ddot{\mathscr{Q}}^{-1} *_{N+1} (\mathscr{X}^a - \mathscr{M}^a) \Big\}
\end{equation} 
where $\mathscr{M}^a = \mathbb{E}[\pmb{\mathscr{X}}^a]$ is the order-$N+1$ augmented mean tensor and $\ddot{\mathscr{Q}}=\mathbb{E}[(\pmb{\mathscr{X}}^a-\mathscr{M}^a) \circ (\pmb{\mathscr{X}}^a-\mathscr{M}^a)^*]$ is the order-$2N+2$ covariance  of the augmented tensor. The equivalence of \eqref{vectorGaussDist} and \eqref{TensorGaussDist} can be directly established based on Lemma 1 from \cite{AccessPaper}. 

\subsection{Circular Symmetric Gaussian distribution}

Statistically, a random tensor $\pmb{\mathscr{X}}$ is called \textit{symmetric} in distribution if $\pmb{\mathscr{X}}$ and $-\pmb{\mathscr{X}}$ have the same probability distribution. Note that it is different than the symmetric tensor definition which applies to the structure of deterministic tensors and states that a symmetric tensor is one which remains invariant under any permutation of its indices \cite{KoldaTensor}.  Further, a random tensor $\pmb{\mathscr{X}}$ is called \textit{circular symmetric} if $\pmb{\mathscr{X}}$ and $e^{j\phi}\pmb{\mathscr{X}}$ have the same distribution for any $\phi \in \mathbb{R}$, which intuitively implies that the distribution of circular $\pmb{\mathscr{X}}$ is rotation invariant. Based on the definition of proper and circular tensors, it can be readily established that  a complex Gaussian tensor is circular symmetric if and only if it is zero mean and proper. Figure \ref{Fig3_ref} shows the scatter plot of a $10 \times 10 \times 10$ tensor containing complex Gaussian distributed elements of different kinds to illustrate the difference between circular symmetric, proper and improper Gaussian distribution.

\begin{figure}[h]
\center
\includegraphics[scale=0.7]{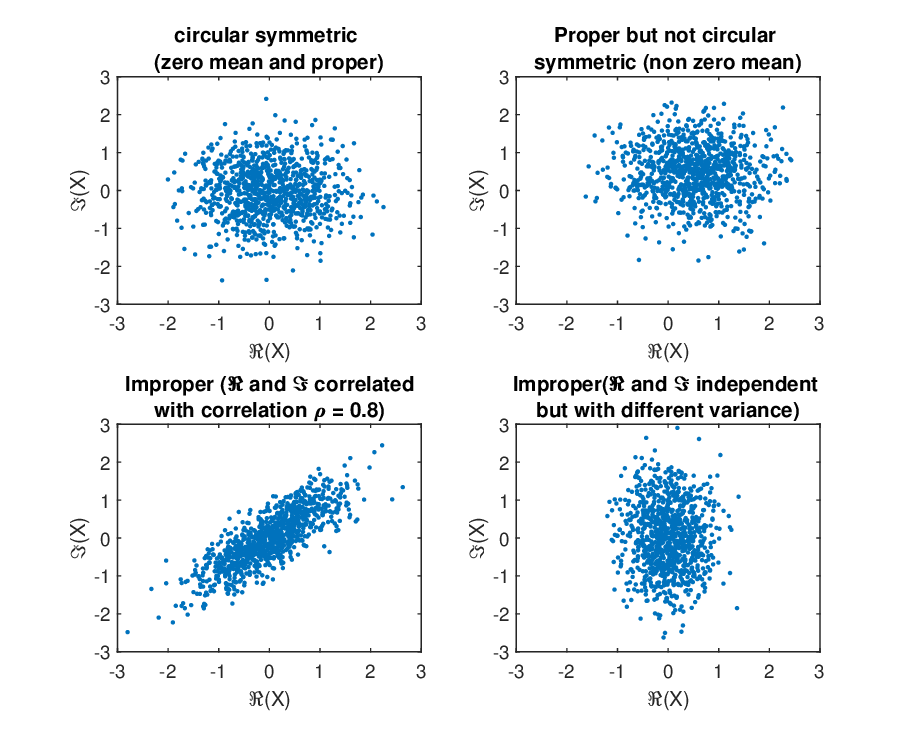}
\caption{Different forms of Gaussian distribution \label{Fig3_ref}}
\end{figure}

For proper complex Gaussian PDF we have  $\tilde{\text{Q}}=0_{\mathscr{T}}$, implying that $\ddot{\text{Q}}$ is block diagonal with $\text{Q}$ and $\text{Q}^*$ being the two blocks. So it is shown in \cite{ComplexPDFBook} that \eqref{vectorGaussDist} reduces to, 
\begin{equation}\label{vectorGaussDist1}
p_{{}_{\pmb{\mathscr{X}}}}(\text{x}) = \dfrac{1}{{(\pi)}^{I_1I_2\dots I_N} (\det(\text{Q})) } \times \exp \Big\{ -(\text{x}-\text{m})^H \text{Q}^{-1} (\text{x} - \text{m}) \Big\}
\end{equation}
which can also be written using tensor representation as:
\begin{equation}\label{TensorProperGaussDist}
p_{{}_{\pmb{\mathscr{X}}}}(\mathscr{X}) = \dfrac{1}{{(\pi)}^{I_1I_2\dots I_N} \det(\mathscr{Q}) } \times \exp \Big\{ -(\mathscr{X}-\mathscr{M})^* *_N \mathscr{Q}^{-1} *_N (\mathscr{X} - \mathscr{M}) \Big\}
\end{equation} 
where $\mathscr{M} = \mathbb{E}[\pmb{\mathscr{X}}]$ is the order-$N$ mean tensor and $\mathscr{Q}=\mathbb{E}[(\pmb{\mathscr{X}}-\mathscr{M}) \circ (\pmb{\mathscr{X}}-\mathscr{M})^*]$ is the order-$2N$ covariance tensor. Further, the PDF of a circular symmetric complex Gaussian tensor is defined as:
\begin{equation}\label{TensorCircSymGaussDist}
p_{\pmb{\mathscr{X}}}(\mathscr{X}) = \dfrac{1}{{(\pi)}^{I_1I_2\dots I_N} \det(\mathscr{Q}) } \times \exp \Big\{ -\mathscr{X}^* *_N \mathscr{Q}^{-1} *_N \mathscr{X} \Big\}.
\end{equation} 

Circular symmetry is a common assumption when considering Gaussian distributed complex random variables, as it implies the real and imaginary parts are independent and have same variance. However, in a more general set up it is important to distinguish between proper (zero pseudo-covariance), circular symmetric (zero mean and proper), and non circular distributions when dealing with complex Gaussian tensors. In case of real tensors, the Gaussian distribution is defined as:
\begin{equation}\label{realgaussiantensor}
p_{{}_{\pmb{\mathscr{X}}}}(\mathscr{X}) = \dfrac{\exp \Big\{  -\frac{1}{2}(\mathscr{X}-\mathscr{M}) *_N \mathscr{Q}^{-1} *_N (\mathscr{X} - \mathscr{M}) \Big\}}{\sqrt{{(2\pi)}^{I_1I_2\dots I_N} \det(\mathscr{Q})} } 
\end{equation}
which for the standard case (zero mean, identity covariance tensor), reduces to \cite{DenizJOAS}:
\begin{equation}\label{realstandard}
p_{{}_{\pmb{\mathscr{X}}}}(\mathscr{X}) = \dfrac{\exp \big( -\frac{1}{2}||\mathscr{X}||^2 \big)}{\sqrt{{(2\pi)}^{I_1I_2\dots I_N}} }. 
\end{equation}
Notationally, we can write  the general case from \eqref{TensorGaussDist} as $\pmb{\mathscr{X}} \sim \mathcal{CN}(\mathscr{M},\mathscr{Q},\tilde{\mathscr{Q}})$ or $\pmb{\mathscr{X}}^a \sim \mathcal{CN}(\mathscr{M}^a,\ddot{\mathscr{Q}})$, the proper case from \eqref{TensorProperGaussDist} as  $\pmb{\mathscr{X}} \sim \mathcal{CN}(\mathscr{M},\mathscr{Q})$, the circular symmetric case from \eqref{TensorCircSymGaussDist} as $\pmb{\mathscr{X}} \sim \mathcal{CN}(0_{\mathscr{T}},\mathscr{Q})$, the real case from \eqref{realgaussiantensor} as  $\pmb{\mathscr{X}} \sim \mathcal{N}(\mathscr{M},\mathscr{Q})$, and the real standard Gaussian case from \eqref{realstandard} as $\pmb{\mathscr{X}} \sim \mathcal{N}(0_{\mathscr{T}},\mathscr{I}_N)$, where $\mathscr{I}_N$ denotes a pseudo-diagonal identity tensor of order $2N$.

\subsection{Gaussian distribution with separable covariance}
The notion of separable covariance as discussed in Section \ref{sectionSepa}, is often used to describe a special case of proper Gaussian distributed random tensors where the covariance can be represented as separable via the Tucker product \cite{HoffCovariance}. A detailed discussion on statistical properties of random tensors with this structure can be found in \cite{DavidThesis, HoffCovariance, MLND}. While most such references deal with real tensors, we can generalize the concepts to complex tensors using the results from Section \ref{sectionSepa}. Let $\pmb{\mathscr{Z}} \in \mathbb{C}^{I_1 \times \dots \times I_K}$ be a tensor of order $K$ with zero mean and variance one entries. Now if we define a set of matrices $\{\text{A}^{(k)} \in \mathbb{C}^{I_k \times I_k} \}_{k=1}^{K}$, then the Tucker product of  $\pmb{\mathscr{Z}}$ with these matrices yield another complex tensor $\pmb{\mathscr{Y}}$: 
\begin{equation}
\pmb{\mathscr{Y}} = \pmb{\mathscr{Z}} {\times}_{1} \text{A}^{(1)} {\times}_{2} \text{A}^{(2)} {\times}_{3} \dots {\times}_{K} \text{A}^{(K)}.
\end{equation}
This Tucker product induces a transformation on the covariance structure of $\pmb{\mathscr{Z}}$. Since $\pmb{\mathscr{Z}}$ contains zero mean unit variance i.i.d. elements, its covariance is an identity tensor. Alternately, we also have  $\Cov[\vecc(\pmb{\mathscr{Z}})] = \mathbb{E}[\vecc(\pmb{\mathscr{Z}})\cdot \vecc(\pmb{\mathscr{Z}})^H]$ as an identity matrix. Let ${ \Sigma}_i = \text{A}^{(i)} \text{A}^{(i)H}$, then using the derivation from \eqref{eq38}-\eqref{KronVecHcorr}, the covariance of $\pmb{\mathscr{Y}}$ can be expressed using Kronecker Delta Matrix Structure by vectorizing the tensor as:
\begin{equation}\label{CovVecTensor}
\Cov[\vecc(\pmb{\mathscr{Y}})] = {\Sigma}_1 \otimes {\Sigma}_2 \otimes \dots \otimes {\Sigma}_K.
\end{equation}
Using this separability, we can analyze the covariance along all the domains of the tensor to account for variance in specific domains separately. For example, with a tensor of order 2, i.e. a matrix, the covariance will be a fourth order tensor but its elements can be described as Kronecker product of two matrices ${\bf \Sigma}_1 \otimes {\bf \Sigma}_2$, in which case ${\bf \Sigma}_1$ represents ``row covariance" and ${\bf \Sigma}_2$ represents ``column covariance". Similarly, for a general $K^{th}$ order tensor, each ${\bf \Sigma}_k$ will correspond to the covariance in the $k^{th}$ mode.
Based on this, a family of normal distributions for random tensors with separable covariance structure as in \eqref{CovVecTensor} can be generated using the Tucker product. A tensor $\pmb{\mathscr{Y}}$ of order $K$ can be written as:
\begin{equation}
\pmb{\mathscr{Y}} = \mathscr{M} + (\pmb{\mathscr{Z}} {\times}_{1} \text{A}^{(1)} {\times}_{2} \text{A}^{(2)} {\times}_{3} \dots {\times}_{K} \text{A}^{(K)})
\end{equation} 
where $\pmb{\mathscr{Z}}$ is a tensor of independent standard normal entries, $\mathscr{M}$ is the mean tensor and $\text{A}^{(i)}$ are matrices such that the $i^{th}$ domain covariance is defined as ${\Sigma}_i =\text{A}^{(i)}\text{A}^{(i)H}$. Notationally, it is often written as $\myvec{y} = \vecc(\pmb{\mathscr{Y}}) \sim \mathcal{CN}(\text{m}, {\Sigma}_1 \otimes \dots \otimes {\Sigma}_K)$ where $ \text{m} = \vecc(\mathscr{M})$. The density function of such a distribution can be written using the complex multivariate Gaussian distribution as:
\begin{equation}\label{NormalDistReal}
p_{{}_{\pmb{\mathscr{Y}}}}(\text{y}) = \dfrac{1}{{(\pi)}^{J} \big( {\prod}_{i=1}^{K} \det(\Sigma_i)^{J/(J_i)} \big) } \times \exp \Big\{ -(\text{y}-\text{m})^H {\Sigma}_{1:K}^{-1} (\text{y} - \text{m}) \Big\}
\end{equation}
where $J = \prod_{k=1}^{K}J_k$ and ${\Sigma}_{1:K}^{-1} = \big( {\Sigma}_1 \otimes {\Sigma}_2 \otimes \dots \otimes {\Sigma}_K \big)^{-1}$. A specific case of this can be seen in the definition of matrix variate normal distribution \cite{MatrixVariateBook}. A random matrix ${\bf X}$ of size $P \times N$ is said to have a matrix variate normal distribution with mean matrix $\text{M}$ and covariance matrix ${{\Sigma}}_1 \otimes {{\Sigma}}_2$, if $\vecc({\bf X}) \sim \mathcal{CN}(\vecc(\text{M}),{{\Sigma}}_1 \otimes {{\Sigma}}_2)$.  

\section{Complex tensor random processes}\label{Sec5}
A complex random tensor process $\{ \pmb{\mathscr{X}}[t]:t \in \mathcal{T}\}$ defined over a probability space $(\Omega, \mathcal{F}, P)$ is a collection of complex random tensors over the same probability space indexed by the variable $t$. When the index set $\mathcal{T}$ is continuous, it denotes a continuous complex random tensor process. When $\mathcal{T}$ is discrete, it denotes a discrete complex random tensor process, also referred as complex random tensor sequence. Consequently, a complex random tensor process, defined over a probability space can be viewed as a function tensor $\pmb{\mathscr{X}}[t,\omega]$ where $t \in \mathcal{T}, \omega \in \Omega$. In other words, a function tensor $\pmb{\mathscr{X}}[t] \in \mathbb{C}_t^{I_1 \times \dots \times I_N}$ is a complex random function tensor if its components are complex random processes. For details on function tensors, refer to \cite{pandey2023linear}.
Throughout this section, unless otherwise specified, the term `random process' applies to both continuous and discrete cases.
\subsection{Auto-correlation and Cross-correlation}
The \textit{mean} of a complex tensor random process $\pmb{\mathscr{X}}[k] \in \mathbb{C}_k^{I_1 \times \dots \times I_N}$ is defined as $\mathscr{M}[k] = \mathbb{E} \big[\pmb{\mathscr{X}}[k] \big]$
with components $\mathscr{M}_{i_1, \dots ,i_N}[k] = \mathbb{E} \big[\pmb{\mathscr{X}}_{i_1, \dots ,i_N}[k] \big]$.

The \textit{auto-correlation} function of a complex tensor random process $\pmb{\mathscr{X}}[k] \in \mathbb{C}_k^{I_1 \times \dots \times I_N}$ is a tensor $\mathscr{R}_{\pmb{\mathscr{X}}}[k,i] \in \mathbb{C}_{(k,i)}^{I_1 \times \dots \times I_N \times I_1 \times \dots \times I_N}$ defined as:
\begin{equation}
\mathscr{R}_{\pmb{\mathscr{X}}}[k,i] = \mathbb{E} \big[\pmb{\mathscr{X}}[k] \circ \pmb{\mathscr{X}}^{*}[k - i] \big].
\end{equation}
The pseudo-diagonal elements of $\mathscr{R}_{\pmb{\mathscr{X}}}[k,i]$, denoted by $\mathscr{R}_{\pmb{\mathscr{X}}_{i_1,\dots ,i_N,i_1,\dots ,i_N}}[k,i]$, are the auto-correlation functions of $\pmb{\mathscr{X}}_{i_1, \dots ,i_N}[k]$ and the cross-correlation between two different components $\pmb{\mathscr{X}}_{i_1, \dots , i_N}[k]$ and $\pmb{\mathscr{X}}_{i_1', \dots ,i_N'}[k]$ is embedded in the off pseudo-diagonal elements $\mathscr{R}_{\pmb{\mathscr{X}}_{i_1,\dots ,i_N,i_1',\dots ,i_N'}}[i]$.

We can define auto-covariance of $\pmb{\mathscr{X}}[k]$ as the auto-correlation of the corresponding centralized (zero mean) tensor process given by  
$(\pmb{\mathscr{X}}[k]- \mathbb{E}[\pmb{\mathscr{X}}[k]])$.

The \textit{cross-correlation} of two different complex tensor random processes  $\pmb{\mathscr{X}}[k] \in \mathbb{C}_k^{I_1 \times \dots \times I_N}$ and $\pmb{\mathscr{Y}}[k] \in \mathbb{C}_k^{J_1 \times \dots \times J_M}$ is a tensor $\mathscr{R}_{\pmb{\mathscr{X}},\pmb{\mathscr{Y}}}[k,i] \in \mathbb{C}_k^{I_1 \times \dots \times I_N \times J_1 \times \dots \times J_M}$ defined as:
\begin{equation}
\mathscr{R}_{\pmb{\mathscr{X}},\pmb{\mathscr{Y}}}[k,i] = \mathbb{E} \big[\pmb{\mathscr{X}}[k] \circ \pmb{\mathscr{Y}}^{*}[k - i] \big]
\end{equation}
where $\mathscr{R}_{\pmb{\mathscr{X}},\pmb{\mathscr{Y}}_{i_1, \dots ,i_N,j_1, \dots ,j_M}}[k,i] = \mathbb{E}\bigg[\pmb{\mathscr{X}}_{i_1, \dots , i_N}[k]\pmb{\mathscr{Y}}^*_{j_1 , \dots , j_M}[k-i]\bigg]$.

Similar to auto-covariance, the cross-covariance between $\pmb{\mathscr{X}}[k]$ and $\pmb{\mathscr{Y}}[k]$ can be defined as the cross-correlation between their respective centralized tensor processes given by  $(\pmb{\mathscr{X}}[k]- \mathbb{E}[\pmb{\mathscr{X}}[k]])$ and $(\pmb{\mathscr{Y}}[k]- \mathbb{E}[\pmb{\mathscr{Y}}[k]])$. 
 
\subsection{Wide Sense Stationary (WSS) and jointly WSS sequences}

A tensor random process $\pmb{\mathscr{X}}[k] \in \mathbb{C}_k^{I_1 \times \dots \times I_N}$ is called wide sense stationary (WSS) if its mean $\mathbb{E} \big[ \pmb{\mathscr{X}}[k] \big]$ is independent of the index $k$ and its auto-correlation $\mathbb{E} \big[\pmb{\mathscr{X}}[k] \circ \pmb{\mathscr{X}}^{*}[k - i] \big]$ depends only on $i$. Two WSS tensor random processes $\pmb{\mathscr{X}}[k] \in \mathbb{C}_k^{I_1 \times \dots \times I_N}$ and $\pmb{\mathscr{Y}}[k] \in \mathbb{C}_k^{J_1 \times \dots \times J_M}$ are jointly WSS if their cross-correlation $\mathbb{E} \big[\pmb{\mathscr{X}}[k] \circ \pmb{\mathscr{Y}}^{*}[k-i] \big]$ depends only on $i$. 

When a tensor random process is passed through a multi-linear time invariant system, the output is given by the contracted convolution of the input tensor with the system tensor \cite{pandey2023linear, AdithyaPaper}.  It can be shown that if the input to a such a system tensor is a WSS tensor process, then the output is also WSS and the input and output tensor processes are jointly WSS. In the following lemma, we establish this result for tensor random sequences, but similarly can be proven for continuous random processes as well. 

\begin{lemma}
If the input $\pmb{\mathscr{X}}[k] \in \mathbb{C}_k^{I_1 \times \dots \times I_N}$ to a multi-linear time invariant system with an impulse  response $\mathscr{H}[k] \in \mathbb{C}_k^{J_1 \times \dots \times J_M \times I_1 \times \dots \times I_N}$ is WSS, then the output which is given by a discrete contracted convolution $\pmb{\mathscr{Y}}[k] = \sum_{n}\mathscr{H}[n]*_N\pmb{\mathscr{X}}[k-n]$, is also WSS, and the input and the output tensor sequences are jointly WSS. 
\end{lemma}

\begin{proof}
The mean of the output sequence $\pmb{\mathscr{Y}}[k]$ can be written as:
\begin{align}
\mathbb{E} \big[ \pmb{\mathscr{Y}}[k] \big] &= \mathbb{E} \big[ \sum_{n = -\infty}^{+\infty}\mathscr{H}[n]*_N\pmb{\mathscr{X}}[k-n] \big] \nonumber \\
&= \sum_{n = -\infty}^{+\infty}\mathscr{H}[n]*_N \mathbb{E} \big[ \pmb{\mathscr{X}}[k-n]\big].
\end{align}
Since the summation is over all values of $n$ and $\pmb{\mathscr{X}}[k]$ is WSS, so $\mathbb{E} \big[ \pmb{\mathscr{X}}[k-n]\big]$ is constant, hence $\mathbb{E} \big[ \pmb{\mathscr{Y}}[k] \big]$ also does not depend on $k$.

The auto-correlation of $\pmb{\mathscr{Y}}[k]$ is
\begin{align}
\mathbb{E} \big[\pmb{\mathscr{Y}}[k] \circ \pmb{\mathscr{Y}}^{*}[k - i] \big] &= \mathbb{E} \bigg[ \big( \sum_{n = -\infty}^{+\infty}\mathscr{H}[n]*_N \pmb{\mathscr{X}}[k-n]\big) \circ  \big( \sum_{m = -\infty}^{+\infty}\mathscr{H}[m]*_N \pmb{\mathscr{X}}[k-i-m] \big)^{*} \bigg] \notag\\\label{transposestep1}
&= \sum_{n = -\infty}^{+\infty}\sum_{m = -\infty}^{+\infty}\mathbb{E} \bigg[ \big( \mathscr{H}[n]*_N \pmb{\mathscr{X}}[k-n] \big) \circ  \big( \mathscr{H}[m]*_N \pmb{\mathscr{X}}[k-i-m] \big)^{*} \bigg].
\end{align}
Using the commutativity result from  \cite[equation (22)]{pandey2023linear}), we can write \eqref{transposestep1} as
\begin{align}
\mathbb{E} \big[\pmb{\mathscr{Y}}[k] \circ \pmb{\mathscr{Y}}^{*}[k - i] \big] &= \sum_{n = -\infty}^{+\infty}\sum_{m = -\infty}^{+\infty}\mathbb{E} \bigg[ \big( \mathscr{H}[n]*_N\pmb{\mathscr{X}}[k-n] \big) \circ  \big( \{\pmb{\mathscr{X}}[k-i-m]*_N\mathscr{H}^{T}[m] \big)^{*} \bigg] \notag\\
&= \sum_{n = -\infty}^{+\infty}\sum_{m = -\infty}^{+\infty}\mathbb{E} \bigg[ \big( \mathscr{H}[n]*_N \pmb{\mathscr{X}}[k-n] \big) \circ  \big( \pmb{\mathscr{X}}^{*}[k-i-m]*_N \mathscr{H}^{H}[m] \big) \bigg].
\end{align} 
Further using the associativity property of the Einstein product from \cite{pandey2023linear}, we get:
\begin{align}
\mathbb{E} \big[\pmb{\mathscr{Y}}[k] \circ \pmb{\mathscr{Y}}^{*}[k - i] \big] &= \sum_{n = -\infty}^{+\infty}\sum_{m = -\infty}^{+\infty}  \big( \mathscr{H}[n]*_N  \mathbb{E} \Big[ \pmb{\mathscr{X}}[k-n] \circ   \pmb{\mathscr{X}}^{*}[k-i-m]\Big] *_N \mathscr{H}^{H}[m] \big) \\
\label{wss1}
\mathscr{R}_{\pmb{\mathscr{Y}}}[i]&= \sum_{n = -\infty}^{+\infty}\sum_{m = -\infty}^{+\infty} \mathscr{H}[n] *_N \mathscr{R}_{\pmb{\mathscr{X}}}[m+i-n]*_N\mathscr{H}^{H}[m].
\end{align}
Since the output auto-correlation tensor as shown in \eqref{wss1} is obtained by summing over $n$ and $m$ it depends only on $i$. Furthermore since the  mean tensor of $\pmb{\mathscr{Y}}[k]$ is constant, the output tensor random sequence is WSS. Further to show that it is jointly WSS with the input sequence, we note that the cross-correlation between $\pmb{\mathscr{X}}[k]$ and $\pmb{\mathscr{Y}}[k]$ can be written as:
\begin{align}
\mathbb{E} \big[\pmb{\mathscr{Y}}[k] \circ \pmb{\mathscr{X}}^{*}[k - i] \big] = \sum_{n = -\infty}^{+\infty} \mathbb{E} \bigg[ \big(\mathscr{H}[n]*_N \pmb{\mathscr{X}}[k-n]\big) \circ \pmb{\mathscr{X}}^{*}[k-i] \bigg]. 
\end{align}
Using the associativity property of the Einstein product from \cite{pandey2023linear} we get:
\begin{align}
\mathbb{E} \big[\pmb{\mathscr{Y}}[k] \circ \pmb{\mathscr{X}}^{*}[k - i] \big] 
&= \sum_{n = -\infty}^{+\infty}  \mathscr{H}[n]*_N \mathbb{E} \bigg[ \pmb{\mathscr{X}}[k-n]  \circ \pmb{\mathscr{X}}^{*}[k-i] \bigg] \\\label{wss2}
\mathscr{R}_{\pmb{\mathscr{Y}},\pmb{\mathscr{X}}}[i]&= \sum_{n = -\infty}^{+\infty}\mathscr{H}[n] *_N \mathscr{R}_{\pmb{\mathscr{X}}}[i-n]
\end{align}
which depends only on $i$ and hence the cross-correlation of $\pmb{\mathscr{X}}[k]$ and $\pmb{\mathscr{Y}}[k]$ also depends only on $i$. Thus, the output and input are jointly WSS.
\end{proof}

\subsection{Power Spectrum Density (PSD) and Power Cross-spectrum Density (PCD)} 
The power spectrum density (PSD) of a WSS random tensor process is defined as the Fourier transform of its auto-correlation. Thus, for a random tensor sequence $\pmb{\mathscr{X}}[k] \in \mathbb{C}_k^{I_1 \times \dots \times I_N}$, the PSD is defined as
\begin{equation}
\breve{\mathscr{S}}_{\pmb{\mathscr{X}}}[f] = \sum_{i=-\infty}^{+\infty}\mathscr{R}_{\pmb{\mathscr{X}}}[i]e^{-j2\pi f i}
\end{equation}
and for a continuous random tensor process $\pmb{\mathscr{X}}[t] \in \mathbb{C}_t^{I_1 \times \dots \times I_N}$, the PSD is defined as
\begin{equation}
\breve{\mathscr{S}}_{\pmb{\mathscr{X}}}[f] = \int_{-\infty}^{\infty} \mathscr{R}_{\pmb{\mathscr{X}}}[\tau]e^{-j2\pi f \tau} d\tau
\end{equation}
where the PSD $\breve{\pmb{\mathscr{S}}}_{\pmb{\mathscr{X}}}[f] \in \mathbb{C}_f^{I_1 \times \dots \times I_N \times I_1 \times \dots \times I_N}$ is an order $2N$ tensor function.
The pseudo-diagonal elements of $\breve{\mathscr{S}}_{\pmb{\mathscr{X}}}[f]$, denoted by $\breve{\mathscr{S}}_{\pmb{\mathscr{X}}_{i_1,\dots ,i_N,i_1,\dots ,i_N}}[f]$, are the power spectra of $\pmb{\mathscr{X}}_{i_1, \dots ,i_N}$ and the power cross-spectrum between two different components $\pmb{\mathscr{X}}_{i_1, \dots , i_N}$ and $\pmb{\mathscr{X}}_{i_1', \dots ,i_N'}$ is captured by the off pseudo-diagonal elements denoted by $\breve{\mathscr{S}}_{\pmb{\mathscr{X}}_{i_1,\dots ,i_N,i_1',\dots ,i_N'}}[f]$. 

Similarly, the \textit{power cross-spectrum density} (PCD) of two jointly WSS tensor processes $\pmb{\mathscr{X}}[k] \in \mathbb{C}_k^{I_1 \times \dots \times I_N}$ and $\pmb{\mathscr{Y}}[k] \in \mathbb{C}_k^{J_1 \times \dots \times J_M}$ is a tensor $\breve{\mathscr{S}}_{\pmb{\mathscr{Y}},\pmb{\mathscr{X}}}[f] \in \mathbb{C}_f^{J_1 \times \dots \times J_M \times I_1 \times \dots \times I_N}$ given by the Fourier transform of their cross-correlation tensor $\mathscr{R}_{\pmb{\mathscr{Y}},\pmb{\mathscr{X}}}[i]$. The relationship between input and output PSD/PCD can be established when a WSS complex random tensor process is passed through a multi-linear time invariant system. In the following lemma, we establish such a relation for complex random tensor sequences, but similarly can be proven for continuous case as well.

\begin{lemma}
Consider a WSS input random sequence $\pmb{\mathscr{X}}[k] \in \mathbb{C}_k^{I_1 \times \dots \times I_N}$ to a multi-linear time invariant system with an impulse  response $\mathscr{H}[k] \in \mathbb{C}_k^{J_1 \times \dots \times J_M \times I_1 \times \dots \times I_N}$ such that the output sequence $\pmb{\mathscr{Y}}[k] \in \mathbb{C}_k^{J_1 \times \dots \times J_M}$ is also WSS. Let $\breve{\mathscr{S}}_{\pmb{\mathscr{Y}}}[f]$, $\breve{\mathscr{S}}_{\pmb{\mathscr{X}}}[f]$, $\breve{\mathscr{S}}_{{\pmb{\mathscr{Y}}\pmb{\mathscr{X}}}}[f]$ and $\breve{\mathscr{H}}[f]$ denote the output PSD, input PSD, output-input PCD, and the Fourier transform of the impulse response respectively, then we have:
\begin{equation}\label{spec1}
\breve{\mathscr{S}}_{\pmb{\mathscr{Y}}}[f] = \breve{\mathscr{H}}[f] *_N \breve{\mathscr{S}}_{\pmb{\mathscr{X}}}[f]*_N \breve{\mathscr{H}}^H[f]
\end{equation}
\begin{equation}\label{spec2}
\breve{\mathscr{S}}_{{\pmb{\mathscr{Y}},\pmb{\mathscr{X}}}}[f] = \breve{\mathscr{H}}[f]*_N \breve{\mathscr{S}}_{\pmb{\mathscr{X}}}[f]
\end{equation}
\end{lemma}
\begin{proof}
Note that \eqref{wss1} represents the output auto-correlation as a discrete contracted convolution between $\mathscr{H}[i], \mathscr{R}_{\pmb{\mathscr{X}}}[i]$ and $\mathscr{H}^H[i]$. Also \eqref{wss2} represents the output-input cross correlation as a discrete contracted convolution between $\mathscr{H}[i]$ and $\mathscr{R}_{\pmb{\mathscr{X}}}[i]$. Discrete contracted convolution in time domain translates to Einstein product in frequency domain as shown in \cite{pandey2023linear}, thus \eqref{spec1} and \eqref{spec2} are direct consequences of taking Fourier transforms of \eqref{wss1} and \eqref{wss2} respectively.
\end{proof}
\section{Tensor spectrum and asymptotic properties of random tensors}\label{sec6}
Spectral theory of matrices and in particular the study of matrix eigenvalues and singular values has played a crucial role in matrix analysis.
Moreover, the study of the asymptotic behavior of the matrix spectrum when the size of the matrix goes towards infinity has lead to the development of Random Matrix Theory (RMT) \cite{RandomMatrixForWireless}.

The extension of RMT to tensors requires the definition of tensor spectrum. However, there is no consensus on such a notion.
Several definitions of tensor eigenpairs, i.e. tensor eigenvalues and their corresponding eigenvectors, have been proposed as generalizations of their matrices counterparts.
These definitions are often defined through fixed point equations involving a specific tensor contraction and a specific normalization of the eigenvector.

For example, H-eigenpairs have been considered independently in \cite{QiLEigen} and \cite{limsingular} as real solutions of a fixed point equation. 
In \cite{QiLEigen}, Qi also considered Z-eigenvalues (respectively E-eigenvalues) as real (respectively complex) solutions of a second fixed point equation.
Qi et al. introduced in \cite{qiwangwu2008} a generalization of the Z-eigenvalues, relaxing the eigenvector normalization, that they named D-eigenvalues.
However, one of the drawback of these definition if that the computation of such eigenpairs are often NP-hard \cite{chang2013}.
On another line of research, a tensor eigenvalue definition related to eigenvalues of the matrix mapping of tensors defined in \eqref{tranformprop} has been considered e.g. in \cite{itskov, LuEVD,pandey2023linear}.
This definition leads to a notion of eigenpair easier to manipulate.
We will simply call them \textit{tensor eigenpairs}.


\subsection{Random Matrix Theory}
One of the aim of RMT is to study the distribution of the eigenvalues of a random matrix.
In particular, we can define their empirical distribution as follows.
\begin{definition}
Let $\textbf{A}\in\mathbb{C}^{I\times J}$ be a random matrix and note $\lambda_1,\dots, \lambda_{I}$ its (left) singular values \cite{RandomMatrixForWireless}. 
The empirical singular value distribution of $\textbf{A}$ is then the distribution defined by
\begin{equation}
f_I(\lambda) = \frac{1}{I}\sum_{i=1}^{I} {\delta}_{\lambda=\lambda_i}
\end{equation}
where $\lambda\mapsto{\delta}_{\lambda=\cdot}$ denotes the Dirac delta function \cite{RandomMatrixForWireless}.
\end{definition}
Major results in RMT provide proofs of convergence of this empirical distribution under weak conditions on the distributions of the entries of a random matrix (making these results universal since they apply to a large family of distributions and not only to Gaussian distributions for example).
We will present the natural extension of such results to the tensor eigenvalue decomposition (EVD) and singular value decomposition (SVD) considered in \cite{pandey2023linear}.
Before doing so, let us introduce two limit distributions classical in RMT.

\begin{definition}\cite{arnold71}
A semi circle distribution of parameter $\beta$ is a distribution of support $[-\beta, \beta]$ with density given by
\begin{equation}\label{eq:semicircle}
f_{\beta}(x) = \frac{2}{\pi \beta^2}\sqrt{\beta^2-x^2}.
\end{equation}
\end{definition}

\begin{definition}\cite[Section 3.2]{RandomMatrixForWireless}
A Marcenko-Pastur distribution of parameter $c$ is a distribution of support included in $[0, +\infty)$ with density equal to
\begin{equation}\label{eq:marcenko}
f_c(x) =\left\{\begin{array}{ll}
 \Big(1-\frac{1}{c}\Big){\delta}_{x=0} + \frac{1}{2\pi cx}\sqrt{(c_+-x)(x-c_-)}&\text{if $c>1$}\\
\frac{1}{2\pi cx}\sqrt{(c_+-x)(x-c_-)}&\text{if $c\leq 1$}
\end{array}\right.
\end{equation}
with
\begin{equation}
\left\{\begin{array}{l}
c_+ = (1+\sqrt{c})^2\\
c_- = (1-\sqrt{c})^2
\end{array}\right..
\end{equation}
\end{definition}


\subsection{Tensor eigenvalues and eigenvectors}\label{sec:evd}

The EVD of a Hermitian tensor $\mathscr{A} \in \mathbb{C}^{I_1 \times \dots \times I_N \times I_1 \times \dots \times I_N}$ is characterized by the existence of a unitary tensor $\mathscr{U} \in \mathbb{C}^{I_1 \times \dots \times I_N \times I_1 \times \dots \times I_N}$  and a pseudo-diagonal tensor $\mathscr{D} \in \mathbb{C}^{I_1 \times \dots \times I_N \times I_1 \times \dots \times I_N}$ (whose non-zero values are the tensor eigenvalues of $\mathscr{A}$) such that (see \cite{pandey2023linear}) :
\begin{equation}\label{EVDTensor}
\mathscr{A} = \mathscr{U}*_N \mathscr{D} *_N \mathscr{U}^H.
\end{equation}
Then, we can relate the asymptotic distribution of the tensor eigenvalues to the semi-circular distribution in Theorem~\ref{th:semicircle}. 
Figure~\ref{fig:semicircle} illustrates the result of Theorem~\ref{th:semicircle} and shows the accuracy of the limiting distribution as the tensor sizes increases.
\begin{theorem}\label{th:semicircle}
Given an Hermitian tensor $\pmb{\mathscr{A}} \in \mathbb{C}^{I_1 \times \dots \times I_N \times I_1 \times \dots \times I_N }$, with i.i.d. entries having zero-mean and unit variance, as $\prod_k I_k\rightarrow\infty$, the distribution of the eigenvalues of $\frac{1}{\sqrt{\prod_{k=1}^N I_k}}\pmb{\mathscr{A}}$ converges\footnote{The convergence considered in this section is the weak convergence; see \cite{RandomMatrixForWireless} for details.} almost surely to the semi circle distribution of parameter $\beta = 2$ with density equal to \eqref{eq:semicircle}.
\end{theorem}
\begin{proof}
Define $\textbf{A}=f_{I_1,...,I_N|I_1,...,I_N}(\pmb{\mathscr{A}})$ with $f_{I_1,...,I_N|J_1,...,J_N}$ being the matricization operator defined in \eqref{tranformprop}.
Then the tensor eigenvalues of $\pmb{\mathscr{A}}$ are equal to the eigenvalues of $\textbf{A}$. We can therefore apply \cite[Th. 2.11]{RandomMatrixForWireless} and the result follows.
\end{proof}

\begin{figure}
\include{subfig}
\centering
\subfloat[Tensor of size $10\times 10\times 10 \times 10$]{
%
%
\definecolor{mycolor1}{rgb}{0.00000,0.44700,0.74100}%
\begin{tikzpicture}
\scriptsize
\begin{axis}[%
width=2.621in,
height=1.866in,
at={(0.758in,0.481in)},
scale only axis,
xmin=-2.5,
xmax=2.5,
ymin=0,
ymax=0.4,
axis background/.style={fill=white},
legend style={legend cell align=left, align=left, draw=white!15!black}
]
\addplot [color=black, line width=2.0pt]
  table[row sep=crcr]{%
-2	0\\
-1.95959595959596	0.06365872942525\\
-1.91919191919192	0.0895665391360429\\
-1.87878787878788	0.109129250443286\\
-1.83838383838384	0.125353575699109\\
-1.7979797979798	0.139409977575227\\
-1.75757575757576	0.151901487890119\\
-1.71717171717172	0.163187865565038\\
-1.67676767676768	0.17350445003183\\
-1.63636363636364	0.183015316743589\\
-1.5959595959596	0.191840338337376\\
-1.55555555555556	0.200070292479357\\
-1.51515151515152	0.207775902647265\\
-1.47474747474747	0.215013550218888\\
-1.43434343434343	0.221829044769409\\
-1.39393939393939	0.228260203928808\\
-1.35353535353535	0.234338672780801\\
-1.31313131313131	0.240091240349262\\
-1.27272727272727	0.245540813497393\\
-1.23232323232323	0.250707151398219\\
-1.19191919191919	0.255607428883026\\
-1.15151515151515	0.260256675047744\\
-1.11111111111111	0.264668119316167\\
-1.07070707070707	0.268853467759932\\
-1.03030303030303	0.272823126108214\\
-0.98989898989899	0.276586381481645\\
-0.949494949494949	0.280151551792053\\
-0.909090909090909	0.283526109539519\\
-0.868686868686869	0.286716785135556\\
-0.828282828282828	0.289729653703338\\
-0.787878787878788	0.292570208429322\\
-0.747474747474747	0.295243422880904\\
-0.707070707070707	0.29775380420277\\
-0.666666666666667	0.300105438719035\\
-0.626262626262626	0.302302031169163\\
-0.585858585858586	0.304346938571744\\
-0.545454545454545	0.306243199525667\\
-0.505050505050505	0.307993559611523\\
-0.464646464646465	0.309600493438511\\
-0.424242424242424	0.311066223787309\\
-0.383838383838384	0.312392738222188\\
-0.343434343434343	0.313581803482522\\
-0.303030303030303	0.314634977911674\\
-0.262626262626263	0.31555362213789\\
-0.222222222222222	0.316338908185412\\
-0.181818181818182	0.316991827163208\\
-0.141414141414141	0.317513195652296\\
-0.101010101010101	0.317903660889791\\
-0.0606060606060606	0.318163704827662\\
-0.0202020202020201	0.31829364712625\\
0.0202020202020203	0.31829364712625\\
0.0606060606060606	0.318163704827662\\
0.101010101010101	0.317903660889791\\
0.141414141414141	0.317513195652296\\
0.181818181818182	0.316991827163208\\
0.222222222222222	0.316338908185412\\
0.262626262626263	0.31555362213789\\
0.303030303030303	0.314634977911674\\
0.343434343434343	0.313581803482522\\
0.383838383838384	0.312392738222188\\
0.424242424242424	0.311066223787309\\
0.464646464646465	0.309600493438511\\
0.505050505050505	0.307993559611523\\
0.545454545454545	0.306243199525667\\
0.585858585858586	0.304346938571744\\
0.626262626262626	0.302302031169163\\
0.666666666666667	0.300105438719035\\
0.707070707070707	0.29775380420277\\
0.747474747474747	0.295243422880904\\
0.787878787878788	0.292570208429322\\
0.828282828282828	0.289729653703338\\
0.868686868686869	0.286716785135556\\
0.909090909090909	0.283526109539519\\
0.949494949494949	0.280151551792053\\
0.98989898989899	0.276586381481645\\
1.03030303030303	0.272823126108214\\
1.07070707070707	0.268853467759932\\
1.11111111111111	0.264668119316167\\
1.15151515151515	0.260256675047744\\
1.19191919191919	0.255607428883026\\
1.23232323232323	0.250707151398219\\
1.27272727272727	0.245540813497393\\
1.31313131313131	0.240091240349262\\
1.35353535353535	0.234338672780801\\
1.39393939393939	0.228260203928808\\
1.43434343434343	0.221829044769409\\
1.47474747474747	0.215013550218888\\
1.51515151515152	0.207775902647265\\
1.55555555555556	0.200070292479357\\
1.5959595959596	0.191840338337376\\
1.63636363636364	0.183015316743589\\
1.67676767676768	0.173504450031829\\
1.71717171717172	0.163187865565038\\
1.75757575757576	0.151901487890119\\
1.7979797979798	0.139409977575227\\
1.83838383838384	0.125353575699109\\
1.87878787878788	0.109129250443286\\
1.91919191919192	0.0895665391360429\\
1.95959595959596	0.0636587294252498\\
2	0\\
};
\addlegendentry{Semi-circle distribution \eqref{eq:semicircle} with $\beta=2$}

\addplot[ybar interval, fill=mycolor1, fill opacity=0.6, draw=black, area legend] table[row sep=crcr] {%
x	y\\
-2	0.115384615384615\\
-1.74	0.230769230769231\\
-1.48	0.230769230769231\\
-1.22	0.269230769230769\\
-0.96	0.269230769230769\\
-0.7	0.307692307692308\\
-0.44	0.269230769230769\\
-0.18	0.307692307692308\\
0.0800000000000001	0.346153846153846\\
0.34	0.307692307692307\\
0.6	0.307692307692307\\
0.86	0.307692307692308\\
1.12	0.192307692307692\\
1.38	0.230769230769231\\
1.64	0.153846153846154\\
1.9	0.153846153846154\\
};
\addlegendentry{Empirical distribution histogram}

\end{axis}
\end{tikzpicture}%
}
\hspace{0.2cm}
\subfloat[Tensor of size $30\times 30\times 30 \times 30$]{
%
%
\definecolor{mycolor1}{rgb}{0.00000,0.44700,0.74100}%
\begin{tikzpicture}
\scriptsize
\begin{axis}[%
width=2.621in,
height=1.866in,
at={(0.758in,0.481in)},
scale only axis,
xmin=-2.5,
xmax=2.5,
ymin=0,
ymax=0.4,
axis background/.style={fill=white},
legend style={legend cell align=left, align=left, draw=white!15!black}
]
\addplot [color=black, line width=2.0pt]
  table[row sep=crcr]{%
-2	0\\
-1.95959595959596	0.06365872942525\\
-1.91919191919192	0.0895665391360429\\
-1.87878787878788	0.109129250443286\\
-1.83838383838384	0.125353575699109\\
-1.7979797979798	0.139409977575227\\
-1.75757575757576	0.151901487890119\\
-1.71717171717172	0.163187865565038\\
-1.67676767676768	0.17350445003183\\
-1.63636363636364	0.183015316743589\\
-1.5959595959596	0.191840338337376\\
-1.55555555555556	0.200070292479357\\
-1.51515151515152	0.207775902647265\\
-1.47474747474747	0.215013550218888\\
-1.43434343434343	0.221829044769409\\
-1.39393939393939	0.228260203928808\\
-1.35353535353535	0.234338672780801\\
-1.31313131313131	0.240091240349262\\
-1.27272727272727	0.245540813497393\\
-1.23232323232323	0.250707151398219\\
-1.19191919191919	0.255607428883026\\
-1.15151515151515	0.260256675047744\\
-1.11111111111111	0.264668119316167\\
-1.07070707070707	0.268853467759932\\
-1.03030303030303	0.272823126108214\\
-0.98989898989899	0.276586381481645\\
-0.949494949494949	0.280151551792053\\
-0.909090909090909	0.283526109539519\\
-0.868686868686869	0.286716785135556\\
-0.828282828282828	0.289729653703338\\
-0.787878787878788	0.292570208429322\\
-0.747474747474747	0.295243422880904\\
-0.707070707070707	0.29775380420277\\
-0.666666666666667	0.300105438719035\\
-0.626262626262626	0.302302031169163\\
-0.585858585858586	0.304346938571744\\
-0.545454545454545	0.306243199525667\\
-0.505050505050505	0.307993559611523\\
-0.464646464646465	0.309600493438511\\
-0.424242424242424	0.311066223787309\\
-0.383838383838384	0.312392738222188\\
-0.343434343434343	0.313581803482522\\
-0.303030303030303	0.314634977911674\\
-0.262626262626263	0.31555362213789\\
-0.222222222222222	0.316338908185412\\
-0.181818181818182	0.316991827163208\\
-0.141414141414141	0.317513195652296\\
-0.101010101010101	0.317903660889791\\
-0.0606060606060606	0.318163704827662\\
-0.0202020202020201	0.31829364712625\\
0.0202020202020203	0.31829364712625\\
0.0606060606060606	0.318163704827662\\
0.101010101010101	0.317903660889791\\
0.141414141414141	0.317513195652296\\
0.181818181818182	0.316991827163208\\
0.222222222222222	0.316338908185412\\
0.262626262626263	0.31555362213789\\
0.303030303030303	0.314634977911674\\
0.343434343434343	0.313581803482522\\
0.383838383838384	0.312392738222188\\
0.424242424242424	0.311066223787309\\
0.464646464646465	0.309600493438511\\
0.505050505050505	0.307993559611523\\
0.545454545454545	0.306243199525667\\
0.585858585858586	0.304346938571744\\
0.626262626262626	0.302302031169163\\
0.666666666666667	0.300105438719035\\
0.707070707070707	0.29775380420277\\
0.747474747474747	0.295243422880904\\
0.787878787878788	0.292570208429322\\
0.828282828282828	0.289729653703338\\
0.868686868686869	0.286716785135556\\
0.909090909090909	0.283526109539519\\
0.949494949494949	0.280151551792053\\
0.98989898989899	0.276586381481645\\
1.03030303030303	0.272823126108214\\
1.07070707070707	0.268853467759932\\
1.11111111111111	0.264668119316167\\
1.15151515151515	0.260256675047744\\
1.19191919191919	0.255607428883026\\
1.23232323232323	0.250707151398219\\
1.27272727272727	0.245540813497393\\
1.31313131313131	0.240091240349262\\
1.35353535353535	0.234338672780801\\
1.39393939393939	0.228260203928808\\
1.43434343434343	0.221829044769409\\
1.47474747474747	0.215013550218888\\
1.51515151515152	0.207775902647265\\
1.55555555555556	0.200070292479357\\
1.5959595959596	0.191840338337376\\
1.63636363636364	0.183015316743589\\
1.67676767676768	0.173504450031829\\
1.71717171717172	0.163187865565038\\
1.75757575757576	0.151901487890119\\
1.7979797979798	0.139409977575227\\
1.83838383838384	0.125353575699109\\
1.87878787878788	0.109129250443286\\
1.91919191919192	0.0895665391360429\\
1.95959595959596	0.0636587294252498\\
2	0\\
};
\addlegendentry{Semi-circle distribution \eqref{eq:semicircle} with $\beta=2$}

\addplot[ybar interval, fill=mycolor1, fill opacity=0.6, draw=black, area legend] table[row sep=crcr] {%
x	y\\
-2	0.0763888888888889\\
-1.84	0.152777777777778\\
-1.68	0.1875\\
-1.52	0.222222222222222\\
-1.36	0.256944444444445\\
-1.2	0.256944444444445\\
-1.04	0.277777777777778\\
-0.88	0.284722222222222\\
-0.72	0.305555555555556\\
-0.56	0.305555555555555\\
-0.4	0.326388888888889\\
-0.24	0.3125\\
-0.0800000000000001	0.326388888888889\\
0.0800000000000001	0.3125\\
0.24	0.312500000000001\\
0.4	0.3125\\
0.56	0.298611111111111\\
0.72	0.291666666666667\\
0.88	0.277777777777778\\
1.04	0.263888888888889\\
1.2	0.243055555555556\\
1.36	0.215277777777778\\
1.52	0.1875\\
1.68	0.159722222222223\\
1.84	0.0833333333333333\\
2	0.0833333333333333\\
};
\addlegendentry{Empirical distribution histogram}

\end{axis}
\end{tikzpicture}%
}
\caption{Empirical tensor eigenvalues histogram and its asymptotic distribution for Hermitian random tensors}
\label{fig:semicircle}
\end{figure}
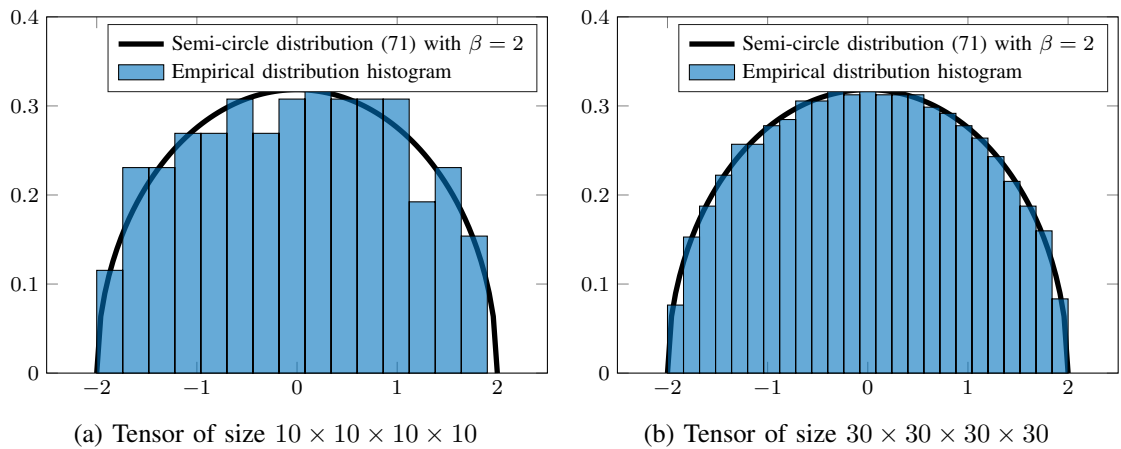

\subsection{Tensor singular values and singular vectors}\label{sec:svd}
The SVD of a tensor $\mathscr{A} \in \mathbb{C}^{I_1 \times \dots \times I_N \times J_1 \times \dots \times J_M}$ is characterized by the existence of unitary tensors $\mathscr{U} \in \mathbb{C}^{I_1 \times \dots \times I_N \times I_1 \times \dots \times I_N}$ and $\mathscr{V} \in \mathbb{C}^{J_1 \times \dots \times J_M \times J_1 \times \dots \times J_M}$ and a pseudo-diagonal tensor $\mathscr{D} \in \mathbb{C}^{I_1 \times \dots \times I_N \times J_1 \times \dots \times J_M}$ (whose non-zero values are the tensor singular values of $\mathscr{A}$) such that (see \cite{pandey2023linear}) :
\begin{equation}\label{SVDTensor}
\mathscr{A} = \mathscr{U}*_N \mathscr{D} *_M \mathscr{V}^H.
\end{equation}
Then, we can relate the asymptotic distribution of the tensor singular values to the Marcenko-Pastur distribution.
\begin{theorem}\label{th:marcenko}
Given a tensor, $\pmb{\mathscr{A}} \in \mathbb{C}^{I_1 \times \dots \times I_N \times J_1 \times \dots \times J_M }$, with i.i.d entries with zero-mean and unit variance, as $\prod_k I_k, \prod_k J_k\rightarrow\infty$ with $\frac{\prod_k I_k}{\prod_k J_k}\rightarrow c\in (0,+\infty)$, the distribution of the squared singular values of $\frac{1}{\sqrt{\prod_{k=1}^N J_k}}\pmb{\mathscr{A}}$ converge almost surely to the Marcenko-Pastur distribution of density \eqref{eq:marcenko}.
\end{theorem}
\begin{proof}
Define $\textbf{A}=f_{I_1,...,I_N|J_1,...,J_N}(\pmb{\mathscr{A}})$ where $f_{I_1,...,I_N|J_1,...,J_N}$ denotes the matricization operator defined in \eqref{tranformprop}. Then the tensor singular values of $\pmb{\mathscr{A}}$ are equal to the singular values of $\textbf{A}$. We can therefore apply Theorem 2.13 in \cite{RandomMatrixForWireless} to get the result.
\end{proof}

Similarly to the tensor eigenvalues, we illustrate the result of Theorem~\ref{th:marcenko} in Figure~\ref{fig:marcenko}. 
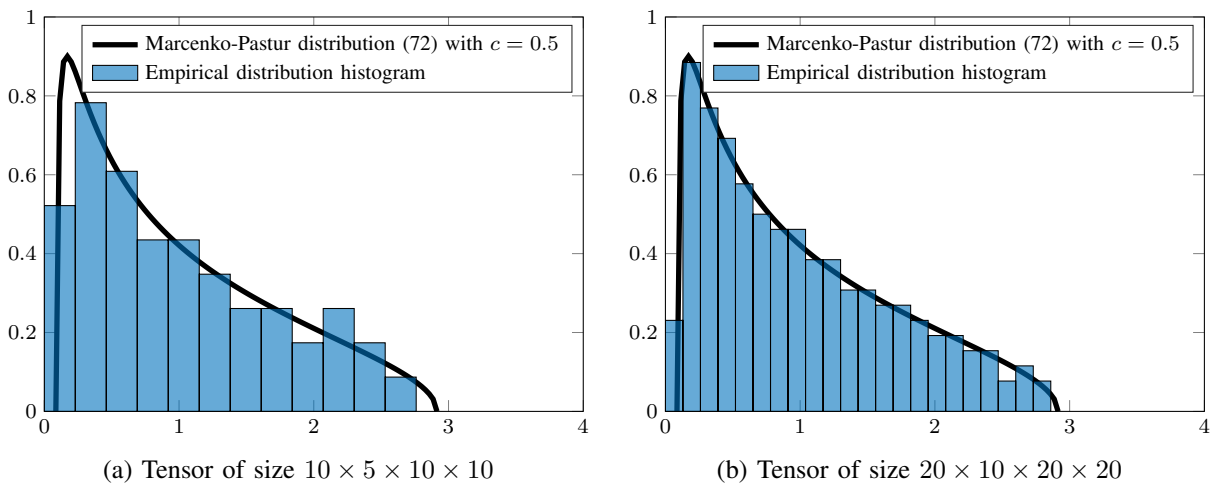
\begin{figure}
\centering
\subfloat[Tensor of size $10\times 5\times 10 \times 10$]{
%
%
\definecolor{mycolor1}{rgb}{0.00000,0.44700,0.74100}%
\begin{tikzpicture}
\scriptsize
\begin{axis}[%
width=2.821in,
height=2.066in,
at={(0.758in,0.481in)},
scale only axis,
xmin=0,
xmax=4,
ymin=0,
ymax=1,
axis background/.style={fill=white},
legend style={legend cell align=left, align=left, draw=white!15!black}
]
\addplot [color=black, line width=2.0pt]
  table[row sep=crcr]{%
0.0857864376269049	0\\
0.114356408583937	0.787249614004345\\
0.142926379540969	0.886233981353344\\
0.171496350498002	0.899914578825413\\
0.200066321455034	0.886089800403947\\
0.228636292412066	0.862310523570274\\
0.257206263369098	0.835209615446203\\
0.28577623432613	0.807563628378639\\
0.314346205283163	0.780579953704442\\
0.342916176240195	0.754769710483123\\
0.371486147197227	0.730318506716542\\
0.400056118154259	0.70725607786645\\
0.428626089111292	0.685538064347222\\
0.457196060068324	0.665086831168446\\
0.485766031025356	0.64581222976621\\
0.514336001982388	0.627622167030847\\
0.54290597293942	0.610427856302289\\
0.571475943896453	0.594146248735935\\
0.600045914853485	0.578700962657039\\
0.628615885810517	0.56402242083675\\
0.657185856767549	0.55004758371961\\
0.685755827724582	0.536719492084972\\
0.714325798681614	0.523986736242092\\
0.742895769638646	0.511802914939299\\
0.771465740595678	0.500126116777846\\
0.80003571155271	0.48891843977801\\
0.828605682509743	0.478145555150151\\
0.857175653466775	0.467776316063098\\
0.885745624423807	0.457782409438925\\
0.914315595380839	0.448138047473916\\
0.942885566337872	0.438819695070829\\
0.971455537294904	0.429805829304649\\
1.00002550825194	0.421076727220498\\
1.02859547920897	0.412614278554655\\
1.057165450166	0.404401820306589\\
1.08573542112303	0.396423990431936\\
1.11430539208007	0.388666598252097\\
1.1428753630371	0.381116509475446\\
1.17144533399413	0.373761543993941\\
1.20001530495116	0.366590384856883\\
1.22858527590819	0.35959249703209\\
1.25715524686523	0.352758054746392\\
1.28572521782226	0.346077876354729\\
1.31429518877929	0.3395433658232\\
1.34286515973632	0.333146460028757\\
1.37143513069335	0.32687958117931\\
1.40000510165039	0.320735593745024\\
1.42857507260742	0.314707765366383\\
1.45714504356445	0.30878973126895\\
1.48571501452148	0.302975461769912\\
1.51428498547852	0.297259232508929\\
1.54285495643555	0.291635597076261\\
1.57142492739258	0.286099361745763\\
1.59999489834961	0.28064556204965\\
1.62856486930664	0.275269440956623\\
1.65713484026368	0.269966428435554\\
1.68570481122071	0.264732122203648\\
1.71427478217774	0.259562269471403\\
1.74284475313477	0.254452749506659\\
1.77141472409181	0.249399556846966\\
1.79998469504884	0.244398784993147\\
1.82855466600587	0.239446610417464\\
1.8571246369629	0.234539276716774\\
1.88569460791993	0.229673078734347\\
1.91426457887697	0.224844346463015\\
1.942834549834	0.220049428526425\\
1.97140452079103	0.215284675013371\\
1.99997449174806	0.210546419411426\\
2.0285444627051	0.205830959348476\\
2.05711443366213	0.201134535802411\\
2.08568440461916	0.196453310376906\\
2.11425437557619	0.191783340161333\\
2.14282434653323	0.187120549589784\\
2.17139431749026	0.182460698581048\\
2.19996428844729	0.177799346068073\\
2.22853425940432	0.173131807798365\\
2.25710423036135	0.168453106986673\\
2.28567420131839	0.163757916000806\\
2.31424417227542	0.159040486721314\\
2.34281414323245	0.154294566478476\\
2.37138411418948	0.149513295450437\\
2.39995408514651	0.144689079975851\\
2.42852405610355	0.139813434194958\\
2.45709402706058	0.134876779474157\\
2.48566399801761	0.129868186691209\\
2.51423396897464	0.124775039838121\\
2.54280393993168	0.119582589140439\\
2.57137391088871	0.114273345550324\\
2.59994388184574	0.10882624160772\\
2.62851385280277	0.103215437878583\\
2.6570838237598	0.0974085727918642\\
2.68565379471684	0.0913641016760201\\
2.71422376567387	0.0850270694757905\\
2.7427937366309	0.0783220194248854\\
2.77136370758793	0.0711402406249339\\
2.79993367854497	0.0633146164154031\\
2.828503649502	0.054563148983625\\
2.85707362045903	0.0443342493780967\\
2.88564359141606	0.0311982528904233\\
2.91421356237309	0\\
};
\addlegendentry{Marcenko-Pastur distribution \eqref{eq:marcenko} with $c=0.5$}

\addplot[ybar interval, fill=mycolor1, fill opacity=0.6, draw=black, area legend] table[row sep=crcr] {%
x	y\\
0	0.521739130434783\\
0.23	0.782608695652174\\
0.46	0.608695652173913\\
0.69	0.434782608695652\\
0.92	0.434782608695652\\
1.15	0.347826086956522\\
1.38	0.260869565217391\\
1.61	0.260869565217391\\
1.84	0.173913043478261\\
2.07	0.260869565217391\\
2.3	0.173913043478261\\
2.53	0.0869565217391304\\
2.76	0.0869565217391304\\
};
\addlegendentry{Empirical distribution histogram}

\end{axis}
\end{tikzpicture}%
}
\hspace{0.2cm}
\subfloat[Tensor of size $20\times 10\times 20 \times 20$]{
%
%
\definecolor{mycolor1}{rgb}{0.00000,0.44700,0.74100}%
\begin{tikzpicture}
\scriptsize
\begin{axis}[%
width=2.821in,
height=2.066in,
at={(0.758in,0.481in)},
scale only axis,
xmin=0,
xmax=4,
ymin=0,
ymax=1,
axis background/.style={fill=white},
legend style={legend cell align=left, align=left, draw=white!15!black}
]
\addplot [color=black, line width=2.0pt]
  table[row sep=crcr]{%
0.0857864376269049	0\\
0.114356408583937	0.787249614004345\\
0.142926379540969	0.886233981353344\\
0.171496350498002	0.899914578825413\\
0.200066321455034	0.886089800403947\\
0.228636292412066	0.862310523570274\\
0.257206263369098	0.835209615446203\\
0.28577623432613	0.807563628378639\\
0.314346205283163	0.780579953704442\\
0.342916176240195	0.754769710483123\\
0.371486147197227	0.730318506716542\\
0.400056118154259	0.70725607786645\\
0.428626089111292	0.685538064347222\\
0.457196060068324	0.665086831168446\\
0.485766031025356	0.64581222976621\\
0.514336001982388	0.627622167030847\\
0.54290597293942	0.610427856302289\\
0.571475943896453	0.594146248735935\\
0.600045914853485	0.578700962657039\\
0.628615885810517	0.56402242083675\\
0.657185856767549	0.55004758371961\\
0.685755827724582	0.536719492084972\\
0.714325798681614	0.523986736242092\\
0.742895769638646	0.511802914939299\\
0.771465740595678	0.500126116777846\\
0.80003571155271	0.48891843977801\\
0.828605682509743	0.478145555150151\\
0.857175653466775	0.467776316063098\\
0.885745624423807	0.457782409438925\\
0.914315595380839	0.448138047473916\\
0.942885566337872	0.438819695070829\\
0.971455537294904	0.429805829304649\\
1.00002550825194	0.421076727220498\\
1.02859547920897	0.412614278554655\\
1.057165450166	0.404401820306589\\
1.08573542112303	0.396423990431936\\
1.11430539208007	0.388666598252097\\
1.1428753630371	0.381116509475446\\
1.17144533399413	0.373761543993941\\
1.20001530495116	0.366590384856883\\
1.22858527590819	0.35959249703209\\
1.25715524686523	0.352758054746392\\
1.28572521782226	0.346077876354729\\
1.31429518877929	0.3395433658232\\
1.34286515973632	0.333146460028757\\
1.37143513069335	0.32687958117931\\
1.40000510165039	0.320735593745024\\
1.42857507260742	0.314707765366383\\
1.45714504356445	0.30878973126895\\
1.48571501452148	0.302975461769912\\
1.51428498547852	0.297259232508929\\
1.54285495643555	0.291635597076261\\
1.57142492739258	0.286099361745763\\
1.59999489834961	0.28064556204965\\
1.62856486930664	0.275269440956623\\
1.65713484026368	0.269966428435554\\
1.68570481122071	0.264732122203648\\
1.71427478217774	0.259562269471403\\
1.74284475313477	0.254452749506659\\
1.77141472409181	0.249399556846966\\
1.79998469504884	0.244398784993147\\
1.82855466600587	0.239446610417464\\
1.8571246369629	0.234539276716774\\
1.88569460791993	0.229673078734347\\
1.91426457887697	0.224844346463015\\
1.942834549834	0.220049428526425\\
1.97140452079103	0.215284675013371\\
1.99997449174806	0.210546419411426\\
2.0285444627051	0.205830959348476\\
2.05711443366213	0.201134535802411\\
2.08568440461916	0.196453310376906\\
2.11425437557619	0.191783340161333\\
2.14282434653323	0.187120549589784\\
2.17139431749026	0.182460698581048\\
2.19996428844729	0.177799346068073\\
2.22853425940432	0.173131807798365\\
2.25710423036135	0.168453106986673\\
2.28567420131839	0.163757916000806\\
2.31424417227542	0.159040486721314\\
2.34281414323245	0.154294566478476\\
2.37138411418948	0.149513295450437\\
2.39995408514651	0.144689079975851\\
2.42852405610355	0.139813434194958\\
2.45709402706058	0.134876779474157\\
2.48566399801761	0.129868186691209\\
2.51423396897464	0.124775039838121\\
2.54280393993168	0.119582589140439\\
2.57137391088871	0.114273345550324\\
2.59994388184574	0.10882624160772\\
2.62851385280277	0.103215437878583\\
2.6570838237598	0.0974085727918642\\
2.68565379471684	0.0913641016760201\\
2.71422376567387	0.0850270694757905\\
2.7427937366309	0.0783220194248854\\
2.77136370758793	0.0711402406249339\\
2.79993367854497	0.0633146164154031\\
2.828503649502	0.054563148983625\\
2.85707362045903	0.0443342493780967\\
2.88564359141606	0.0311982528904233\\
2.91421356237309	0\\
};
\addlegendentry{Marcenko-Pastur distribution \eqref{eq:marcenko} with $c=0.5$}

\addplot[ybar interval, fill=mycolor1, fill opacity=0.6, draw=black, area legend] table[row sep=crcr] {%
x	y\\
0	0.230769230769231\\
0.13	0.884615384615385\\
0.26	0.769230769230769\\
0.39	0.692307692307692\\
0.52	0.576923076923077\\
0.65	0.5\\
0.78	0.461538461538462\\
0.91	0.461538461538462\\
1.04	0.384615384615385\\
1.17	0.384615384615384\\
1.3	0.307692307692307\\
1.43	0.307692307692308\\
1.56	0.26923076923077\\
1.69	0.269230769230769\\
1.82	0.230769230769231\\
1.95	0.192307692307692\\
2.08	0.192307692307692\\
2.21	0.153846153846154\\
2.34	0.153846153846153\\
2.47	0.076923076923077\\
2.6	0.115384615384615\\
2.73	0.0769230769230767\\
2.86	0.0769230769230767\\
};
\addlegendentry{Empirical distribution histogram}

\end{axis}
\end{tikzpicture}%
}
\caption{Empirical tensor singular values histogram and its asymptotic distribution}
\label{fig:marcenko}
\end{figure}

\section{Spiked random tensors}\label{sec7}
We present now asymptotic properties of estimators of a rank-1 tensor ``signal'' perturbed by additive noise through the analysis of Z-eigenvalues and Z-singular values in the context of symmetric and asymmetric tensors respectively.
The studies of spiked tensors have been only developed in the context of real tensors and their extension to complex tensors is still missing to the best of our knowledge.
Note that properties of tensor eigenvalues and singular values of the random tensor underlying the spiked model have been provided in \cite{benarous21}.

\subsection{Tensor Z-eigenvalues and Z- singular values}\label{sec:zeigen}
Let us first define Z-singular values, Z-eigenvalues, and their associated Z-singular vectors and Z-eigenvectors through the variational definition considered in \cite{limsingular}.
\begin{definition}
Given a vector $\myvec{x}\in \mathbb{R}^I$, we define $\myvec{x}^{P}$ as the iterative tensor outer product, i.e.
\begin{equation}
\myvec{x}^{P} = \underbrace{ \myvec{x}\circ \dots \circ \myvec{x}}_{P\text{\ times}}.
\end{equation}
where $\myvec{x}^{P}$ is an order $P$ tensor of size $I \times \dots \times I$.
\end{definition}

\begin{definition}
Given an order-$N$ symmetric tensor $\mathscr{A} \in \mathbb{R}^{I \times \dots \times I}$, the tuple $(\lambda, {\myvec{u}})\in\mathbb{R}^+\times \mathbb{S}_{\mathbb{R}}^{I}$ are coupled Z-eigenvalues and Z-eigenvectors\footnote{These eigenvalues are alternatively called $\ell_2$-eigenvalues in \cite{de2022random}.} if and only if
\begin{equation}\label{eq:defzeig}
\mathscr{A} *_{N-1} \myvec{u}^{N-1}= \lambda\myvec{u}
\end{equation}
where we noted $\mathbb{S}_{\mathbb{R}}^{I}$ the real unit sphere in $\mathbb{R}^I$.
\begin{definition}
Given an order-$N$ tensor $\mathscr{A} \in \mathbb{C}^{I_1 \times \dots \times I_N}$, the tuple $(\lambda, {\myvec{u}_1}, \dots,{\myvec{u}}_N)\in\mathbb{R}^+\times \mathbb{S}_{\mathbb{R}}^{I_1}\times\dots\times \mathbb{S}_{\mathbb{R}}^{I_N}$ contains coupled Z-singular values and Z-singular vectors if and only if
\begin{equation}\label{eq:defzsing}
\mathscr{A} *_{N} \left(\myvec{u}_1\circ\dots\circ \myvec{u}_{i-1}\circ \text{e}_{i,j} \circ \myvec{u}_{i+1}\circ\dots\circ\myvec{u}_{N}\right) = \lambda u_{i,j} \text{\ \quad for $i=1,\dots,N$}.
\end{equation}
where $u_{i,j}$ denotes the $j$-th element of $\myvec{u}_{i}$ and $\text{e}_{i,j}$ denotes the $I_i$-dimensional vector which elements are zeros except its $j$-th element equal to $1$.
\end{definition}
\end{definition}
Unfortunately, the asymptotic distributions of Z-eigenvalues and Z-singular values are not known as for the tensor eigenvalues and the tensor singular values.
We can however illustrate some of their properties in the study of tensor spiked models.

\subsection{Spiked symmetric random tensor: asymptotic properties}
Let us focus first on the symmetric spiked model. 
Consider an order-$N$ random symmetric tensor  $\pmb{\mathscr{W}}$ of size $I\times\dots\times I $  which follows a standard Gaussian distribution as specified in \eqref{realstandard}. Therefore its probability density is defined on the set of symmetric tensors through
$p(\pmb{\mathscr{W}}) \propto \exp(-\frac{1}{2}\|\pmb{\mathscr{W}}\|_F^2)$ where $\|.\|_F$ denotes the Frobenius norm.
Note that the considered symmetry means that its elements do not change under any permutation of its indices. Thus, the distinct entries of $\pmb{\mathscr{W}}$ do not have the same variance.
Indeed, similarly to \cite{de2022random}, we can expand $\|\pmb{\mathscr{W}}\|_F^2$ in the case $N=3$ as 
\begin{equation}
\|\pmb{\mathscr{W}}\|_F^2 = \sum_{i} \mathscr{W}_{iii}^2 + 3\sum_{i\neq j} \mathscr{W}_{iij}^2 + 6\sum_{i<j<k} \mathscr{W}_{ijk}^2.
\end{equation}
Consequently, the diagonal elements $\mathscr{W}_{iii}$ have variance $1$, the elements with two similar indices $\mathscr{W}_{iij}$ have variance $\frac13$ and  the elements with three distinct indices $\mathscr{W}_{ijk}$ have variance $\frac16$.

The rank-1 symmetric spiked model results in observing a rank-1 tensor perturbed by noise, i.e.
\begin{equation}\label{eq:symspike}
\pmb{\mathscr{A}} = \beta \myvec{x}^N + \frac{1}{\sqrt{I}}\pmb{\mathscr{W}},
\end{equation}
with $\beta$ a scalar parameter, $\myvec{x}$ uniformly spread on the unit-sphere and $\pmb{\mathscr{W}}$ a symmetric standard Gaussian tensor.
The main challenge regarding this model is then to recover the \textit{signal} parts $\beta$ and  $\myvec{x}$ while observing $\pmb{\mathscr{A}}$.
For this objective, let us define the $\ell_2$-loss associated with \eqref{eq:symspike} (which is proportional to the log-likelihood of the model) defined on $\mathbb{R}\times \mathbb{S}_{\mathbb{R}}^{I}$ as
\begin{equation}\label{eq:ls_symmetrical}
\ell_s(\lambda,\myvec{u}) := \|\pmb{\mathscr{A}} - \lambda \myvec{u}^N \|^2.
\end{equation}
Then, the Z-eigenvalues and Z-eigenvectors of $\pmb{\mathscr{A}}$ are the critical points of $\ell_s$. Indeed, the tuple $(\lambda, \myvec{u})$ for which the gradient of \eqref{eq:ls_symmetrical} is equal to zero satisfies \eqref{eq:defzeig}, as shown in \cite{limsingular}. We remark that a Z-eigenpair $(\lambda,\myvec{u})$ satisfies also $\pmb{\mathscr{A}}*_N\myvec{u}^{N}$.
Moreover, the Maximum Likelihood (ML) estimator is the global minimum of $\ell_s$. We can call therefore the ML estimator the dominating Z-eigenpair.

In \cite{jagannath20} the spin glass theory is employed in order to derive asymptotic properties (when $I\rightarrow\infty$) of the ML estimator exhibiting a so called Baik-Ben Arous-Péché (BBP) phase transition \cite{bbp05} depending on the \textit{power} level $\beta$.
Indeed, it is shown that there exists some $\beta_c$ such that for $\beta<\beta_c$ the correlation between the ML estimate $\myvec{u}_{ML}$ and $\myvec{x}$ is asymptotically zero when $I$ goes towards $\infty$ while for $\beta>\beta_c$, this correlation is positive.
Recently, Goulart et al. highlighted a different threshold $\beta_s<\beta_c$ (defined in an implicit manner) by considering in \cite{de2022random} the asymptotic behavior of the Z-eigenvectors (and not only the ML estimator). 
\begin{theorem}\cite[Th.3]{de2022random}\label{th:symmetrical_limit}
Let $\pmb{\mathscr{A}}$ be an order-$N$ symmetric tensor generated as in \eqref{eq:symspike} with $N\geq 3$. Fix $\lambda>0$ and suppose that $I\rightarrow \infty$. Assume  there exists a sequence\footnote{The sequence of Z-eigenvectors depends on $I$. We omit the dependency for the sake of brevity.} of Z-eigenvectors $\myvec{u}$ of $\pmb{\mathscr{A}}$ defined in \eqref{eq:symspike} such that
\begin{equation}
\begin{array}{l}
\myvec{x}^T\myvec{u} \xrightarrow{a.s} \alpha\\
\pmb{\mathscr{A}}*_N\myvec{u}^{N}\xrightarrow{a.s} \lambda
\end{array}
\end{equation}
for some $\alpha>0$ and some $\lambda$ such that $\lambda>(N-1)\beta_N$ with $\beta_N = \frac{2}{\sqrt{N(N-1)}}$. 
Then we have $\alpha = \omega_N(\lambda)$ with $\lambda$ being the solution of the fixed-point equation $\lambda = \phi_N(\lambda)$ where
\begin{equation}
\begin{array}{l}
\phi_N(z) = \beta\omega_N(z)^N - \frac{1}{N-1}m_{\beta_N}(\frac{z}{N-1}),\\
\omega_N(z) = \frac{1}{\beta^{\frac{1}{N-2}}}\Big(z + \frac{1}{N}m_{\beta_N}(\frac{z}{N-1})\Big)^{\frac{1}{N-2}},\\
m_{\beta}(z) = \frac{2}{\beta^2}\Big(-z+z\sqrt{1-\frac{\beta^2}{z^2}}\Big).
\end{array}
\end{equation}
\end{theorem}
\begin{figure}
\centering
%
%
\begin{tikzpicture}
\scriptsize
\begin{axis}[%
width=3.021in,
height=2.066in,
at={(0.758in,0.481in)},
scale only axis,
xmin=0,
xmax=5,
xlabel=$\beta$,
ylabel=Alignment,
ymin=-0.01,
ymax=1.03,
axis background/.style={fill=white},
legend style={at={(0.435,0.582)}, anchor=south west, legend cell align=left, align=left, draw=white!15!black}
]
\addplot [color=black, line width=2.0pt]
  table[row sep=crcr]{%
0	0\\
0.128188948708806	0\\
0.256377897417611	0\\
0.384566846126417	0\\
0.512755794835223	0\\
0.640944743544029	0\\
0.769133692252834	0\\
0.89732264096164	0\\
1.02551158967045	0\\
1.15370053837925	0\\
};
\addlegendentry{Limit alignment $\alpha$}

\addplot [color=black, line width=1.5pt, forget plot]
  table[row sep=crcr]{%
1.15470053837925	0.707106781186548\\
1.25329796047209	0.833301374415569\\
1.35189538256493	0.87179473033197\\
1.45049280465777	0.895879024210809\\
1.54909022675061	0.912855737107416\\
1.64768764884345	0.925569335569384\\
1.74628507093629	0.935462715466659\\
1.84488249302913	0.943373808745919\\
1.94347991512197	0.949831859506455\\
2.04207733721481	0.955191059672913\\
2.14067475930765	0.959698783355541\\
2.23927218140049	0.963533598224634\\
2.33786960349333	0.966827858122106\\
2.43646702558617	0.969681823118591\\
2.53506444767901	0.97217284623884\\
2.63366186977185	0.974361550800712\\
2.73225929186469	0.976296100389152\\
2.83085671395753	0.97801522053828\\
2.92945413605037	0.979550380844452\\
3.02805155814321	0.980927398942944\\
3.12664898023605	0.982167638098631\\
3.22524640232889	0.983288913908639\\
3.32384382442172	0.984306189402217\\
3.42244124651456	0.985232113970115\\
3.5210386686074	0.986077445520062\\
3.61963609070024	0.986851384275536\\
3.71823351279308	0.987561838992562\\
3.81683093488592	0.988215640969286\\
3.91542835697876	0.988818717355135\\
4.0140257790716	0.989376232460992\\
4.11262320116444	0.989892703713459\\
4.21122062325728	0.990372097369974\\
4.30981804535012	0.990817907968511\\
4.40841546744296	0.991233224621692\\
4.5070128895358	0.991620786606665\\
4.60561031162864	0.991983030196077\\
4.70420773372148	0.992322128283784\\
4.80280515581432	0.992640024053521\\
4.90140257790716	0.99293845969907\\
5	0.993219001015157\\
};
\addplot [color=red, line width=1.0pt, only marks, mark=+, mark options={solid, fill=red, red}]
  table[row sep=crcr]{%
0	0.219986923687999\\
0.128188948708806	0.171897733093228\\
0.256377897417611	0.0777173147821717\\
0.384566846126417	0.169811262615709\\
0.512755794835223	0.349909644078729\\
0.640944743544029	0.227728637370758\\
0.769133692252834	0.170677045899541\\
0.89732264096164	0.076597955416401\\
1.02551158967045	0.0624241118256037\\
1.15370053837925	0.866771990067387\\
1.15470053837925	0.0302878092287078\\
1.25329796047209	0.246032819754287\\
1.35189538256493	0.666413946413072\\
1.45049280465777	0.878365764748175\\
1.54909022675061	0.931518667249871\\
1.64768764884345	0.913950722014577\\
1.74628507093629	0.940537196737079\\
1.84488249302913	0.905306400206858\\
1.94347991512197	0.891012877011108\\
2.04207733721481	0.980012379132985\\
2.14067475930765	0.97458715321593\\
2.23927218140049	0.973985194696286\\
2.33786960349333	0.982338294478058\\
2.43646702558617	0.974047710189451\\
2.53506444767901	0.975079251713513\\
2.63366186977185	0.983467850054279\\
2.73225929186469	0.985736613064466\\
2.83085671395753	0.982079606435315\\
2.92945413605037	0.972486532407369\\
3.02805155814321	0.981566511329575\\
3.12664898023605	0.992437262569028\\
3.22524640232889	0.982749010786624\\
3.32384382442172	0.984678788658857\\
3.42244124651456	0.989227180313509\\
3.5210386686074	0.983054757847039\\
3.61963609070024	0.988288637784845\\
3.71823351279308	0.988769679457254\\
3.81683093488592	0.992313988700653\\
3.91542835697876	0.991769656124261\\
4.0140257790716	0.994292034533498\\
4.11262320116444	0.990353728202791\\
4.21122062325728	0.989649032078625\\
4.30981804535012	0.9895271674888\\
4.40841546744296	0.991478725395378\\
4.5070128895358	0.99176123243856\\
4.60561031162864	0.992706812592279\\
4.70420773372148	0.991390907994095\\
4.80280515581432	0.988533383049842\\
4.90140257790716	0.994804426014695\\
5	0.995528204013999\\
};
\addlegendentry{Empirical alignment $x^Tu$}

\addplot [color=black, dashed, line width=1.0pt, forget plot]
  table[row sep=crcr]{%
1.15470053837925	0\\
1.15470053837925	1.1\\
};
\end{axis}
\end{tikzpicture}%
\caption{Illustration of limit and empirical alignments for a symmetric tensor of size $20\times 20\times 20$ with respect to $\beta$.}
\label{fig:symmetrical}
\end{figure}
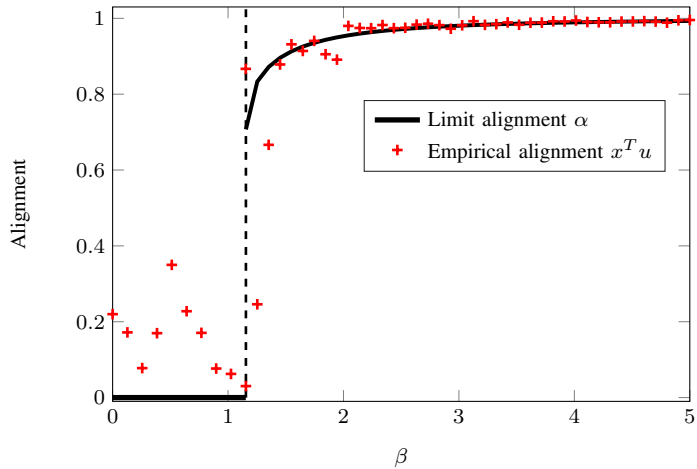

On the practical side, the classical way to numerically approximate a Z-eigenvector of $\pmb{\mathscr{A}}$ is to use the tensor power iteration initialized with some arbitrary vector $\myvec{x}_0$ and iterating until convergence \cite{delathauwer00}
\begin{equation}
\myvec{x}_{t+1} = \frac{\pmb{\mathscr{A}}*_{N-1}\myvec{x}_t^{N-1}}{\|\pmb{\mathscr{A}}*_{N-1}\myvec{x}_t^{N-1}\|}.
\end{equation}
We compare in Figure~\ref{fig:symmetrical} the empirical alignement $\myvec{x}^T\myvec{u}$ where $\myvec{u}$ is obtained with a tensor power iteration with the limit alignement $\alpha$ provided by Theorem~\ref{th:symmetrical_limit}. 
We observe in particular a phase transition at $\beta=\sqrt{\frac43}$.

\subsection{Spiked asymmetric random tensor: asymptotic properties}

Let us now focus on the asymmetric spiked model $\pmb{\mathscr{A}} \in \mathbb{R}^{I_1 \times \dots \times I_N}$ which results from observing a rank-1 asymmetric tensor perturbed by a Gaussian random tensor
\begin{equation}\label{eq:asymspike}
\pmb{\mathscr{A}} = \beta \myvec{x}_1\circ\dots\circ \myvec{x}_N + \frac{1}{\sqrt{\sum_{k=1}^N I_k}}\pmb{\mathscr{W}},
\end{equation}
with $\myvec{x}_1,\dots,\myvec{x}_N$ uniformly spread on the respective unit-spheres $\mathbb{S}_{\mathbb{R}}^{I_k}$ for $1\leq k\leq N$ and $\pmb{\mathscr{W}}$ a tensor with i.i.d standard Gaussian entries.
Note that the normalization by $\sqrt{\sum_{k=1}^N I_k}$ in \eqref{eq:asymspike} comes from the fact that the tensor spectral norm of $\pmb{\mathscr{W}}$ is upper bounded by $\sqrt{\sum_{k=1}^N I_k}$ (see \cite[Lemma~4]{seddik23}).
We recall that the spectral norm of $\pmb{\mathscr{A}}$ is the supremum of $\pmb{\mathscr{A}}*_N\left( \myvec{u}_1\circ\dots \circ \myvec{u}_N\right)$ for any $\myvec{u}_k\in\mathbb{S}_{\mathbb{R}}^{I_k}$ with $1\leq k\leq N$.
Similarly to the Z-eigenvalues and Z-eigenvectors in the symmetric case, we write the $\ell_2$-loss function associated with \eqref{eq:asymspike} (proportional to the log-likelihood of the model), defined on $\mathbb{R}\times \mathbb{S}_{\mathbb{R}}^{I_1}\times \dots\times \mathbb{S}_{\mathbb{R}}^{I_N}$, as
\begin{equation}\label{eq:ls_asymmetrical}
\ell_a(\lambda,\myvec{u}_1,\dots,\myvec{u}_N) = \|\pmb{\mathscr{A}} - \lambda \myvec{u}_1\circ\dots\circ \myvec{u}_N \|^2.
\end{equation}
Then, the Z-singular values and  Z-singular vectors of $\pmb{\mathscr{A}}$ are the critical points of $\ell_a$. Indeed, the tuple $(\lambda, \myvec{u}_1,\dots, \myvec{u}_N)$ for which the gradient of \eqref{eq:ls_asymmetrical} is equal to zero satisfies \eqref{eq:defzsing}, as shown in \cite{limsingular}.
Moreover, the Maximum Likelihood (ML) estimator is the global minimum of $\ell_a$. Hence, we can call the ML estimator the dominating Z-singular value and Z-singular vector. Recently, the asymptotic properties of sequences of Z-singular vectors associated with the rank-1 signal were provided by using similar RMT tools as in the symmetric case.
The result is presented in the following theorem.
\begin{theorem}\cite[Th.8]{seddik23}
Let $\pmb{\mathscr{A}}$ be an order-$N$ tensor generated as in \eqref{eq:asymspike} with $N\geq 3$.
Suppose that $I_1,\dots,I_N\rightarrow\infty$ such that $\frac{I_k}{\sum_{\ell=1}^N I_\ell}\rightarrow c_k\in (0,1)$ and suppose that there exists a sequence of critical points $\myvec{u}_1,\dots,\myvec{u}_N$ of \eqref{eq:ls_asymmetrical} such that
\begin{equation}
\begin{array}{l}
\myvec{x}_k^T\myvec{u}_k \xrightarrow{a.s} \alpha_k \text{\quad for $1\leq k\leq N$}\\
\pmb{\mathscr{A}}*_N\left( \myvec{u}_1\circ\dots \circ \myvec{u}_N\right)\xrightarrow{a.s} \lambda
\end{array}
\end{equation}
for some $\alpha_1,\dots,\alpha_N>0$ and some $\lambda>0$.
Define $(g_k(z))_{1\leq k\leq N}$ as the solution of the set of equations $g_k(z)^2- \Big(\sum_{j=1}^N g_j(z) +z\Big)g_k(z) - c_k=0 $ for $1\leq k\leq N$ and $\mathcal{S}$ the set of $z$ such that $(g_k(z))_{1\leq k\leq N}$ exist.
Assume that $\lambda\in\mathcal{S}$.

Then, we have $\alpha_k=q_k(\lambda)$ for $1\leq k\leq N$ where $\lambda$ is the solution of $\lambda + \sum_{k=1}^N g_k(\lambda) - \beta\prod_{k=1}^N q_k(\lambda) = 0$ with
\begin{equation}
q_k(z) = \sqrt{1-\frac{g_k(z)^2}{c_k}} \text{\quad for $1\leq k\leq N$}.
\end{equation}
\end{theorem}

\section{Conclusions}\label{sec8}
Starting with a brief review of tensor algebra, this paper considered the subject of  random tensors in an attempt to provide a centralized reference for this topic. The complete probabilistic characterization of complex random tensors was presented in terms of the joint distribution of tensor components. For this purpose we considered the composite real representation of complex tensors, as well as the augmented complex representation involving the tensor itself and its complex conjugate.The second order characterization of order $N$ complex tensors, both proper and improper, was presented in terms of covariance and pseudo-covariance being expressed as order $2N$ tensors. In particular, we presented various structures where correlation can span across a subset of modes thereby studying the separability of correlation across some or all modes. This study culminated with the presentation of the Kronecker tensor correlation model which is an extension of the corresponding matrix model commonly used in the area of Multi-Input Multi-Output (MIMO) wireless communication systems. Next  complex Gaussian tensors were presented where the joint PDF function of the components was expressed using the Einstein product while accounting for both covariance and pseudo-covariance tensors. Furthermore, we considered random tensor processes, and their transformation by a tensor multi-linear time invariant system. The eigenvalues and singular values of Hermitian tensors were considered, along with their asymptotic probabilistic characterization in terms of the semi-circular distribution for eigenvalues and Marcenko-Pastur distribution for singular values, using tools from Random Matrix Theory.  Finally, the spiked symmetric and asymmetric real Gaussian tensor models were presented along with corresponding asymptotics. The phase transition phenomenon for the power level of such models was exposed. Such properties so far have been explored only for real tensors, and their extension to the complex case is a topic for future research.

\appendix 
\section*{Proof of Lemma \ref{lemma1}}
We prove Lemma \ref{lemma1} using the characteristic function of Gaussian tensors.

Consider the channel tensor $\pmb{\mathscr{H}} \in \mathbb{C}^{J_1 \times \dots \times J_M \times I_1 \times \dots \times I_N }$ whose entries are zero mean jointly complex Gaussian. If the correlation tensor of the channel is given by $\mathscr{R} \in \mathbb{C}^{J_1 \times \dots \times J_M \times I_1 \times \dots \times I_N \times J_1 \times \dots \times J_M \times I_1 \times \dots \times I_N}$, the characteristic function of such a tensor channel is given as \cite{MMSEJournal}:
\begin{equation}\label{channelCFtensor}
\Phi_{\pmb{\mathscr{H}}}(\mathscr{W}) = \exp \bigg\{ - \dfrac{1}{4} \Big(\mathscr{W}^* *_{N+M} \mathscr{R} *_{N+M} \mathscr{W} \Big) \bigg\}
\end{equation}
where the tensor $\mathscr{W}$ is of same size as $\pmb{\mathscr{H}}$. 

Finding joint second order moments using the characteristic function relies on complex partial derivatives, which we define next. The complex derivative of a scalar function $f$ is defined as \cite{VerduGrad}: 
\begin{align}
\dfrac{\partial f}{\partial x^*} &= \dfrac{1}{2}\Big( \dfrac{\partial f}{\partial \Re(x)} + j \dfrac{\partial f}{\partial \Im(x)} \Big) \\
\dfrac{\partial f}{\partial x} &= \dfrac{1}{2}\Big( \dfrac{\partial f}{\partial \Re(x)} - j \dfrac{\partial f}{\partial \Im(x)} \Big)
\end{align}
Thus, we have $\partial x/\partial x =1$ and $\partial x/\partial x^* =\partial x^*/\partial x =0$ \cite{HaykinAdaptiveFilterTheory}. Extending the definition of complex gradient vector from \cite{HaykinAdaptiveFilterTheory} to tensors, we define the first derivative of a scalar function $f$ with respect to an order-$N$ complex tensor $\mathscr{X} \in \mathbb{C}^{I_1 \times \dots \times I_N}$ as:
\begin{equation}\label{firstordeDer}
[\nabla_{\mathscr{X}} f]_{i_1,\dots,i_N} = \dfrac{\partial f}{\partial \Re(\mathscr{X}_{i_1,\dots,i_N})} + j \dfrac{\partial f}{\partial \Im(\mathscr{X}_{i_1,\dots,i_N})}  =2 \cdot \dfrac{\partial f}{\partial \mathscr{X}^*_{i_1,\dots,i_N}}.
\end{equation}
Note that $\nabla_{\mathscr{X}} f$ is an order-$N$ tensor where its each element is specified by \eqref{firstordeDer}.  Further, the second derivative of the function $f$ with respect to the order-$N$ tensor $\mathscr{X}$ is described using an order $2N$ tensor where each element is derived by taking the derivative of \eqref{firstordeDer} with respect to $\mathscr{X}_{i_1',\dots,i_N'}$  as:
\begin{equation}\label{secondorderDer}
[\nabla^2_{\mathscr{X},\mathscr{X}^*} f]_{i_1,\dots,i_N,i_1',\dots,i_N'} = 4 \cdot \dfrac{\partial^2 f}{\partial \mathscr{X}^*_{i_1,\dots,i_N} \partial \mathscr{X}_{i_1',\dots,i_N'}}
\end{equation}

Second derivative of characteristic functions can be used to find the joint moments. Thus the joint moment between two elements of the channel tensor can be found using the second derivative of the channel's characteristic function $\Phi_{\pmb{\mathscr{H}}}(\mathscr{W})$ being evaluated at an all zero tensor, $0_{\mathscr{T}}$. Hence, the joint moment is written as:
\begin{align}\label{charfunjwot}
\mathbb{E}[\pmb{\mathscr{H}}_{j_1,\dots,j_M,i_1,\dots,i_N} \cdot \pmb{\mathscr{H}}_{j_1',\dots,j_M',i_1',\dots,i_N'}^* ] = - [\nabla^2_{\mathscr{W}, \mathscr{W}^*} \Phi_{\pmb{\mathscr{H}}}]_{j_1,\dots,j_M,i_1,\dots,i_N,j_1',\dots,j_M',i_1',\dots,i_N'} \Bigg\rvert_{\mathscr{W}=0_{\mathscr{T}}}.
\end{align} 

For simplicity of notation, let $\theta = \Big(\mathscr{W}^* *_{N+M} \mathscr{R} *_{N+M} \mathscr{W} \Big)$, then \eqref{channelCFtensor} becomes $\Phi_{\pmb{\mathscr{H}}}(\mathscr{W}) = \exp(-\theta/4)$. Upon expanding the Einstein product, $\theta $ can be written as:
\begin{equation}\label{thetaeqw}
\theta = \sum_{j_1,\dots,j_M,i_1,\dots,i_N,j_1',\dots,j_M',i_1',\dots,i_N'} \mathscr{W}^*_{j_1,\dots,j_M,i_1,\dots,i_N} \mathscr{R}_{j_1,\dots,j_M,i_1,\dots,i_N,j_1',\dots,j_M',i_1',\dots,i_N'} \mathscr{W}_{j_1',\dots,j_M',i_1',\dots,i_N'}
\end{equation}
On substituting the tensor $\mathscr{R}$ from \eqref{SepcortenTh} into \eqref{thetaeqw}, we get:
\begin{equation} 
\theta=\sum_{j_1,\dots,j_M,i_1,\dots,i_N,j_1',\dots,j_M',i_1',\dots,i_N'} \mathscr{W}^*_{j_1,\dots,j_M,i_1,\dots,i_N} (\mathscr{G}_{R})_{j_1,\dots,j_M,j_1',\dots,j_M'} \cdot (\mathscr{G}_{T})_{ i_1,\dots,i_N,i_1',\dots,i_N'} \mathscr{W}_{j_1',\dots,j_M',i_1',\dots,i_N'} \label{Theta2eq}
\end{equation}
Using \eqref{firstordeDer}, the first derivative of $\Phi_{\pmb{\mathscr{H}}}(\mathscr{W})$ is given as:
\begin{equation}\label{nambWerh}
[\nabla_{\mathscr{W}^*} \Phi_{\pmb{\mathscr{H}}}]_{j_1',\dots,j_M',i_1',\dots,i_N'} = 2 \cdot \dfrac{\partial \Phi_{\pmb{\mathscr{H}}}}{\partial \mathscr{W}_{j_1',\dots,j_M',i_1',\dots,i_N'} }.
\end{equation}
Since $\Phi_{\pmb{\mathscr{H}}}(\mathscr{W})= \exp(-\theta/4)$, using chain rule of derivatives in \eqref{nambWerh} we get:
\begin{equation}\label{nablWHji}
[\nabla_{\mathscr{W}^*} \Phi_{\pmb{\mathscr{H}}}]_{j_1',\dots,j_M',i_1',\dots,i_N'} =-\dfrac{1}{2} \exp(-\theta/4) \dfrac{\partial \theta}{ \partial \mathscr{W}_{j_1',\dots,j_M',i_1',\dots,i_N'}}= -\dfrac{1}{2} \Phi_{\mathscr{H}}\dfrac{\partial \theta}{ \partial \mathscr{W}_{j_1',\dots,j_M',i_1',\dots,i_N'}} 
\end{equation}
Further, using \eqref{secondorderDer} the second derivative is given as:
\begin{equation}\label{nablWWconj}
[\nabla^2_{\mathscr{W}, \mathscr{W}^*} \Phi_{\pmb{\mathscr{H}}}]_{j_1,\dots,j_M,i_1,\dots,i_N,j_1',\dots,j_M',i_1',\dots,i_N'} = 4 \cdot \dfrac{\partial ^2 \Phi_{\pmb{\mathscr{H}}}(\mathscr{W})}{\partial \mathscr{W}_{j_1,\dots,j_M,i_1,\dots,i_N}^* \partial \mathscr{W}_{j_1',\dots,j_M',i_1',\dots,i_N'} } 
\end{equation}
We can further expand \eqref{nablWWconj} by applying product rule of derivatives on \eqref{nablWHji} to get: 
\begin{align}
[\nabla^2_{\mathscr{W}, \mathscr{W}^*} \Phi_{\pmb{\mathscr{H}}}]_{j_1,\dots,j_M,i_1,\dots,i_N,j_1',\dots,j_M',i_1',\dots,i_N'} &= \underbrace{ \Phi_{\mathscr{H}}\Big( \dfrac{\partial \theta}  {4 \cdot \partial \mathscr{W}_{j_1',\dots,j_M',i_1',\dots,i_N'}} \Big)\cdot \Big( \dfrac{\partial \theta}  {4 \cdot \partial \mathscr{W}^*_{j_1',\dots,j_M',i_1',\dots,i_N'}} \Big)}_{\zeta} \nonumber \\ &- \Phi_{\mathscr{H}} \dfrac{\partial^2 \theta}{\partial \mathscr{W}_{j_1,\dots,j_M,i_1,\dots,i_N}^* \partial \mathscr{W}_{j_1',\dots,j_M',i_1',\dots,i_N'}} \label{ParDev2nd}
\end{align}
Also, on taking the double derivative of \eqref{Theta2eq} we get: 
\begin{equation}
\dfrac{\partial^2 \theta}{\partial \mathscr{W}_{j_1,\dots,j_M,i_1,\dots,i_N}^* \partial \mathscr{W}_{j_1',\dots,j_M',i_1',\dots,i_N'}} = (\mathscr{G}_{R})_{j_1,\dots,j_M,j_1',\dots,j_M'} \cdot (\mathscr{G}_{T})_{ i_1,\dots,i_N,i_1',\dots,i_N'}  
\end{equation}
where $\mathscr{G}_R$ and $\mathscr{G}_T$ are as defined in \eqref{SepcortenTh}. Also $\zeta$ in \eqref{ParDev2nd} when evaluated at $\mathscr{W}=0_{\mathscr{T}}$ gives $0$, and $\Phi_{\pmb{\mathscr{H}}}$ when evaluated at $\mathscr{W}=0_{\mathscr{T}}$ gives $1$. Hence \eqref{ParDev2nd} when evaluated at  $\mathscr{W}=0_{\mathscr{T}}$ gives
\begin{align}\label{GrGtresNabla}
[\nabla^2_{\mathscr{W}, \mathscr{W}^*} \Phi_{\pmb{\mathscr{H}}}]_{j_1,\dots,j_M,i_1,\dots,i_N,j_1',\dots,j_M',i_1',\dots,i_N'} \Big\rvert_{\mathscr{W}=0_{\mathscr{T}}} =  - (\mathscr{G}_{R})_{j_1,\dots,j_M,j_1',\dots,j_M'} \cdot (\mathscr{G}_{T})_{ i_1,\dots,i_N,i_1',\dots,i_N'}  
\end{align}
Based on the definition of the transmit correlation tensor from \eqref{G_T_3}, the pseudo-diagonal elements of the transmit correlation tensor for any receive element would be
\begin{equation}
(\mathscr{G}_{T}^{(\bar{j}_1,\dots,\bar{j}_M)})_{i_1,\dots,i_N,i_1,\dots,i_N} = \mathbb{E}[|\pmb{\mathscr{H}}_{\bar{j}_1,\dots,\bar{j}_M,i_1,\dots,i_N}|^2 ] = 1
\end{equation} 
Similarly the pseudo-diagonal elements of the receive correlation tensor from \eqref{G_R_3} for any transmit element would be
\begin{equation}
(\mathscr{G}_{R}^{(\bar{i}_1,\dots,\bar{i}_N)})_{j_1,\dots,j_M,j_1,\dots,j_M} = \mathbb{E}[|\pmb{\mathscr{H}}_{j_1,\dots,j_M,\bar{i}_1,\dots,\bar{i}_N}|^2 ] = 1
\end{equation}
Hence the pseudo-diagonal elements of both the transmit and receive correlation tensors do not depend on $(\bar{i}_1,\dots,\bar{i}_N)$ and $(\bar{j}_1,\dots,\bar{j}_M)$ respectively. Now we look at the off pseudo-diagonal elements. For the transmit correlation tensor we have
\begin{align}\label{gtdots}
(\mathscr{G}_{T}^{(\bar{j}_1,\dots,\bar{j}_M)})_{i_1,\dots,i_N,i_1',\dots,i_N'} &= \mathbb{E}[\pmb{\mathscr{H}}_{\bar{j}_1,\dots,\bar{j}_M,i_1,\dots,i_N} \cdot \pmb{\mathscr{H}}_{\bar{j}_1,\dots,\bar{j}_M,i_1',\dots,i_N'}^* ]. 
\end{align}
From \eqref{charfunjwot} and \eqref{gtdots}, we get: 
\begin{align}\label{GTgotsjM}
(\mathscr{G}_{T}^{(\bar{j}_1,\dots,\bar{j}_M)})_{i_1,\dots,i_N,i_1',\dots,i_N'} =  - [\nabla^2_{\mathscr{W}, \mathscr{W}^*} \Phi_{\pmb{\mathscr{H}}}]_{\bar{j}_1,\dots,\bar{j}_M,i_1,\dots,i_N,\bar{j}_1,\dots,\bar{j}_M,i_1',\dots,i_N'} \Big\rvert_{\mathscr{W}=0_{\mathscr{T}}}. 
\end{align}
From \eqref{GrGtresNabla} and \eqref{GTgotsjM}, we get: 
\begin{align}
(\mathscr{G}_{T}^{(\bar{j}_1,\dots,\bar{j}_M)})_{i_1,\dots,i_N,i_1',\dots,i_N'} = \underbrace{(\mathscr{G}_{R})_{\bar{j}_1,\dots,\bar{j}_M,\bar{j}_1,\dots,\bar{j}_M}}_{\text{pseudo-diagonal elements}} \cdot (\mathscr{G}_{T})_{ i_1,\dots,i_N,i_1',\dots,i_N'},
\end{align}
where the pseudo-diagonal elements of $\mathscr{G}_R$ are 1 for any $(\bar{j}_1,\dots,\bar{j}_M)$. Hence we see that the elements of transmit correlation tensor do not depend on the receive element indices $(\bar{j}_1,\dots,\bar{j}_M)$  i.e. the transmit correlation tensor is uniform across all the receive elements. Now we look at the off pseudo-diagonal elements of the receive correlation tensor, where we have
\begin{align}\label{Grbari}
(\mathscr{G}_{R}^{(\bar{i}_1,\dots,\bar{i}_N)})_{j_1,\dots,j_M,j_1',\dots,j_M'} = \mathbb{E}[\pmb{\mathscr{H}}_{j_1,\dots,j_M,\bar{i}_1,\dots,\bar{i}_N} \cdot \pmb{\mathscr{H}}_{j_1',\dots,j_M',\bar{i}_1,\dots,\bar{i}_N}^* ]
\end{align}
From \eqref{charfunjwot} and \eqref{Grbari}, we get
\begin{align}\label{GRbari2}
(\mathscr{G}_{R}^{(\bar{i}_1,\dots,\bar{i}_N)})_{j_1,\dots,j_M,j_1',\dots,j_M'} =  - [\nabla^2_{\mathscr{W}, \mathscr{W}^*} \Phi_{\pmb{\mathscr{H}}}]_{j_1,\dots,j_M,\bar{i}_1,\dots,\bar{i}_N,j_1',\dots,j_M',\bar{i}_1,\dots,\bar{i}_N} \Big\rvert_{\mathscr{W}=0_{\mathscr{T}}}.
\end{align}
Further, from \eqref{GrGtresNabla} and \eqref{GRbari2}, we get
\begin{align}
(\mathscr{G}_{R}^{(\bar{i}_1,\dots,\bar{i}_N)})_{j_1,\dots,j_M,j_1',\dots,j_M'} = (\mathscr{G}_{R})_{j_1,\dots,j_M,j_1',\dots,j_M'} \cdot \underbrace{(\mathscr{G}_{T})_{ \bar{i}_1,\dots,\bar{i}_N,\bar{i}_1,\dots,\bar{i}_N}}_{\text{pseudo-diagonal elements}},
\end{align}
where the pseudo-diagonal elements of $\mathscr{G}_T$ are 1 for any $(\bar{i}_1,\dots,\bar{i}_N)$. Hence we see that the receive correlation tensor elements do not depend on the transmit element indices $(\bar{i}_1,\dots,\bar{i}_N)$  i.e. the receive correlation tensor is uniform across all the transmit elements.   

\end{spacing}

\bibliography{AllRef}

\begin{thebibliography}{10}
\providecommand{\url}[1]{#1}
\csname url@samestyle\endcsname
\providecommand{\newblock}{\relax}
\providecommand{\bibinfo}[2]{#2}
\providecommand{\BIBentrySTDinterwordspacing}{\spaceskip=0pt\relax}
\providecommand{\BIBentryALTinterwordstretchfactor}{4}
\providecommand{\BIBentryALTinterwordspacing}{\spaceskip=\fontdimen2\font plus
\BIBentryALTinterwordstretchfactor\fontdimen3\font minus
  \fontdimen4\font\relax}
\providecommand{\BIBforeignlanguage}[2]{{%
\expandafter\ifx\csname l@#1\endcsname\relax
\typeout{** WARNING: IEEEtran.bst: No hyphenation pattern has been}%
\typeout{** loaded for the language `#1'. Using the pattern for}%
\typeout{** the default language instead.}%
\else
\language=\csname l@#1\endcsname
\fi
#2}}
\providecommand{\BIBdecl}{\relax}
\BIBdecl

\bibitem{PierreComon}
P.~Comon, ``{Tensors : A brief introduction},'' \emph{IEEE Signal Processing
  Magazine}, vol.~31, no.~3, pp. 44--53, May 2014.

\bibitem{tucker64extension}
L.~R. Tucker, ``{T}he extension of factor analysis to three-dimensional
  matrices,'' in \emph{{Contributions to mathematical psychology.}},
  H.~Gulliksen and N.~Frederiksen, Eds.\hskip 1em plus 0.5em minus 0.4em\relax
  New York: Holt, Rinehart and Winston, 1964, pp. 110--127.

\bibitem{bro2006review}
R.~Bro, ``{Review on multiway analysis in Chemistry—2000--2005},''
  \emph{Critical reviews in analytical chemistry}, vol.~36, no. 3-4, pp.
  279--293, 2006.

\bibitem{NikosCDMA}
N.~D. Sidiropoulos, G.~B. Giannakis, and R.~Bro, ``{Blind PARAFAC receivers for
  {DS-CDMA} systems},'' \emph{IEEE Transactions on Signal Processing}, vol.~48,
  no.~3, pp. 810--823, 2000.

\bibitem{Cichocki}
A.~{Cichocki}, D.~{Mandic}, L.~{De Lathauwer}, G.~{Zhou}, Q.~{Zhao},
  C.~{Caiafa}, and H.~A. {PHAN}, ``{Tensor Decompositions for Signal Processing
  Applications: From two-way to multiway component analysis},'' \emph{IEEE
  Signal Processing Magazine}, vol.~32, no.~2, pp. 145--163, 2015.

\bibitem{MMSEJournal}
D.~Pandey and H.~Leib, ``{A Tensor Framework for Multi-Linear Complex MMSE
  Estimation},'' \emph{IEEE Open Journal of Signal Processing}, vol.~2, pp.
  336--358, 2021.

\bibitem{cichocki2014era}
A.~Cichocki, ``{Era of big data processing: A new approach via tensor networks
  and tensor decompositions},'' \emph{arXiv preprint arXiv:1403.2048}, 2014.

\bibitem{KoldaTensor}
\BIBentryALTinterwordspacing
T.~G. Kolda and B.~W. Bader, ``{Tensor Decompositions and Applications},''
  \emph{SIAM Review}, vol.~51, no.~3, pp. 455--500, 2009. [Online]. Available:
  \url{http://dx.doi.org/10.1137/07070111X}
\BIBentrySTDinterwordspacing

\bibitem{BaderTensor}
B.~W. Bader and T.~G. Kolda, ``{Algorithm 862: {MATLAB} Tensor Classes for Fast
  Algorithm Prototyping},'' \emph{ACM Transactions on Mathematical Software},
  vol.~32, no.~4, pp. 635--653, December 2006.

\bibitem{NikosTensor}
N.~D. Sidiropoulos, L.~D. Lathauwer, X.~Fu, K.~Huang, E.~E. Papalexakis, and
  C.~Faloutsos, ``{Tensor Decomposition for Signal Processing and Machine
  Learning},'' \emph{IEEE Transactions on Signal Processing}, vol.~65, no.~13,
  pp. 3551--3582, July 2017.

\bibitem{TensorBook2020}
I.~Kisil, G.~G. Calvi, B.~S. Dees, and D.~P. Mandic, ``Tensor decompositions
  and practical applications: A hands-on tutorial,'' in \emph{Recent Trends in
  Learning From Data}.\hskip 1em plus 0.5em minus 0.4em\relax Springer, 2020,
  pp. 69--97.

\bibitem{TensorTut2021}
H.~Chen, F.~Ahmad, S.~Vorobyov, and F.~Porikli, ``{Tensor Decompositions in
  Wireless Communications and MIMO Radar},'' \emph{IEEE Journal of Selected
  Topics in Signal Processing}, vol.~15, no.~3, pp. 438--453, 2021.

\bibitem{pandey2023linear}
D.~Pandey, A.~Venugopal, and H.~Leib, ``Linear to multi-linear algebra and
  systems using tensors,'' \emph{Frontiers in Applied Mathematics and
  Statistics}, vol.~9, p. 1259836, 2024.

\bibitem{guruau2017random}
R.~G. Gur{\u{a}}u, \emph{Random tensors}.\hskip 1em plus 0.5em minus
  0.4em\relax Oxford University Press, 2017.

\bibitem{PandeyIMMSE}
D.~Pandey and H.~Leib, ``{Capacity Performance of Tensor Multi-Domain
  Communication Systems With Discrete Signalling Constellations},'' \emph{IEEE
  Open Journal of the Communications Society}, vol.~4, pp. 534--551, 2023.

\bibitem{panagakis2021tensor}
Y.~Panagakis, J.~Kossaifi, G.~G. Chrysos, J.~Oldfield, M.~A. Nicolaou,
  A.~Anandkumar, and S.~Zafeiriou, ``Tensor methods in computer vision and deep
  learning,'' \emph{Proceedings of the IEEE}, vol. 109, no.~5, pp. 863--890,
  2021.

\bibitem{muti2008lower}
D.~Muti, S.~Bourennane, and J.~Marot, ``Lower-rank tensor approximation and
  multiway filtering,'' \emph{SIAM journal on Matrix Analysis and
  Applications}, vol.~30, no.~3, pp. 1172--1204, 2008.

\bibitem{song2019tensor}
Q.~Song, H.~Ge, J.~Caverlee, and X.~Hu, ``Tensor completion algorithms in big
  data analytics,'' \emph{ACM Transactions on Knowledge Discovery from Data
  (TKDD)}, vol.~13, no.~1, pp. 1--48, 2019.

\bibitem{MLTI2}
C.~Chen, A.~Surana, A.~M. Bloch, and I.~Rajapakse, ``{Multilinear Control
  Systems Theory},'' \emph{SIAM Journal on Control and Optimization}, vol.~59,
  no.~1, pp. 749--776, 2021.

\bibitem{zhou2016linked}
G.~Zhou, Q.~Zhao, Y.~Zhang, T.~Adal{\i}, S.~Xie, and A.~Cichocki, ``Linked
  component analysis from matrices to high-order tensors: Applications to
  biomedical data,'' \emph{Proceedings of the IEEE}, vol. 104, no.~2, pp.
  310--331, 2016.

\bibitem{TamonTensorInversion}
\BIBentryALTinterwordspacing
M.~Brazell, N.~Li, C.~Navasca, and C.~Tamon, ``{Solving Multilinear Systems via
  Tensor Inversion},'' \emph{SIAM Journal on Matrix Analysis and Applications},
  vol.~34, no.~2, pp. 542--570, 2013. [Online]. Available:
  \url{https://doi.org/10.1137/100804577}
\BIBentrySTDinterwordspacing

\bibitem{LuEVD}
\BIBentryALTinterwordspacing
L.-B. Cui, C.~Chen, W.~Li, and M.~K. Ng, ``{An Eigenvalue problem for even
  order tensors with its applications},'' \emph{Linear and Multilinear
  Algebra}, vol.~64, no.~4, pp. 602--621, 2016. [Online]. Available:
  \url{https://doi.org/10.1080/03081087.2015.1071311}
\BIBentrySTDinterwordspacing

\bibitem{RandomMatrixForWireless}
R.~Couillet and M.~Debbah, \emph{{Random Matrix Methods for Wireless
  Communications}}.\hskip 1em plus 0.5em minus 0.4em\relax Cambridge University
  Press, 2011.

\bibitem{Pennington17}
J.~Pennington and Y.~Bahri, ``{Geometry of Neural Network Loss Surfaces via
  Random Matrix Theory},'' in \emph{ICML'17: Proceedings of the 34th
  International Conference on Machine Learning}, vol.~70, 2017, pp. 2798--2806.

\bibitem{soize05}
C.~Soize, ``Random matrix theory for modeling uncertainties in computational
  mechanics,'' \emph{Computer Methods in Applied Mechanics and Engineering},
  vol. 194, pp. 1333--1366, 2005.

\bibitem{limsingular}
L.-H. Lim, ``{Singular values and Eigenvalues of tensors: A variational
  approach},'' in \emph{1st IEEE International Workshop on Computational
  Advances in Multi-Sensor Adaptive Processing, 2005.}\hskip 1em plus 0.5em
  minus 0.4em\relax IEEE, 2005, pp. 129--132.

\bibitem{benarous21}
G.~Ben~Arous, D.~Z. Huang, and J.~Huang, ``{Long Random Matrices and Tensor
  Unfolding},'' \emph{arXiv preprint arXiv:2110.10210}, 2021.

\bibitem{montanari14}
A.~Montanari and E.~Richard, ``{A statistical model for tensor PCA},'' in
  \emph{Advances in Neural Information Processing Systems (NIPS)}, 2014, pp.
  2897--2905.

\bibitem{perry20}
A.~Perry, A.~Wein, and A.~Bandeira, ``{Statistical limits of spiked tensor
  models},'' \emph{Annales de l’Institut Henri Poincar\'e}, vol.~56, pp.
  230--264, 2020.

\bibitem{jagannath20}
A.~Jagannath, P.~Lopatto, and L.~Miolane, ``{Statistical thresholds for Tensor
  PCA},'' \emph{The Annals of Applied Probability}, vol.~30, no.~4, pp.
  1910--1933, 2020.

\bibitem{Chen21}
W.-K. Chen, M.~Handschy, and G.~Lerman, ``{Phase transition in random tensors
  with multiple independent spikes},'' \emph{The Annals of Applied
  Probability}, vol.~31, no.~4, pp. 1868--1913, 2021.

\bibitem{de2022random}
J.~H. de~Morais~Goulart, R.~Couillet, and P.~Comon, ``A random matrix
  perspective on random tensors,'' \emph{The Journal of Machine Learning
  Research}, vol.~23, no.~1, pp. 12\,110--12\,145, 2022.

\bibitem{seddik23}
M.~Seddik, M.~Guillaud, and R.~Couillet, ``{When Random Tensors meet Random
  Matrices},'' \emph{Annals of Applied Probability}, 2023, in press, preprint
  arXiv:2112.12348.

\bibitem{TensorDet}
\BIBentryALTinterwordspacing
M.~lin Liang, B.~Zheng, and R.~juan Zhao, ``{Tensor inversion and its
  application to the tensor equations with Einstein product},'' \emph{Linear
  and Multilinear Algebra}, vol.~67, no.~4, pp. 843--870, 2019. [Online].
  Available: \url{https://doi.org/10.1080/03081087.2018.1500993}
\BIBentrySTDinterwordspacing

\bibitem{pan2014tensor}
R.~Pan, ``Tensor transpose and its properties,'' \emph{arXiv preprint
  arXiv:1411.1503}, 2014.

\bibitem{Tucker1966}
\BIBentryALTinterwordspacing
L.~R. Tucker, ``Some mathematical notes on three-mode factor analysis,''
  \emph{Psychometrika}, vol.~31, no.~3, pp. 279--311, Sep 1966. [Online].
  Available: \url{https://doi.org/10.1007/BF02289464}
\BIBentrySTDinterwordspacing

\bibitem{LathauwerSVD}
\BIBentryALTinterwordspacing
L.~D. Lathauwer, B.~D. Moor, and J.~Vandewalle, ``{A Multilinear Singular Value
  Decomposition},'' \emph{SIAM J. Matrix Anal. Appl.}, vol.~21, no.~4, pp.
  1253--1278, Mar. 2000. [Online]. Available:
  \url{https://doi.org/10.1137/S0895479896305696}
\BIBentrySTDinterwordspacing

\bibitem{ProperComplex}
F.~D. Neeser and J.~L. Massey, ``{Proper Complex Random Processes with
  applications to {I}nformation {T}heory},'' \emph{IEEE Transactions on
  Information Theory}, vol.~39, no.~4, pp. 1293--1302, Jul 1993.

\bibitem{ComplexPDFBook}
P.~J. Schreier and L.~L. Scharf, \emph{{Statistical Signal Processing of
  Complex-Valued Data: The Theory of Improper and Noncircular Signals}}.\hskip
  1em plus 0.5em minus 0.4em\relax Cambridge University Press, 2010.

\bibitem{andersenlinear}
H.~H. Andersen, M.~Hojbjerre, D.~Sorensen, and P.~S. Eriksen, \emph{Linear and
  graphical models: for the multivariate complex normal distribution}.\hskip
  1em plus 0.5em minus 0.4em\relax Springer Science \& Business Media, 1995,
  vol. 101.

\bibitem{ComplexSignalsPJS}
T.~Adali, P.~J. Schreier, and L.~L. Scharf, ``{Complex-Valued Signal
  Processing: The Proper Way to Deal With Impropriety},'' \emph{IEEE
  Transactions on Signal Processing}, vol.~59, no.~11, pp. 5101--5125, Nov
  2011.

\bibitem{MDPIpaper}
\BIBentryALTinterwordspacing
D.~Pandey, A.~Venugopal, and H.~Leib, ``{Multi-Domain Communication Systems and
  Networks: A Tensor-Based Approach},'' \emph{MDPI Network}, vol.~1, no.~2, pp.
  50--74, 2021. [Online]. Available: \url{https://www.mdpi.com/2673-8732/1/2/5}
\BIBentrySTDinterwordspacing

\bibitem{duan2019newton}
\BIBentryALTinterwordspacing
X.-F. Duan, C.-Y. Wang, and C.-M. Li, ``Newton’s method for solving the
  tensor square root problem,'' \emph{Applied Mathematics Letters}, vol.~98,
  pp. 57--62, 2019. [Online]. Available:
  \url{https://www.sciencedirect.com/science/article/pii/S0893965919302162}
\BIBentrySTDinterwordspacing

\bibitem{MIMOHaykin}
N.~Costa and S.~Haykin, \emph{{Multiple-Input Multiple-Output Channel Models:
  Theory and Practice}}.\hskip 1em plus 0.5em minus 0.4em\relax John Wiley,
  2010.

\bibitem{kermoal2002a}
J.~P. {Kermoal}, L.~{Schumacher}, K.~I. {Pedersen}, P.~E. {Mogensen}, and
  F.~{Frederiksen}, ``{A stochastic MIMO radio channel model with experimental
  validation},'' \emph{IEEE Journal on Selected Areas in Communications},
  vol.~20, no.~6, pp. 1211--1226, 2002.

\bibitem{Kai2004modeling}
{Kai Yu}, M.~{Bengtsson}, B.~{Ottersten}, D.~{McNamara}, P.~{Karlsson}, and
  M.~{Beach}, ``{Modeling of wide-band MIMO radio channels based on NLoS indoor
  measurements},'' \emph{IEEE Transactions on Vehicular Technology}, vol.~53,
  no.~3, pp. 655--665, 2004.

\bibitem{yu2002models}
K.~Yu and B.~Ottersten, ``{Models for MIMO propagation channels: a review},''
  \emph{Wireless communications and mobile computing}, vol.~2, no.~7, pp.
  653--666, 2002.

\bibitem{CapacityCorrelatedMIMO}
L.~Hanlen and A.~Grant, ``{Capacity Analysis of Correlated MIMO Channels},''
  \emph{IEEE Transactions on Information Theory}, vol.~58, no.~11, pp.
  6773--6787, Nov 2012.

\bibitem{fu2020ber}
Y.~Fu, C.-X. Wang, X.~Fang, L.~Yan, and S.~Mclaughlin, ``{BER Performance of
  Spatial Modulation Systems Under a Non-Stationary Massive MIMO Channel
  Model},'' \emph{IEEE Access}, vol.~8, pp. 44\,547--44\,558, 2020.

\bibitem{he2018model}
H.~He, C.-K. Wen, S.~Jin, and G.~Y. Li, ``{A model-driven deep learning network
  for MIMO detection},'' in \emph{2018 IEEE Global Conference on Signal and
  Information Processing (GlobalSIP)}.\hskip 1em plus 0.5em minus 0.4em\relax
  IEEE, 2018, pp. 584--588.

\bibitem{9234486}
G.~Yang, H.~Zhang, Z.~Shi, S.~Ma, and H.~Wang, ``{Asymptotic Outage Analysis of
  Spatially Correlated Rayleigh MIMO Channels},'' \emph{IEEE Transactions on
  Broadcasting}, vol.~67, no.~1, pp. 263--278, 2021.

\bibitem{PosDetermBook}
J.~Gallier and J.~Quaintance, \emph{{Linear Algebra and Optimization with
  Applications to Machine Learning : Vol I}}.\hskip 1em plus 0.5em minus
  0.4em\relax World Scientific, 2020.

\bibitem{da2011multi}
J.~P. C.~L. da~Costa, F.~Roemer, M.~Haardt, and R.~T. de~Sousa,
  ``Multi-dimensional model order selection,'' \emph{EURASIP Journal on
  Advances in Signal Processing}, vol. 2011, pp. 1--13, 2011.

\bibitem{HoffCovariance}
P.~D. Hoff \emph{et~al.}, ``{Separable covariance arrays via the {T}ucker
  product, with applications to multivariate relational data},'' \emph{Bayesian
  Analysis}, vol.~6, no.~2, pp. 179--196, 2011.

\bibitem{DenizJOAS}
D.~Akdemir and A.~Gupta, ``{Array variate random variables with multiway
  {K}ronecker delta covariance matrix structure},'' \emph{Journal of Algebraic
  Statistics}, vol.~2, pp. 98--113, 01 2011.

\bibitem{favier2016nested}
G.~Favier, C.~A.~R. Fernandes, and A.~L. de~Almeida, ``{Nested Tucker tensor
  decomposition with application to MIMO relay systems using tensor space--time
  coding (TSTC)},'' \emph{Signal Processing}, vol. 128, pp. 318--331, 2016.

\bibitem{6288476}
C.~F. Caiafa and A.~Cichocki, ``{Block sparse representations of tensors using
  Kronecker bases},'' in \emph{2012 IEEE International Conference on Acoustics,
  Speech and Signal Processing (ICASSP)}, 2012, pp. 2709--2712.

\bibitem{zhang2013kronecker}
\BIBentryALTinterwordspacing
H.~Zhang and F.~Ding, ``{On the Kronecker products and their applications},''
  \emph{Journal of Applied Mathematics}, vol. 2013, 2013. [Online]. Available:
  \url{https://doi.org/10.1155/2013/296185}
\BIBentrySTDinterwordspacing

\bibitem{AccessPaper}
D.~Pandey and H.~Leib, ``{The Tensor Multi-Linear Channel and Its Shannon
  Capacity},'' \emph{IEEE Access}, vol.~10, pp. 34\,907--34\,944, 2022.

\bibitem{DavidThesis}
\BIBentryALTinterwordspacing
D.~Gerard, ``{Theory and Methods for Tensor Data},'' Ph.D. dissertation,
  {University of Washington}, 2015. [Online]. Available:
  \url{https://digital.lib.washington.edu/researchworks/handle/1773/34188}
\BIBentrySTDinterwordspacing

\bibitem{MLND}
\BIBentryALTinterwordspacing
M.~Ohlson, M.~R. Ahmad, and D.~von Rosen, ``The multilinear normal
  distribution: Introduction and some basic properties,'' \emph{Journal of
  Multivariate Analysis}, vol. 113, no. Supplement C, pp. 37 -- 47, 2013,
  special Issue on Multivariate Distribution Theory in Memory of Samuel Kotz.
  [Online]. Available:
  \url{http://www.sciencedirect.com/science/article/pii/S0047259X11001047}
\BIBentrySTDinterwordspacing

\bibitem{MatrixVariateBook}
A.~Gupta and D.~Nagar, \emph{{Matrix Variate Distributions, Chapter 2}}.\hskip
  1em plus 0.5em minus 0.4em\relax Chapman and Hall CRC Press, 2000.

\bibitem{AdithyaPaper}
A.~{Venugopal} and H.~{Leib}, ``A tensor based framework for multi-domain
  communication systems,'' \emph{IEEE Open Journal of the Communications
  Society}, vol.~1, pp. 606--633, 2020.

\bibitem{QiLEigen}
L.~Qi, ``{Eigenvalues of a real supersymmetric tensor},'' \emph{Journal of
  Symbolic Computation}, vol.~40, pp. 1302--1324, 2005.

\bibitem{qiwangwu2008}
L.~Qi, Y.~Wang, and X.~Wu, ``D-eigenvalues of diffusion kurtosis tensors,''
  \emph{Journal of Computational and Applied Mathematics}, vol. 221, pp.
  150--157, 2008.

\bibitem{chang2013}
K.~Chang, L.~Qi, and T.~Zhang, ``A survey on the spectral theory of nonnegative
  tensor,'' \emph{Numer. Linear Algebra Appl.}, vol.~20, pp. 891--912, 2013.

\bibitem{itskov}
M.~Itskov, ``On the theory of fourth-order tensors and their applications in
  computational mechanics,'' \emph{Comput. Methods Appl. Mech. Eng.}, vol. 189,
  no.~4, pp. 419--438, 2000.

\bibitem{arnold71}
L.~Arnold, ``On wigner’s semi-circle law for the eigenvalues of random
  matrices,'' \emph{Probability Theory and Related Fields}, vol.~19, no.~3, pp.
  191--198, 1971.

\bibitem{bbp05}
J.~Baik, G.~Ben~Arous, and S.~P\'eche\'e, ``{Phase transition of the largest
  eigenvalue for nonnull complex sample covariance matrices},'' \emph{The
  Annals of Probability}, vol.~33, no.~5, pp. 1643--1697, 2005.

\bibitem{delathauwer00}
L.~De~Lathauwer, B.~De~Moor, and J.~Vendewalle, ``On the best rank-1 and
  rank-(r1,r2,...,rn) approximation of higher-order tensors,'' \emph{SIAM
  journal on Matrix Analysis and Applications}, vol.~21, no.~4, pp. 1324--1342,
  2000.

\bibitem{VerduGrad}
D.~P. Palomar and S.~Verdu, ``{Gradient of mutual information in linear vector
  {G}aussian channels},'' \emph{IEEE Transactions on Information Theory},
  vol.~52, no.~1, pp. 141--154, Jan 2006.

\bibitem{HaykinAdaptiveFilterTheory}
S.~Haykin, \emph{Adaptive filter theory}, 4th~ed.\hskip 1em plus 0.5em minus
  0.4em\relax Prentice Hall, 2002.

\end{thebibliography}
\bibliographystyle{IEEEtran}

\end{document}